\documentclass[11pt]{article} 

\usepackage[utf8]{inputenc} 
\usepackage{authblk}

\usepackage{geometry} 
\usepackage{graphicx} 

\usepackage{booktabs} 
\usepackage{array} 
\usepackage{paralist} 
\usepackage{verbatim} 
\usepackage{subfig} 
\usepackage{amsmath}
\usepackage{amssymb}
\usepackage{amsthm}
\usepackage{listings}
\usepackage{color}
\usepackage{mathtools}
\usepackage{amsfonts}
\usepackage{caption}
\usepackage{amscd}
\usepackage{url}
\usepackage{lscape}
\usepackage{multirow}
\usepackage{tikz-cd}

\mathtoolsset{showonlyrefs}
\newtheorem{theorem}{Theorem}[section]
\newtheorem{proposition}{Proposition}[section]
\newtheorem{lemma}{Lemma}[section]
\newtheorem{corollary}{Corollary}[section]

\theoremstyle{definition}
\newtheorem{definition}{Definition}[section]
\newtheorem{example}{Example}[section]
\newtheorem{remark}{Remark}[section]

\usepackage{fancyhdr} 
\usepackage{sectsty}
\allsectionsfont{\sffamily\mdseries\upshape} 

\usepackage[nottoc,notlof,notlot]{tocbibind} 
\usepackage[titles,subfigure]{tocloft} 

\newcommand{\V}[1]{\vspace{#1}}

\newcommand{\R}{\mathbb{R}}

\newcommand{\Pa}{\partial}

\newcommand{\Ker}{\operatorname{Ker}}

\renewcommand{\Im}{\operatorname{Im}}

\newcommand{\rank}{\operatorname{rank}}
\newcommand{\corank}{\operatorname{corank}}

\newcommand{\GL}{\operatorname{GL}}

\title{Constraint Qualification for Generic Parameter Families of Constraints in Optimization}
\author[1]{Naoki Hamada}
\author[2]{Kenta Hayano}
\author[3]{Hiroshi Teramoto}
\affil[1]{KLab Inc. Roppongi Hills Mori Tower 6-10-1 Roppongi, Minato-ku, Tokyo, 106-6122, Japan}
\affil[2]{Department of Mathematics, Faculty of Science and Technology, Keio University, Yokohama,223-8522}
\affil[3]{Department of Mathematics, Faculty of Engineering Science, Kansai University, 3-3-35 Yamate-cho, Suita-shi, Osaka, Japan}
\date{} 

\numberwithin{equation}{section}

\begin{document}
\maketitle
\begin{abstract}

Constraint qualifications (CQs) are central to the local analysis of constrained optimization.  
In this paper, we completely determine the validity of the four classical CQs---LICQ, MFCQ, ACQ, and GCQ---for constraint map-germs that arise in generic four-parameter families.  
Our approach begins by proving that all four CQs are invariant under the action of the group $\mathcal{K}[G]$ and under the operation of reduction.  
As a consequence, the verification of CQ-validity for a generic constraint reduces to checking CQ-validity on the $\mathcal{K}[G]$-normal forms of fully reduced map-germs.  
Such normal forms have been classified in our recent work.  
In the present paper, we verify which CQs hold in each germ appearing in the classification tables from that work.  
This analysis provides a complete picture of the generic landscape of the four classical CQs.  
Most notably, we find that there exist numerous generic map-germs for which GCQ holds while all stronger CQs fail, showing that the gap between GCQ and the other qualifications is not an exceptional phenomenon but arises generically.  

\end{abstract}

\tableofcontents

\section{Introduction}

A constrained optimization problem asks for the minimization of objective functions subject to a collection of equality and inequality constraints. 
Formally, one seeks minimizers (or more generally, ``Pareto solutions'', for its definition, see e.g.,~\cite{Miettinen_book}) of functions $f_1(x),\ldots, f_p(x)$ under the conditions $g_1(x),\ldots, g_q(x) \leq 0$ and $h_1(x)=\cdots =h_r(x)=0$.
A fundamental tool for characterizing solutions is the Karush–Kuhn–Tucker (KKT) condition, which requires the existence of Lagrange multipliers such that first-order stationarity (the gradients of the objective and active constraints are balanced) and complementary slackness (inactive constraints have zero corresponding multipliers) hold \cite{BertsekasNP,Bazaraa2006,Nocedal2006,Luenberger2008,Miettinen_book}.
In unconstrained optimization, this reduces to the familiar first-order condition $\nabla f(x)=0$, which every local minimizer satisfies.
In constrained problems, by contrast, the existence of multipliers does not automatically follow from local minimality. 
This fact is precisely what motivates constraint qualifications (CQs): they are assumptions placed only on the constraint system, ensuring the existence of multipliers and thus the validity of the KKT condition at all local minimizers. 


In this paper we focus on four classical CQs that have been most widely studied in optimization theory and applied practice (see Definitions~\ref{def:LICQ}--\ref{def:GCQ} for the precise definitions of them):

\begin{enumerate}

\item \textbf{Linear Independence Constraint Qualification (LICQ) \cite{Hestenes1966}:} requires linear independence of the gradients of all active constraints.
It implies uniqueness of the associated Lagrange multipliers.
\item \textbf{Mangasarian–Fromovitz Constraint Qualification (MFCQ) \cite{Mangasarian1967}:} weaker than LICQ, guaranteeing nonemptiness and boundedness of the multiplier set \cite{Gauvin1977}.
It also underlies stability results of feasible sets \cite{Guddat1986}.

\item \textbf{Abadie Constraint Qualification (ACQ) \cite{Abadie1967}:} equivalent to metric regularity in the differentiable convex inequality setting \cite{doi:10.1137/S1052623495287927}.
\item \textbf{Guignard Constraint Qualification (GCQ) \cite{Guignard1969}:} the weakest among the four; it is necessary and sufficient for the KKT condition to hold at local minimizers of any objection functions \cite{01bf7f42-b4ce-3a05-b3a1-87856b488f37}.
\end{enumerate}

\noindent
These conditions form a strict hierarchy \cite{Peterson1973}:
\begin{equation}
  \mathrm{LICQ}\ \Longrightarrow\ \mathrm{MFCQ}\ \Longrightarrow\ \mathrm{ACQ}\ \Longrightarrow\ \mathrm{GCQ}. \label{eq:strict_hierarchy}
\end{equation}
\noindent
In practice, LICQ and MFCQ are frequently invoked because they can be checked by simple rank conditions or directional arguments. 
By contrast, ACQ and GCQ, despite their generality, require analyzing tangent cones and are therefore much harder to verify directly.
It is also known that each implication above is \emph{strict}: none of the reverse implications hold in general.
Concrete counterexamples witnessing the failure of the converses (MFCQ $\nRightarrow$ LICQ, ACQ $\nRightarrow$ MFCQ, GCQ $\nRightarrow$ ACQ) are well documented in the literature (see, e.g., \cite{Peterson1973}).
However, the available constructions tend to be rather bespoke and leave open whether such separations occur \emph{generically}—that is, in typical parameterized constraint families encountered in practice. 
In Appendix~\ref{sec:appendix}, we investigate genericity of some of the counterexamples in the literature and show that some of them are far from generic.
One of our aims is to address this gap by identifying generic constraints in which a stronger CQ fails while a weaker one still holds.
As shown by our later results, such gaps indeed arise generically; see Theorem~\ref{thm:generic_CQ_classification} and Tables~\ref{tab:eq-only}, \ref{tab:ineq-only}, and \ref{tab:mixed} for details.

The aim of this work is to determine, for each point of a feasible set arising from a generic parameter family of constraints, whether or not a given constraint qualification (CQ) holds.
CQ-validity is a local property: whether a CQ holds at a point depends only on the germ of the constraint at that point.
In \cite{HamadaHayanoTeramoto2025_1}, we classify the \emph{full reductions} of constraint map-germs appearing in generic four parameter families up to $\mathcal{K}[G]$-equivalence (see Section~\ref{sec:prelim} for the definitions of (full) reduction and $\mathcal{K}[G]$-equivalence).
\emph{As we establish below, the four CQs under study are invariant under $\mathcal{K}[G]$-equivalence and reduction; hence the classification of \cite{HamadaHayanoTeramoto2025_1} is the natural vehicle for deciding CQ-validity.}
We first prove that reduction commutes with $\mathcal{K}[G]$-equivalence (Lemmas~\ref{lem:K[G]-equiv if reduction is so} and \ref{lem:K[G]-equiv if original is so}) and that LICQ, MFCQ, ACQ, and GCQ are invariant under $\mathcal{K}[G]$-equivalence and reduction (Theorems~\ref{thm:invariance four CQs under K[G]-eq} and \ref{thm:invariance four CQs under reduction}).
Consequently, it suffices to check, for each reduced normal-form listed in \cite{HamadaHayanoTeramoto2025_1}, which of the four CQs hold.
We carry out this verification and compile complete validity tables across all generic classes (Theorem~\ref{thm:generic_CQ_classification}). 
In particular, we find that GCQ is satisfied
in a broader range of the generic classes than any of the other stronger CQs (LICQ,MFCQ, ACQ), underscoring its role as an essential fallback condition that guarantees the existence of Lagrange multipliers when those stronger qualifications fail to hold. These results collectively provide a generic picture of how and when each CQ is satisfied.

Beyond the immediate theoretical contributions, we anticipate our framework will support applications in the following ways:
\begin{itemize}
    \item \textbf{Generic examples for CQ research:} By exhibiting constraint systems that satisfy, say, MFCQ but violate LICQ, or that satisfy ACQ but not MFCQ, etc., one can confirm the well-known strictness relations or investigate potential new CQs lying in the gaps. The genericity of these families underscores that these degeneracies are neither rare nor pathological in practical contexts.
    
\item \textbf{Benchmarking and algorithmic impact:} Optimization algorithms often assume or exploit certain CQs. Having a collection of representative examples---classified by which CQs hold---can facilitate more rigorous testing of algorithmic reliability and performance. 
    
\item \textbf{Path to efficient recognition of ACQ and GCQ:} In principle, verifying whether ACQ or GCQ holds can be challenging because they require analysis of tangent cones. In upcoming work, we plan to design recognition algorithms that use the singularity classification of \cite{HamadaHayanoTeramoto2025_1} to decide systematically which class a given constraint system belongs to. Once that class is identified, our results immediately specify whether ACQ or GCQ is satisfied. This can significantly simplify the process of CQ checking in software for large-scale or complex problems.
\end{itemize}

We note certain limitations in this study. First, our results restrict to problems involving $C^\infty$ smooth equality and inequality constraints, leaving out non-smooth or more abstract constraint structures (such as cone constraints). Second, we focus exclusively on CQs that are invariant under the action of $\mathcal{K} \left[ G \right]$ studied in \cite{HamadaHayanoTeramoto2025_1}, thereby excluding conditions like Pseudonormality \cite{Bertsekas2002}, Slater’s CQ (specific to convex settings) \cite{Slater2014} or linearity conditions \cite{BertsekasNP,Bazaraa2006,Nocedal2006,Luenberger2008}. 
Third, we target CQs that appear in a \emph{generic} sense as determined by the equivalences in \cite{HamadaHayanoTeramoto2025_1}, leaving aside specialized constraints like the constant rank CQ (CRCQ). However, we remark that the condition CRCQ often implies MFCQ for suitably reformulated problems \cite{Lu2011}, so our omission of CRCQ does not substantially affect the generic analysis.

The rest of this paper is organized as follows. In Section~\ref{sec:prelim}, we briefly recall the necessary background on standard CQs and the main classification results from \cite{HamadaHayanoTeramoto2025_1}. Then, in Section~\ref{sec:CQ-definition_and_properties}, we introduce the mathematical definition for each CQ, and show that these are invariant under $\mathcal{K}[G]$-equivalence and reduction. 
In Section~\ref{sec:verification_of_CQ_for_each_class}, we verify, for each class, which of the standard CQs (LICQ, MFCQ, ACQ, GCQ) are satisfied and discuss the resulting hierarchy in detail.

\section{Preliminaries: constraints, $\mathcal{K}[G]$-equivalence, reduction, and classification of constraint map-germs}
\label{sec:prelim}

In this section, we recall the basic notions of constraint map-germs, feasible set-germs, and the group $\mathcal{K}[G]$, following \cite{HamadaHayanoTeramoto2025_1}. We then discuss the notion of reduction of a constraint map-germ and state Theorem 5.1 from \cite{HamadaHayanoTeramoto2025_1}, which classifies constraint map-germs appearing in generic parameter families of constraints with up to $4$-parameters.

The precise definition of constraint map-germ and feasible set is as follows. 
\begin{definition}[Constraint map-germ and feasible set]\label{def:constraint}
Let $g = (g_1,\dots,g_q)\colon (\mathbb{R}^n,0)\to \mathbb{R}^q$ and 
$h = (h_1,\dots,h_r)\colon (\mathbb{R}^n,0)\to(\mathbb{R}^r,0)$ 
be smooth map-germs at $0 \in \mathbb{R}^n$. We call $(g,h)$ a \emph{constraint map-germ} at $0$, where $g_i$ represent \emph{inequality} constraint functions and $h_j$ represent \emph{equality} constraint functions. 
The corresponding \emph{feasible set-germ} $M(g,h)$ is the germ at $0$ of the set
\[
M(g,h) \;=\; \{\, x \in \left( \mathbb{R}^n, 0 \right) \;\mid\; g_i(x)\,\le\,0 \ \ \text{for all}\ 1 \le i \le q, \ \ h_j(x)\,=\,0 \ \ \text{for all}\ 1 \le j \le r \}.
\]
If $g_i(0) > 0$ for some $i$, then $M(g,h)$ is empty. On the other hand, if $g_i(0) < 0$ for some $i$, then $g_i$ does not affect the feasible set-germ near $0$ and may be removed from the constraint list without changing $M(g,h)$.
\end{definition}

We next introduce 
\begin{math}
\mathcal{K} \left[ G \right]
\end{math}-equivalence. Following \cite{HamadaHayanoTeramoto2025_1}, let $\mathcal{K}$ denote Mather’s group \cite{Mather1968} and let $\mathcal{K}[G]$ be the subgroup of $\mathcal{K}$ consisting of those coordinate changes that preserve the inequality/equality structure of constraints. Concretely, a pair $(\Phi,\Psi)$ belongs to $\mathcal{K} \left[ G \right]$ if $\Phi\colon (\mathbb{R}^n,0)\to (\mathbb{R}^n,0)$ is a diffeomorphism-germ of the source, $\Psi\colon (\mathbb{R}^n,0)\to G$ is a smooth map-germ into a target group $G$ of block matrices of the form
    \[
    G \;=\; \left\{
    \begin{pmatrix}
      C & B \\
      O_{r,q} & A
    \end{pmatrix}
    \;\middle|\;
    C \in G_{gp}, \ B \in M_{q,r}(\mathbb{R}), \ A \in GL(r,\mathbb{R})
    \right\},
    \]
    where 
\begin{math}
G_{gp} = G_d \rtimes P_q
\end{math}, the semidirect product of 
\begin{math}
G_d
\end{math}
and the group of 
\begin{math}
q \times q
\end{math}
permutation matrices 
\begin{math}
P_q
\end{math}, 
\begin{math}
O_{r,q}
\end{math}
is the $r \times q$ zero matrix, 
\begin{math}
M_{q,r} \left( \mathbb{R} \right)
\end{math}
is the set of 
\begin{math}
q \times r
\end{math}
matrices, and 
\begin{math}
GL \left( r, \mathbb{R} \right)
\end{math}
is the set of 
\begin{math}
r \times r
\end{math}
regular matrices. For a constraint map-germ $(g,h)$, the action of $(\Phi,\Psi)\in \mathcal{K}[G]$ is given by
\[
(\Phi,\Psi)\cdot (g(x),h(x))
\;=\;
\Bigl(C(x)\,g(\Phi^{-1}(x)) + B(x)\,h(\Phi^{-1}(x)) \,,\, A(x)\,h(\Phi^{-1}(x))\Bigr),
\]
where $C(x), B(x), A(x)$ denote the respective block components of $\Psi(x)$. 
Two constraint map-germs $(g,h)$ and $(g',h')$ are \emph{$\mathcal{K}[G]$-equivalent} if they lie in the same $\mathcal{K}[G]$-orbit, i.e.,~there exists $(\Phi,\Psi)\in \mathcal{K}[G]$ such that 
\[
(g,h) \;=\; (\Phi,\Psi)\cdot (g',h').
\]
By construction, $\mathcal{K}[G]$-equivalence preserves the feasible set-germs: if $(g,h)$ is $\mathcal{K}[G]$-equivalent to $(g',h')$, then $M(g,h)$ is diffeomorphic to $M(g',h')$ via the map $\Phi$.


The concept of \emph{reduction} deals with simplifying a constraint map-germ by removing inactive (strictly satisfied) inequalities and by restricting it to the submanifold defined by certain equality constraints.
%
Again, we summarize the definition from \cite{HamadaHayanoTeramoto2025_1}:

\begin{definition}[Reduction]\label{def:reduction}
Let $(g,h)\colon (\mathbb{R}^n,0)\to(\mathbb{R}^q \times \mathbb{R}^r,0)$ be a constraint map-germ with nonempty feasible set $M(g,h)$,
\[
(k) \;=\; (k_1,\dots,k_{q-s}) \;\subset\; \{1,\ldots,q\}
\]
be a subset of indices corresponding to $q-s$ inequality components $g_{k_i}(0)$ that are all $<0$ (i.e.,~inactive), and
\[
(i) \;=\; (i_1,\dots,i_{r-\ell}) \;\subset\; \{1,\ldots,r\}
\]
be a subset for which $d h_{i_1}(0),\ldots, d h_{i_r-\ell}(0)$ are linearly independent. 
Let $\iota_{(i)}\colon (\mathbb{R}^{\,n-r+\ell},0)\to (\mathbb{R}^n,0)$ be an immersion-germ to the submanifold $(h_{i_1},\ldots, h_{i_{r-\ell}})^{-1}(0)$ near $0$. Then the \emph{reduction of $(g,h)$} relative to $(k)$ and $(i)$ is given by the map-germ
\[
(g,h)_{\iota_{(i)},(k)}:= (g_{\iota_{(i)},(k)},h_{\iota_{(i)}}) := \left(g_1\circ\iota_{(i)},\overset{\hat{k}}{\ldots},g_q\circ \iota_{(i)},h_1\circ\iota_{(i)},\overset{\hat{i}}{\ldots},h_r\circ\iota_{(i)}\right), 
\]
The reduction $(g_{\iota_{(i)},(k)},h_{\iota_{(i)}})$ of $(g,h)$ is called a \emph{full reduction} if $g_{\iota_{(i)},(k)}(0)=0$ and the rank of $dh_{\iota_{(i)},0}$ is zero.

Specifically, if the resulting constraint satisfies that its all the inequality constraints are active and the rank of $dh_{\iota_{(i)},0}$ is zero, we call the reduction full reduction. 

\end{definition}

\noindent
Note that the feasible set-germ of $(g,h)$ is preserved by this operation (up to diffeomorphism).

\subsection{Relation between reduction and $\mathcal{K} \left[ G \right]$-equivalence} \label{sec:reduction_and_K[G]}
In this section, we clarify how the reduction of a constraint map‑germ interacts with the group‑action by $\mathcal{K} \left[ G \right]$. First, we show that if two constraint map‑germs have the same numbers of active inequality and equality constraints, and their respective reductions are $\mathcal{K} \left[ G \right]$-equivalent, then the original (unreduced) germs are also $\mathcal{K} \left[ G \right]$-equivalent (Lemma~\ref{lem:K[G]-equiv if reduction is so}). 
Second, we show that conversely when both germs are finitely $\mathcal{K} \left[ G \right]$-determined, $\mathcal{K} \left[ G \right]$-equivalence of the unreduced germs forces $\mathcal{K} \left[ G \right]$-equivalence of any pair of reductions that retain the same numbers of constraints (Lemma~\ref{lem:K[G]-equiv if original is so}). These results imply that one may freely work with fully reduced representatives when classifying constraint map‑germs under $\mathcal{K} \left[ G \right]$ without losing information. Any $\mathcal{K} \left[ G \right]$-orbit is uniquely determined by, and can be recovered from, the orbit of its reduction.

\begin{lemma} \label{lem:K[G]-equiv if reduction is so}
Let 
\begin{math}
\left( g, h \right)
\end{math}
and 
\begin{math}
\left( g', h' \right)
\end{math}
be two constraint map-germs with the equal number of inequality and equality constraints. 
Then, 
\begin{math}
\left( g, h \right)
\end{math}
and 
\begin{math}
\left( g', h' \right)
\end{math}
are
\begin{math}
\mathcal{K} \left[ G \right]
\end{math}-equivalent if their reductions 
\begin{math}
\left( g, h \right)_{\iota_{\left( i \right)}, \left( k \right)}
\end{math}
and 
\begin{math}
\left( g', h' \right)_{\iota_{\left( i' \right)}, \left( k' \right)}
\end{math}
are 
\begin{math}
\mathcal{K} \left[ G \right]
\end{math}-equivalent.
\end{lemma}
\begin{proof}
Suppose 
\begin{math}
\left( g, h \right)_{\iota_{\left( i \right)}, \left( k \right)}
\end{math}
and 
\begin{math}
\left( g', h' \right)_{\iota_{\left( i' \right)}, \left( k' \right)}
\end{math}
are 
\begin{math}
\mathcal{K} \left[ G \right]
\end{math}-equivalent.  In what follows, we write 
\begin{math}
h = \left( h_R, h_S \right)
\end{math}, 
\begin{math}
h' = \left( h'_R, h'_S \right)
\end{math}, 
\begin{math}
g = \left( g_R, g_S \right)
\end{math}, 
\begin{math}
g' = \left( g'_R, g'_S \right)
\end{math}, where 
\begin{math}
h_R = \left( h_{i_1}, \ldots, h_{i_{r-\ell}} \right)
\end{math}, 
\begin{math}
h'_R = \left( h'_{i'_1}, \ldots, h'_{i'_{r-\ell}} \right)
\end{math}, \begin{math}
h_S, h'_S
\end{math}
are remaining part of 
\begin{math}
h
\end{math}
and 
\begin{math}
h'
\end{math}, respectively, and we take $g_R,g_S,g_R',g_S'$ in the same way. 
%
By the definition of constraint map-germs, \begin{math}
h \left( 0 \right) = h' \left( 0 \right) = 0
\end{math}
holds. 
Without loss of generality, by appropriate permutation of components, we can assume 
\begin{math}
\left( g, h \right)
\end{math}
and 
\begin{math}
\left( g', h' \right)
\end{math}
can be written as above. By choosing coordinates of 
\begin{math}
\mathbb{R}^n
\end{math}
appropriately, we can suppose $h_{i_j}(x)=x_j$ for $j=1,\ldots, r-\ell$ and \begin{math}
\iota_{\left( i \right)} \left( x_{r-\ell+1}, \ldots, x_n \right) = \left( 0, \ldots, 0, x_{r-\ell+1}, \ldots, x_n \right)
\end{math}.
%
By the assumption, there is a diffeomorphism-germ 
\begin{math}
\tilde{\phi} \colon \left( \mathbb{R}^{n-r+\ell}, 0 \right) \rightarrow \left( \mathbb{R}^{n-r+\ell}, 0 \right)
\end{math}
and map-germs 
\begin{math}
\tilde{C}_{22} \colon \left( \mathbb{R}^{n-r+\ell}, 0 \right) \rightarrow G_{gp}
\end{math}, 
\begin{math}
\tilde{B}_{22} \colon \left( \mathbb{R}^{n-r+\ell}, 0 \right) \rightarrow M_{q-s,\ell} \left( \mathbb{R} \right)
\end{math}, and 
\begin{math}
\tilde{A}_{22} \colon \left( \mathbb{R}^{n-r+\ell}, 0 \right) \rightarrow GL \left( \mathbb{R}, \ell \right)
\end{math}
such that 
\begin{equation}
\begin{pmatrix}
\tilde{C}_{22} & \tilde{B}_{22} \\
O & \tilde{A}_{22}
\end{pmatrix}
\begin{pmatrix}
g_S \circ \iota_{\left( i \right)} \left( x_{r-\ell+1}, \ldots, x_n \right) \\
h_S \circ \iota_{\left( i \right)} \left( x_{r-\ell+1}, \ldots, x_n \right)
\end{pmatrix}
=
\begin{pmatrix}
g'_S \circ \iota_{\left( i' \right)} \circ \phi \left( x_{r-\ell+1}, \ldots, x_n \right) \\
h'_S \circ \iota_{\left( i' \right)} \circ \phi \left( x_{r-\ell+1}, \ldots, x_n \right)
\end{pmatrix} \label{eq:K[G]-equivalence_reduction}
\end{equation}
holds. We define 
\begin{math}
\Psi \colon \left( \mathbb{R}^n, 0 \right) \rightarrow \left( \mathbb{R}^n, 0 \right)
\end{math}
such as 
\begin{equation}
\Psi \left( x \right) = x_1 d \left( h'_{i'_1} \right)_0 + \cdots + x_{r-\ell} d \left( h'_{i'_{r-\ell}} \right)_0 + \iota_{\left( i' \right)} \circ \phi \left( x_{r-\ell+1}, \ldots, x_n \right).
\end{equation}
Then, 
\begin{math}
\Psi
\end{math}
is a diffeomorphism-germ. 
Then, 
\begin{align}
g'_S \circ \Psi \left( x \right) &= g'_S \circ \Psi \left( 0, \ldots, 0, x_{r-\ell+1}, \ldots, x_n \right) + \sum_{j=1}^{r-\ell} x_j \tilde{g}_j \left( x \right) \\
&= g'_S \circ \iota_{\left( i' \right)} \circ \phi \left( x_{r-\ell+1}, \ldots, x_n \right) + \sum_{j=1}^{r-\ell} x_j \tilde{g}_j \left( x \right)
\end{align}
holds by Hadamard's theorem by choosing 
\begin{math}
\tilde{g}_j
\end{math}
appropriately. In the similar manner, 
\begin{equation}
h'_R \circ \Psi \left( x \right) = h'_R \circ \Psi \left( 0, \ldots, 0, x_{r-\ell+1}, \ldots, x_n \right) + \sum_{j=1}^{r-\ell} x_j \left( \tilde{h}_R' \right)_j \left( x \right) = \sum_{j=1}^{r-\ell} x_j \left( \tilde{h}_R' \right)_j \left( x \right),
\end{equation}
by noting that 
\begin{math}
h'_R \circ \iota_{\left( i' \right)} \circ \phi = 0
\end{math}
and again choosing 
\begin{math}
\left( \tilde{h}_R' \right)_j
\end{math}
appropriately. In addition, 
\begin{align}
h'_S \circ \Psi \left( x \right) &= h'_S \circ \Psi \left( 0, \ldots, 0, x_{r-\ell+1}, \ldots, x_n \right) + \sum_{j=1}^{r-\ell} x_j \left( \tilde{h}_S' \right)_j \left( x \right), \\
&= h'_S \circ \iota_{\left( i' \right)} \circ \phi \left( x_{r-\ell+1}, \ldots, x_n \right) + \sum_{j=1}^{r-\ell} x_j \left( \tilde{h}_S' \right)_j \left( x \right)
\end{align}
holds. 
By noting that $h_R(x) = (x_1,\ldots, x_{r-\ell})$, there exist
$\tilde{A}_{11}, \tilde{A}_{21}$ and $\tilde{B}_{21}$
such that
\begin{align*}
\begin{pmatrix}
g'_R \circ \Psi \left( x \right) \\
g'_S \circ \Psi \left( x \right) \\
h'_R \circ \Psi \left( x \right) \\
h'_S \circ \Psi \left( x \right)
\end{pmatrix}
=&
\begin{pmatrix}
g'_R \circ \Psi \left( x \right) \\
g'_S \circ \iota_{\left( i' \right)} \circ \phi \left( x_{r-\ell+1}, \ldots, x_n \right) \\
0 \\
h'_S \circ \iota_{\left( i' \right)} \circ \phi \left( x_{r-\ell+1}, \ldots, x_n \right)
\end{pmatrix}
+
\begin{pmatrix}
O & O & O & O \\
O & O & \tilde{B}_{21} & O \\
O & O & \tilde{A}_{11} & O \\
O & O & \tilde{A}_{21} & O
\end{pmatrix}
\begin{pmatrix}
0 \\
0 \\
h_R \left( x \right) \\
0
\end{pmatrix}\\
=&
\begin{pmatrix}
I & O & O & O \\
O & \tilde{C}_{22} & \tilde{B}_{21} & \tilde{B}_{22} \\
O & O & \tilde{A}_{11} & O \\
O & O & \tilde{A}_{21} & \tilde{A}_{22}
\end{pmatrix}
\begin{pmatrix}
g'_R \circ \Psi \left( x \right) \\
g_S \circ \iota_{\left( i \right)} \left( x_{r-\ell+1}, \ldots, x_n \right) \\
h_R \left( x \right) \\
h_S \circ \iota_{\left( i \right)} \left( x_{r-\ell+1}, \ldots, x_n \right)
\end{pmatrix}
\end{align*}
holds.
We can also choose
\begin{math}
\left( \tilde{g}_S \right)_j, \left( \tilde{h}_S \right)_j
\end{math}
so that the equalities
\begin{align}
g_S \circ \iota_{\left( i \right)} \left( x_{r-\ell+1}, \ldots, x_n \right) &= g_S \left( 0, \ldots, 0, x_{r-\ell+1}, \ldots, x_n \right) \\
&= g_S \left( x_1, \ldots, x_n \right) - \sum_{j=1}^{r-\ell} x_j \left( \tilde{g}_S \right)_j \left( x \right)
\end{align}
and 
\begin{align}
h_S \circ \iota_{\left( i \right)} \left( x_{r-\ell+1}, \ldots, x_n \right) &= h_S \left( 0, \ldots, 0, x_{r-\ell+1}, \ldots, x_n \right) \\
&= h_S \left( x_1, \ldots, x_n \right) - \sum_{j=1}^{r-\ell} x_j \left( \tilde{h}_S \right)_j \left( x \right)
\end{align}
hold.
Since all the components of $g_R(0)$ is less than $0$, we can define the diagonal matrix
\begin{math}
\tilde{C}_{11}
\end{math}
such that 
\begin{math}
\left( \tilde{C}_{11} \right)_{j,j} = g'_{k'_j} \circ \Psi \left( x \right) / g_{k_j} \left( x \right)
\end{math}
holds for each 
\begin{math}
j \in \left\{ 1, \ldots, s \right\}
\end{math}. 
We then obtain 
\begin{equation}
\begin{pmatrix}
g'_R \circ \Psi \left( x \right) \\
g'_S \circ \Psi \left( x \right) \\
h'_R \circ \Psi \left( x \right) \\
h'_S \circ \Psi \left( x \right)
\end{pmatrix}
=
\begin{pmatrix}
\tilde{C}_{11} & O & O & O \\
O & \tilde{C}_{22} & \tilde{B}_{21}' & \tilde{B}_{22} \\
O & O & \tilde{A}_{11} & O \\
O & O & \tilde{A}_{21}' & \tilde{A}_{22}
\end{pmatrix}
\begin{pmatrix}
g_R \left( x \right) \\
g_S \left( x \right) \\
h_R \left( x \right) \\
h_S \left( x \right)
\end{pmatrix},
\end{equation}
where 
\begin{math}
\tilde{B}_{21}'
\end{math}
and 
\begin{math}
\tilde{A}_{21}'
\end{math}
are modified accordingly. 
The image of
\begin{math}
\begin{pmatrix}
\tilde{C}_{11} & O \\
O & \tilde{C}_{22}
\end{pmatrix}
\end{math}
is in $G_{gp}$ by the construction.
Moreover, the matrix 
\begin{math}
\begin{pmatrix}
\tilde{A}_{11} & O \\
\tilde{A}_{21}' & \tilde{A}_{22}
\end{pmatrix}
\end{math}
is in 
\begin{math}
GL \left( \mathbb{R}, r \right)
\end{math} since the image of $\tilde{A}_{22}$ is in $\GL(\R,\ell)$, $d(h_R'\circ \Psi)_0 = \tilde{A}_{11}(0) \cdot d(h_R)_0$ and its rank is $r-\ell$.
This proves that the two constraint map-germs 
\begin{math}
\left( g, h \right)
\end{math}
and 
\begin{math}
\left( g', h' \right)
\end{math}
are 
\begin{math}
\mathcal{K} \left[ G \right]
\end{math}-equivalent. 
\end{proof}

The converse of this theorem holds if one of the constraint map-germs 
\begin{math}
\left( g, h \right)
\end{math}
and 
\begin{math}
\left( g', h' \right)
\end{math}
is finitely 
\begin{math}
\mathcal{K} \left[ G \right]
\end{math}-determined.
(In that case, of course both of the constraint map-germs are finitely \begin{math}
\mathcal{K} \left[ G \right]
\end{math}-determined.) To show this, let us introduce some necessary terminologies. Let 
\begin{math}
\mathcal{E}_n = \left\{ f \middle| f \colon \left( \mathbb{R}^n, 0 \right) \rightarrow \mathbb{R} \right\}
\end{math}
be the ring of function-germs. For a constraint map-germ 
\begin{math}
\left( g, h \right) \colon \left( \mathbb{R}^n, 0 \right) \rightarrow \mathbb{R}^{q+r}
\end{math}
whose feasible set contains the origin (this condition is equivalent to 
\begin{math}
g \left( 0 \right) \le 0
\end{math}), we define an 
\begin{math}
\mathbb{R}
\end{math}-algebra
\begin{equation}
Q \left( h \right) = \mathcal{E}_n / \langle h_1, \ldots, h_r \rangle_{\mathcal{E}_n}
\end{equation}
by following Mather \cite{Mather1969} where 
\begin{math}
\langle h_1, \ldots, h_r \rangle_{\mathcal{E}_n}
\end{math}
is the ideal in 
\begin{math}
\mathcal{E}_n
\end{math}
generated by 
\begin{math}
h_1, \ldots, h_r
\end{math}. In addition, we define 
\begin{math}
Q_k \left( h \right) = Q \left( h \right) / \left( \langle x_1, \ldots, x_n \rangle^{k+1}_{\mathcal{E}_n} \cdot Q \left( h \right) \right)
\end{math}
for 
\begin{math}
k \in \mathbb{N}
\end{math}. Note that 
\begin{equation}
\langle x_1, \ldots, x_n \rangle_{\mathcal{E}_n}^{k+1} \cdot Q \left( h \right) \cong \left( \langle x_1, \ldots, x_n \rangle^{k+1}_{\mathcal{E}_n} + \langle h_1, \ldots, h_r \rangle_{\mathcal{E}_n} \right) / \langle h_1, \ldots, h_r \rangle_{\mathcal{E}_n}
\end{equation}
holds. Then, the third isomorphism theorem implies that 
\begin{align}
Q_k \left( h \right) &\cong \mathcal{E}_n / \left( \langle x_1, \ldots, x_n \rangle^{k+1}_{\mathcal{E}_n} + \langle h_1, \ldots, h_r \rangle_{\mathcal{E}_n} \right). 
\end{align}
\begin{lemma} \label{lem:K[G] and local R-algebra}
Let 
\begin{math}
\left( g, h \right), \left( g', h' \right) \colon \left( \mathbb{R}^n, 0 \right) \rightarrow \mathbb{R}^{q+r}
\end{math}
be two constraint map-germs whose feasible set-germs are non-empty and let 
\begin{math}
k \in \mathbb{N}
\end{math}. Then, the corresponding $k$-jets
\begin{math}
j^k \left( g, h \right)
\end{math}
and 
\begin{math}
j^k \left( g', h' \right)
\end{math}
are in the same orbit under the action of 
\begin{math}
\mathcal{K} \left[ G \right]^k
\end{math}
if and only if there exist an 
\begin{math}
\mathbb{R}
\end{math}-algebra isomorphism $\phi: Q_k \left( h \right)
\to Q_k \left( h' \right)$, a permutation $\sigma:\{1,\ldots,q\}\to \{1,\ldots, q\}$ and $u_j\in \mathcal{E}_n$ with $u_j(0)>0$ such that $\phi \circ \pi_h^k(g_j)$ is equal to $\pi_{h'}^k(u_jg'_{\sigma(j)})$, where $\pi_h^k: \mathcal{E}_n \rightarrow Q_k \left( h \right)$ is the projection. 

\end{lemma}

\begin{proof}
%
\noindent
\textbf{Proof of ``if'' part:}
Since 
\begin{math}
Q_k \left( h \right) \cong Q_k \left( h' \right)
\end{math}
holds, 
\begin{math}
h
\end{math}
and 
\begin{math}
h'
\end{math}
have the same rank at the origin. We denote it by
\begin{math}
r - \ell
\end{math}
for 
\begin{math}
\ell \in \left\{ 0, \ldots, r \right\}
\end{math}. Without loss of generality, we can suppose 
\begin{equation}
h \left( x_1, \ldots, x_n \right) = \left( x_1, \ldots, x_{r-\ell}, h_{r-\ell+1} \left( x_{r-\ell+1}, \ldots, x_n \right), \ldots, h_r \left( x_{r-\ell+1}, \ldots, x_n \right) \right),
\end{equation}
and 
\begin{equation}
h' \left( x_1, \ldots, x_n \right) = \left( x_1, \ldots, x_{r-\ell}, h'_{r-\ell+1} \left( x_{r-\ell+1}, \ldots, x_n \right), \ldots, h'_r \left( x_{r-\ell+1}, \ldots, x_n \right) \right).
\end{equation}

Let 
\begin{math}
\phi \colon Q_k \left( h \right) \to Q_k \left( h' \right)
\end{math}, $\sigma:\{1,\ldots,q\}\to \{1,\ldots, q\}$ and $u_j\in \mathcal{E}_n$ satisfy the assumption.
For 
\begin{math}
r-\ell+1 \le j \le n
\end{math}, 
we take a polynomial $p_j(X)$ with variables $X_{r-\ell+1},\ldots, X_n$ so that 
\begin{math}
\phi \left( \pi_h^k(x_j) \right)
\end{math}
is equal to 
\begin{math}
p_j \left( \pi_{h'}^k(x_{r-\ell+1}),\ldots,  \pi_{h'}^k(x_{n})\right)
\end{math}.
%
Define 
\begin{math}
\psi \colon \left( \mathbb{R}^n, 0 \right) \rightarrow \left( \mathbb{R}^n, 0 \right)
\end{math}
by 
\begin{equation}
x_j \circ \psi = 
\begin{cases}
x_j & \left( 1 \le j \le r-\ell \right) \\
p_j \left( x_{r-\ell+1},\ldots, x_n \right) & \left( r-\ell+1 \le j \le n \right)
\end{cases}.
\end{equation}
Then, 
\begin{math}
\psi
\end{math}
is invertible, since
\begin{math}
\phi
\end{math}
is an isomorphism and thus the matrix 
\begin{math}
\begin{pmatrix}
\cfrac{\partial p_j}{\partial x_l}
\end{pmatrix}_{r-\ell+1 \le j,l \le n}
\end{math}
is invertible. Furthermore the following diagram commutes: 
\[
\begin{tikzcd}
\mathcal{E}_n \arrow[r, "\psi^*"] \arrow[d, "\pi_h^k"] & \mathcal{E}_n \arrow[d, "\pi_{h'}^k"] \\
Q_k \left( h \right) \arrow[r, "\phi"'] & Q_k \left( h' \right)
\end{tikzcd}
\]
Hence, replacing 
\begin{math}
h'
\end{math}
by 
\begin{math}
h' \circ \psi^{-1}
\end{math}
and 
\begin{math}
g'
\end{math}
by 
\begin{math}
g' \circ \psi^{-1}
\end{math}, we may suppose that
\begin{equation}
\langle h_1, \ldots, h_r \rangle_{\mathcal{E}_n} + \langle x_1, \ldots, x_n \rangle^{k+1}_{\mathcal{E}_n} = \langle h'_1, \ldots, h'_r \rangle_{\mathcal{E}_n} + \langle x_1, \ldots, x_n \rangle^{k+1}_{\mathcal{E}_n}, \label{eq:hhd_modk+1}
\end{equation}
and 
\begin{equation}
g_j = u_j g'_{\sigma \left( j \right)} \mod \langle h'_1, \ldots, h'_r \rangle_{\mathcal{E}_n} + \langle x_1, \ldots, x_n \rangle^{k+1}_{\mathcal{E}_n}
\end{equation}
for 
\begin{math}
j \in \left\{ 1, \ldots, q \right\}
\end{math}. Replacing 
\begin{math}
h'
\end{math}
by another map-germ having the same 
\begin{math}
k
\end{math}-jet, we may suppose that 
\begin{equation}
\langle h_1, \ldots, h_r \rangle_{\mathcal{E}_n} = \langle h'_1, \ldots, h'_r \rangle_{\mathcal{E}_n} \label{eq:hhd}
\end{equation}
by the following argument.  By Eq.~\eqref{eq:hhd_modk+1}, there exist map-germs 
\begin{math}
\tilde{A}, \tilde{B} \colon \left( \mathbb{R}^n, 0 \right) \rightarrow M_{r,r} \left( \mathbb{R} \right)
\end{math}
and 
\begin{math}
\tilde{h}, \tilde{h}' \in \langle x_1, \ldots, x_n \rangle_{\mathcal{E}_n}^{k+1} \mathcal{E}_n^r
\end{math}
such that 
\begin{align}
h &= \tilde{A} h' + \tilde{h} \\
h' &= \tilde{B} h + \tilde{h}'
\end{align}
hold. By using Lemma in \cite[\S.2]{Mather1968}, there exists a matrix 
\begin{math}
\tilde{C} \in M_{r,r} \left( \mathbb{R} \right)
\end{math}
such that 
\begin{equation}
\tilde{C} \left( I_r - \tilde{A} \left( 0 \right) \tilde{B} \left( 0 \right) \right) + \tilde{B} \left( 0 \right)
\end{equation}
is regular, where 
\begin{math}
I_r
\end{math}
is 
\begin{math}
r \times r
\end{math}
unit matrix. Put 
\begin{math}
D \left( x \right) = \tilde{C} \left( I_r - \tilde{A} \left( x \right) \tilde{B} \left( x \right) \right) + \tilde{B} \left( x \right)
\end{math}. Then, 
\begin{align}
D h &= \tilde{C} \left( I_r - \tilde{A} \tilde{B} \right) h + \tilde{B} h \\
&= \tilde{C} \left( h - \tilde{A} \tilde{B} h \right) + \tilde{B} h \\
&= \tilde{C} \left( h - \tilde{A} \left( h' - \tilde{h}' \right) \right) + h' - \tilde{h}' \\
&= \tilde{C} \left( \tilde{h} + \tilde{A} \tilde{h}' \right) + h' - \tilde{h}' \\
&= h' + \left( \tilde{C} \left( \tilde{h} + \tilde{A} \tilde{h}' \right) - \tilde{h}' \right).
\end{align}
Since 
$\tilde{C} \left( \tilde{h} + \tilde{A} \tilde{h}' \right) - \tilde{h}' \in \langle x_1, \ldots, x_n \rangle_{\mathcal{E}_n}^{k+1} \mathcal{E}_n^r$
holds, 
\begin{math}
h'
\end{math}
has the same $k$-jet as  
\begin{math}
h' + \left( \tilde{C} \left( \tilde{h} + \tilde{A} \tilde{h}' \right) - \tilde{h}' \right)
\end{math}. By replacing
\begin{math}
h'
\end{math}
by the latter germ, we obtain Eq.~\eqref{eq:hhd}. Therefore, there exist map-germs
\begin{math}
C \colon \left( \mathbb{R}^n, 0 \right) \rightarrow G_{gp}
\end{math}, 
\begin{math}
B \colon \left( \mathbb{R}^n, 0 \right) \rightarrow M_{q,r} \left( \mathbb{R} \right)
\end{math}
and 
\begin{math}
A \colon \left( \mathbb{R}^n, 0 \right) \rightarrow GL \left( r, \mathbb{R} \right)
\end{math}
such that 
\begin{equation}
\begin{pmatrix}
g \\
h
\end{pmatrix} = 
\begin{pmatrix}
C & B \\
O & A
\end{pmatrix}
\begin{pmatrix}
g' \\
h'
\end{pmatrix}
\mod \langle x_1, \ldots, x_n \rangle^{k+1}_{\mathcal{E}_n}
\end{equation}
holds.

\noindent
\textbf{Proof of ``only if'' part:}
By the assumption, there exist a diffeomorphism-germ 
\begin{math}
\phi \colon \left( \mathbb{R}^n, 0 \right) \rightarrow \left( \mathbb{R}^n, 0 \right)
\end{math}
and map-germs
\begin{math}
C \colon \left( \mathbb{R}^n, 0 \right) \rightarrow G_{gp}
\end{math}, 
\begin{math}
B \colon \left( \mathbb{R}^n, 0 \right) \rightarrow M_{q,r} \left( \mathbb{R} \right)
\end{math}
and 
\begin{math}
A \colon \left( \mathbb{R}^n, 0 \right) \rightarrow GL \left( r, \mathbb{R} \right)
\end{math}
such that 
\begin{equation}
\begin{pmatrix}
g \\
h
\end{pmatrix} \circ \psi = 
\begin{pmatrix}
C & B \\
O & A
\end{pmatrix}
\begin{pmatrix}
g' \\
h'
\end{pmatrix} \mod \langle x_1, \ldots, x_n \rangle^{k+1}_{\mathcal{E}_n}
\end{equation}
holds. Define a homomorphism of 
\begin{math}
\mathbb{R}
\end{math}-algebra
\begin{math}
\overline{\psi^*} \colon Q_k \left( h \right) \rightarrow Q_k \left( h' \right)
\end{math}
as 
\begin{math}
\overline{\psi^*} \left( \pi^k_h \left( f \right) \right) = \pi^k_{h'} \left( f \circ \psi \right)
\end{math}
for 
\begin{math}
f \in \mathcal{E}_n
\end{math}. 
\begin{math}
\overline{\psi^*}
\end{math}
is well-defined. It is because if 
\begin{math}
\pi^k_h \left( f_1 \right) = \pi^k_h \left( f_2 \right)
\end{math}
holds, there exists 
\begin{math}
c_j \in \mathcal{E}_n
\end{math}
for 
\begin{math}
j \in \left\{ 1, \ldots, r \right\}
\end{math}
such that
\begin{equation}
f_1 - f_2 = \sum_{j=1}^r c_j h_j \mod \langle x_1, \ldots, x_n \rangle^{k+1}_{\mathcal{E}_n}
\end{equation}
holds. By composing it with 
\begin{math}
\psi
\end{math}, we obtain 
\begin{equation}
f_1 \circ \psi - f_2 \circ \psi = \sum_{j=1}^r \left( c_j \circ \psi \right) \left( h_j \circ \psi \right) \mod \langle x_1, \ldots, x_n \rangle^{k+1}_{\mathcal{E}_n}
\end{equation}
since 
\begin{math}
\psi^* \langle x_1, \ldots, x_n \rangle^{k+1}_{\mathcal{E}_n} \subset \langle x_1, \ldots, x_n \rangle^{k+1}_{\mathcal{E}_n}
\end{math}
holds. Since 
\begin{math}
h_j \circ \psi = \sum_{l=1}^r A_{jl} h'_l
\end{math}
holds, we obtain 
\begin{math}
\pi^k_{h'} \left( f_1 \circ \psi \right) = \pi^k_{h'} \left( f_2 \circ \psi \right)
\end{math}. It is easy to check that 
\begin{math}
\overline{\psi^*}
\end{math}
is an isomorphism of 
\begin{math}
\mathbb{R}
\end{math}-algebra. Since the image of 
\begin{math}
C
\end{math}
is in 
\begin{math}
G_{gp}
\end{math}, there exists a permutation 
\begin{math}
\sigma \colon \left\{ 1, \ldots, q \right\} \rightarrow \left\{ 1, \ldots, q \right\}
\end{math}
such that 
\begin{equation}
g_j \circ \psi = u_j \cdot g'_{\sigma \left( j \right)} + \sum_{l=1}^r B_{jl} h'_l \mod \langle x_1, \ldots, x_n \rangle^{k+1}_{\mathcal{E}_n}
\end{equation}
holds for some 
\begin{math}
u_j \in \mathcal{E}_n
\end{math}
such that 
\begin{math}
u_j \left( 0 \right) > 0
\end{math}
for all 
\begin{math}
j \in \left\{ 1, \ldots, q \right\}
\end{math}. This implies that 
\begin{equation}
\overline{\psi^*} \left( \pi^k_h \left( g_j \right) \right) = \pi^k_{h'} \left( g_j \circ \psi \right) = \pi^k_{h'} \left( u_j \cdot g'_{\sigma \left( j \right)} \right).
\end{equation}
This proves the lemma. 
\end{proof}

\begin{lemma} \label{lem:K[G]-equiv if original is so}
Let 
\begin{math}
\left( g, h \right)
\end{math}
and 
\begin{math}
\left( g', h' \right)
\end{math}
be two constraint map-germs whose feasible set-germs are non-empty. If
\begin{math}
\left( g, h \right)
\end{math}
and 
\begin{math}
\left( g', h' \right)
\end{math}
are finitely 
\begin{math}
\mathcal{K} \left[ G \right]
\end{math}-determined and 
\begin{math}
\mathcal{K} \left[ G \right]
\end{math}-equivalent, their reductions \begin{math}
\left( g, h \right)_{\iota_{\left( i \right)}, \left( k \right)}
\end{math}
and 
\begin{math}
\left( g', h' \right)_{\iota_{\left( i' \right)}, \left( k' \right)}
\end{math}
having the equal number of inequality and equality constraints are 
\begin{math}
\mathcal{K} \left[ G \right]
\end{math}-equivalent.
\end{lemma}
\begin{proof}
By Proposition~2.1 (2) in \cite{HamadaHayanoTeramoto2025_1}, 
\begin{math}
\left( g, h \right)
\end{math}
and 
\begin{math}
\left( g', h' \right)
\end{math}
have a finite 
\begin{math}
\mathcal{K} \left[ G \right]
\end{math}-codimension. By Lemma~3.3 in \cite{HamadaHayanoTeramoto2025_1}, their reductions have equal or smaller 
\begin{math}
\mathcal{K} \left[ G \right]
\end{math}-codimension. Again by Proposition~2.1 (2) in \cite{HamadaHayanoTeramoto2025_1}, their reductions are finitely 
\begin{math}
\mathcal{K} \left[ G \right]
\end{math}-determined. 
Thus, we can take \begin{math}
k\in \mathbb{N}
\end{math}
so that 
\begin{math}
\left( g, h \right)_{\iota_{\left( i \right)}, \left( k \right)}
\end{math}
and 
\begin{math}
\left( g', h' \right)_{\iota_{\left( i' \right)}, \left( k' \right)}
\end{math}
are both $k$-
\begin{math}
\mathcal{K} \left[ G \right]
\end{math}-determined.

By Lemma~\ref{lem:K[G] and local R-algebra}, there exist an isomorphism of 
\begin{math}
\mathbb{R}
\end{math}-algebra 
\begin{math}
\phi \colon Q_k \left( h \right) \cong Q_k \left( h' \right)
\end{math}, a permutation
\begin{math}
\sigma \colon \left\{ 1, \ldots, q \right\} \rightarrow \left\{ 1, \ldots, q \right\}
\end{math}
and some units 
\begin{math}
u_j \in \mathcal{E}_n
\end{math}
satisfying 
\begin{math}
u_j \left( 0 \right) > 0
\end{math}
such that
\begin{math}
\phi \left( \pi^k_h \left( g_j \right) \right) = \pi^k_{h'} \left( u_j g'_{\sigma \left( j \right)} \right)
\end{math}
holds for all 
\begin{math}
j \in \left\{ 1, \ldots, q \right\}
\end{math}.
Define 
\begin{math}
\overline{\iota_{\left( i \right)}^*} \colon Q_k \left( h \right) \rightarrow Q_k \left( \iota_{\left( i \right)}^* h \right)
\end{math}
as 
\begin{math}
\overline{\iota_{\left( i \right)}^*} \left( \pi^k_h \left( f \right) \right) = \pi^k_{\iota^*_{\left( i \right)} h} \left( f \circ \iota_{\left( i \right)} \right)
\end{math}, then it is easy to check that this is a well-defined isomorphism of 
\begin{math}
\mathbb{R}
\end{math}-algebra. 
Define 
\begin{math}
\tilde{\phi} := \overline{\iota_{\left( i' \right)}^*} \circ \phi \circ \overline{\iota_{\left( i \right)}^*}^{-1}:Q_k \left( \iota_{\left( i \right)}^* h \right)\to  Q_k \left( \iota_{\left( i' \right)}^* h' \right)
\end{math}.
Then, this gives an isomorphism of 
\begin{math}
\mathbb{R}
\end{math}-algebra from 
\begin{math}
Q_k \left( h_{\iota_{\left( i \right)}} \right)
\end{math}
to 
\begin{math}
Q_k \left( h'_{\iota_{\left( i' \right)}} \right)
\end{math}
since
\begin{math}
Q_k \left( h_{\iota_{\left( i \right)}} \right) \cong Q_k \left( \iota_{\left( i \right)}^* h \right)
\end{math}
and 
\begin{math}
Q_k \left( h'_{\iota_{\left( i' \right)}} \right) \cong Q_k \left( \iota_{\left( i' \right)}^* h' \right)
\end{math}
hold. Take any 
\begin{math}
j \in \left\{ 1, \ldots, q \right\}
\end{math}. Then, 
\begin{multline}
\tilde{\phi} \left( \pi^k_{\iota^*_{\left( i \right)} h} \left( g_j \circ \iota_{\left( i \right)} \right) \right) = \overline{\iota_{\left( i' \right)}^*} \circ \phi \circ \overline{\iota_{\left( i \right)}^*}^{-1} \left( \pi^k_{\iota^*_{\left( i \right)} h} \left( g_j \circ \iota_{\left( i \right)} \right) \right) = \overline{\iota_{\left( i' \right)}^*} \circ \phi \left( \pi_h^k \left( g_j \right) \right) \\
= \overline{\iota_{\left( i' \right)}^*} \left( \pi^k_{h'} \left( u_j g'_{\sigma \left( j \right)} \right) \right) = \pi^k_{\iota^*_{\left( i' \right)} h'} \left( \left( u_j \circ \iota_{\left( i' \right)} \right) \left( g'_{\sigma \left( j \right)} \circ \iota_{\left( i' \right)} \right) \right)
\end{multline}
holds. Since 
\begin{math}
u_j \circ \iota_{\left( i' \right)} \left( 0 \right) = u_j \left( 0 \right) > 0
\end{math}
holds, 
\begin{math}
j^k \left( g, h \right)_{\iota_{\left( i \right)}, \left( k \right)}
\end{math}
and 
\begin{math}
j^k \left( g', h' \right)_{\iota_{\left( i' \right)}, \left( k' \right)}
\end{math}
are 
\begin{math}
\mathcal{K} \left[ G \right]^k
\end{math}-equivalent by Lemma~\ref{lem:K[G] and local R-algebra}. Since they are 
\begin{math}
k
\end{math}-\begin{math}
\mathcal{K} \left[ G \right]
\end{math}-determined, this proves the lemma. 
\end{proof}

\begin{lemma}\label{lem:K[G]-codim reduction}
 
The $\mathcal{K}[G]$-codimension of $(g,h)_{\iota_{(i)},(k)}$ is equal to that of $(g,h)$, and the same is true for the $\mathcal{K}[G]_e$-codimension if  $\mathcal{K}[G]$ ($\mathcal{K}[G]_e$)-codimension is finite. 

\end{lemma}

\begin{proof}
If $(g,h)$ is a submersion, the $\mathcal{K}[G]$-codimensions and $\mathcal{K}[G]_e$-codimensions of $(g,h)$ and its reduction are all equal to $0$, in particular the statement holds. 
In what follows, we assume that $(g,h)$ is not a submersion. 
The $\mathcal{K}[G]_e$-codimension is the sum of the $\mathcal{K}[G]$-codimension and $-n+q+r$ by Proposition~2.2 \cite{HamadaHayanoTeramoto2025_1} and $-n+q+r$ is invariant under reduction. (Note that $n,q,r$ are respectively the number of variables, active inequality constraints, and active equality constraints.) 
It is thus enough to show the statement for the $\mathcal{K}[G]_e$-codimension. 

Lemma~\ref{lem:K[G]-equiv if original is so} implies that
\begin{math}
\mathcal{K} \left[ G \right]
\end{math}-action to 
\begin{math}
\left( g, h \right)
\end{math}
does not change 
\begin{math}
\mathcal{K} \left[ G \right]
\end{math}-class of its reduction. Therefore, we can assume $(k) = (s+1,\ldots, q)$, $(i)=(1,\ldots, r-\ell)$, and the following hold:
\begin{align*}
h \left( x \right) =& \left( x_1, \ldots, x_{r-\ell}, h_{r-\ell+1} \left( 0, \ldots, 0, x_{r-\ell+1}, \ldots, x_n \right), \ldots, h_r \left( 0, \ldots, 0, x_{r-\ell+1}, \ldots, x_n \right) \right),\\
g \left( x \right) =& \left( g_{1} \left( 0, \ldots, 0, x_{r-\ell+1}, \ldots, x_n \right), \ldots, g_q \left( 0, \ldots, 0, x_{r-\ell+1}, \ldots, x_n \right) \right), \\
\iota_{\left( i \right)} \left( y \right) =& \left( 0, y \right) \hspace{1em}(\mbox{where $y\in \R^{n-r+\ell}$}).
\end{align*}

By the definition, the tangent space $T\mathcal{K}[G]_e(g,h)$ is equal to 
\begin{equation}
t(g,h)(\mathcal{E}_n^{n}) + h^\ast \mathcal{M}_r \mathcal{E}_n^{q+r} + \left<g_1e_1,\ldots, g_qe_q\right>_{\mathcal{E}_n}.
\end{equation}
The image $t(g,h)(\mathcal{E}_n^{n})$ has the following generating set as an $\mathcal{E}_n$-module: 
\begin{align*}
&\left\{\left.\left(\sum_{i=1}^{q}\frac{\Pa g_i}{\Pa x_j}e_i +\sum_{i=1}^{r}\frac{\Pa h_i}{\Pa x_j}e_{i +q} \right)~\right|~j=1,\ldots, n\right\}\\
=&\left\{e_{q+1},\ldots, e_{q+r-\ell}\right\}\cup \left\{\left.\left(\sum_{i=1}^{q}\frac{\Pa g_i}{\Pa x_j}e_i +\sum_{i=r-\ell+1}^{r}\frac{\Pa h_i}{\Pa x_j}e_{i +q} \right)~\right|~j=r-\ell +1,\ldots, n\right\}.
\end{align*}
Furthermore, $\left<g_1e_1,\ldots, g_qe_q\right>_{\mathcal{E}_{n}}$ contains $\left<e_{s+1},\ldots, e_q\right>_{\mathcal{E}_{n}}$ since $g_{s+1},\ldots, g_q\in \mathcal{E}_n$ are units. 
Since $h^\ast \mathcal{M}_r$ contains $x_1,\ldots, x_{r-\ell}$, the tangent space $T\mathcal{K}[G]_e(g,h)$ contains 
\[
\left<x_1,\ldots, x_{r-\ell}\right>_{\mathcal{E}_n}\mathcal{E}_n^{q+r}+\left<e_{s+1},\ldots, e_{q+r-\ell}\right>_{\mathcal{E}_n},
\]
which is the kernel of the map $\iota_{(i)}^\ast \circ \rho:\mathcal{E}_n^{q+r} \to \mathcal{E}_{n-r+\ell}^{s+\ell}$, where $\rho:\mathcal{E}_n^{q+r}\to \mathcal{E}_{n}^{s+\ell}$ is the projection removing the $s+1,\ldots, q+r-\ell$-th components. 
We thus obtain:
\begin{align*}
&\mathcal{E}_n^{q+r} / T\mathcal{K}[G]_e(g,h)\\
\cong& \left.\left(\frac{\mathcal{E}_n^{q+r}}{\left<x_1,\ldots, x_{r-\ell}\right>_{\mathcal{E}_n}\mathcal{E}_n^{q+r}+ \left<e_{s+1},\ldots, e_{q+r-\ell}\right>_{\mathcal{E}_n}}\right)\right/\left(\frac{T\mathcal{K}[G]_e(g,h)}{\left<x_1,\ldots, x_{r-\ell}\right>_{\mathcal{E}_n}\mathcal{E}_n^{q+r}+\left<e_{s+1},\ldots, e_{q+r-\ell}\right>_{\mathcal{E}_n}}\right).
\end{align*}
It is easy to see that the map $\iota_{(i)}^\ast \circ \rho$ sends $T\mathcal{K}[G]_e(g,h)$ to $T\mathcal{K}[G]_e(g,h)_{\iota_{(i)},(k)}$. 
Thus, $\mathcal{E}_n^{q+r} / T\mathcal{K}[G]_e(g,h)$ is isomorphic to $\mathcal{E}_{n-r+\ell}^{s+\ell} / T\mathcal{K}[G]_e(g,h)_{\iota_{(i)},(k)}$.  
\end{proof}

\subsection{Classification of generic constraint map-germs}
\label{subsec:thm5.1}

We now recall the main theorem in \cite{HamadaHayanoTeramoto2025_1}, which classifies constraint map-germs appearing in generic parameter families of constraints with up to four parameters. 
Compared with Tables 1, 2, and 3 in \cite{HamadaHayanoTeramoto2025_1}, the normal forms presented here have been slightly modified by eliminating signs that can be removed under the $\mathcal{K}[G]$-action. These modifications are not essential and preserve the completeness of the classification.
For the definitions of the $\mathcal{K}[G]$-codimension and $\mathcal{K}[G]$-determinacy, see \cite{HamadaHayanoTeramoto2025_1}.


\begin{theorem}[{\cite[Theorem~5.2]{HamadaHayanoTeramoto2025_1}}]\label{thm:5.1}

Suppose 
\begin{math}
n \gg q, r
\end{math}. 
Let $N$ be an $n$-manifold without boundary, $b\leq 4$, and $U\subset \R^b$ be an open subset. 
The set consisting of constraint mappings $(g,h)\in C^\infty(N\times U,\R^{q+r})$ with the following conditions is residual in $C^\infty(N\times U,\R^{q+r})$. 

\begin{enumerate}

\item 
For any $u\in U$ and ${\overline{x}}\in M(g_u,h_u)$, the corank of $(dh_u)_{\overline{x}}$ is at most $1$. 

\item 
For any $u\in U$ and ${\overline{x}}\in M(g_u,h_u)$ at which there is no active inequality constraint (i.e.,~there is no $k\in \{1,\ldots, q\}$ with $g_k({\overline{x}},u)=0$), a full reduction of the germ $(g,h):(N\times U,({\overline{x}},u))\to \R^{q+r}$ is $\mathcal{K}[G]$-equivalent to either the trivial family of the constant map-germ, or a versal unfolding of one of the germs in Table~\ref{table:generic constraint q=0} with the $\mathcal{K}[G]_e$-codimension at most $b$. 

\end{enumerate}

\noindent
In what follows, we will assume that $(g_u,h_u)$ has an active inequality constraint at $x\in M(g_u,h_u)$.

\begin{enumerate}

\addtocounter{enumi}{2}
\item 
For any $u\in U$ and ${\overline{x}}\in M(g_u,h_u)$ with $\corank((dh_u)_{\overline{x}})=0$, a full reduction of the germ $(g_u,h_u):(N,{\overline{x}})\to \R^{q+r}$ is $\mathcal{K}[G]$-equivalent to either a submersion-germ, or one of the germs in Table~\ref{table:generic constraint r=0} with stratum $\mathcal{K}[G]_e$-codimension at most $b$.
Furthermore, if a full reduction of $(g_u,h_u)$ is $\mathcal{K}[G]$-equivalent to the germ of neither type (6) nor type (10), a full reduction of $(g,h):(N\times U,({\overline{x}},u))\to \R^{q+r}$ is a versal unfolding of $(g_u,h_u)$. 

\item 
For any $({\overline{x}},u)\in N\times U$ with $\corank((dh_u)_{\overline{x}})=1$, a full reduction of the germ $(g_u,h_u):(N,{\overline{x}})\to \R^{q+r}$ is $\mathcal{K}[G]$-equivalent to one of the germs in Table~\ref{table:generic constraint q>0 r=1} with stratum $\mathcal{K}[G]_e$-codimension at most $b$ (in particular the number of active inequality constraints is at most $3$).
Furthermore, if a full reduction of $(g_u,h_u)$ is $\mathcal{K}[G]$-equivalent to the germ of neither type (4) nor type (8), a full reduction of $(g,h):(N\times U,({\overline{x}},u))\to \R^{q+r}$ is a versal unfolding of $(g_u,h_u)$. 

\end{enumerate}

\end{theorem}

\vspace*{1em}

\begin{table}[htbp]
\renewcommand{\arraystretch}{1.2}
 \begin{center}
  \begin{tabular}{|c|l|c|c|c|} \hline
   type & jet & range &$\mathcal{K}$-determinacy & $\mathcal{K}[G]_e$-cod. \\ \hline
   $(1,k)$ & $x_1^k +\sum_{j=2}^{n}\epsilon_jx_j^2 $ & $2\leq k \leq 5$ &$k$& $k-1$ \\
\hline   $(2)$ & $x_1^3 +\epsilon_2 x_1 x_2^2 + x_3^2+ \sum_{j=4}^{n}\epsilon_j x_j^2$ & &$3$ & $4$ \\
 \hline
  \end{tabular}
  \caption{The $\mathcal{K}[G]$-equivalent classes of map-germs without inequality constraints appearing as a full reduction of a generic four-parameter family of constraint mappings, where $\epsilon_j \in \{1,-1\}$.}
  \label{table:generic constraint q=0}
 \end{center}
\end{table}

\begin{landscape}
\begin{table}[tp]
 \begin{center}
{\renewcommand{\arraystretch}{1.2}
  \begin{tabular}{|c|l|c|c|c|c|} \hline
   type& $\tilde{g} \left( x_{l+1}, \ldots, x_n \right)$ & $l$ & range & 
$\mathcal{K}[G]$-determinacy
& \begin{minipage}[c]{20mm}
\V{.2em}
(stratum)

$\mathcal{K}[G]_e$-cod.
\V{.2em}
\end{minipage}
\\ \hline
   $(1,k)$ & $\epsilon_q x_q^k + \sum_{j=q+1}^n \epsilon_jx_j^2$ & \multirow{2}{*}{$q-1$} & $2\leq k \leq 5$ &$k$ &$k-1$ \\ \cline{1-2}\cline{4-6}
   $(2)$ & $x_q^3 +\epsilon_{q+1} x_q x^2_{q+1} + \sum_{j=q+2}^n \epsilon_j x_j^2$ &&  &$3$& $4$ \\ \hline
   $(3,k)$ & $\epsilon_{q-1} x_{q-1}^k + \sum_{j=q}^n \epsilon_{j} x_j^2$ & &  $2\leq k \leq 4$ &$k$& $k$ \\ \cline{1-2}\cline{4-6}
   $(4,k)$ & $\epsilon_{q} x_q^k + x_{q-1} x_q + \sum_{j=q+1}^n \epsilon_{j} x_j^2$ & $q-2$ & $k=3,4$ &$k$& $k$ \\ \cline{1-2}\cline{4-6}
   $(5)$ & $\epsilon_{q-1} x_{q-1}^2 +x_q^3 + \sum_{j=q+1}^n \epsilon_{j} x_j^2$ & &  &$3$& $4$ \\ \hline
   $(6)$ & $\sum_{j=1}^{2} \delta_{j} x_{q-j}^2 + \alpha x_{q-2} x_{q-1} + \sum_{j=q}^n \epsilon_{j} x_j^2$ & & $\alpha\in \R,\delta_j=\pm 1, (\ast)$ &$2$& $3$ \\ \cline{1-2}\cline{4-6}
   $(7)$ & $\epsilon_{q-2} \left( x_{q-2}+ \epsilon_{q-1}' x_{q-1} \right)^2 +\epsilon_{q-1} x_{q-1}^3 + \sum_{j=q}^n \epsilon_{j} x_j^2$ &  & &$3$& $4$ \\ \cline{1-2}\cline{4-6}
   $(8)$ & $\epsilon_{q-2} x_{q-2}^3 +\epsilon_{q-1} x_{q-1}^2 +\epsilon_{q-1}' x_{q-2} x_{q-1} + \sum_{j=q}^n \epsilon_{j} x_j^2$ &$q-3$ &  &$3$& $4$ \\ \cline{1-2}\cline{4-6}
   $(9)$ & $\begin{matrix*}[l] &x_q^3  +\epsilon_{01} x_{q-1} x_q+\epsilon_{02} x_{q-2} x_q \\ &\epsilon_{12} x_{q-2} x_{q-1} + \sum_{j=q+1}^n \epsilon_{j} x_j^2\end{matrix*}$ & && $3$ & $4$ \\ \hline
   $(10)$ & $\begin{matrix*}[l] &\sum_{j=1}^3 \delta_j x_{q-4+j}^2 \\ &+ \sum_{1 \le i < j \le 3} \alpha_{ij} x_{q-4+i} x_{q-4+j} \\ &\epsilon_{0} x_{q-3} x_{q-2} x_{q-1} + \sum_{j=q}^n \epsilon_{j} x_j^2\end{matrix*}$ & $q-4$ & $
\alpha_{ij}\in \R,\delta_j=\pm 1, (\ast\ast)$ &$3$& $4$ \\
\hline
  \end{tabular}
}
  \captionsetup{singlelinecheck=off}
  \caption[.]{Normal forms 
$  	\left( x_1, \ldots, x_{q-1}, \sum_{j=1}^{l_1} x_j - \sum_{j=l_1+1}^l x_j + \tilde{g} \left( x_{l+1}, \ldots, x_n \right) \right)$
  of map-germs without equality constraints appearing as a full reduction of a generic four-parameter family of constraint mappings, where $\epsilon_j,\epsilon_j',\epsilon_{ij}\in \{1,-1\}$, $0\leq l_1 \leq \lceil \frac{l}{2} \rceil$, $(\ast)$ (for type $(6)$) is the condition $4 \delta_{1} \delta_{2} - \alpha^2 \neq 0$, and $(\ast\ast)$ (for type $(10)$) is the following condition: 
\begin{center}
$4 \delta_i \delta_j - \alpha_{ij}^2 \neq 0\hspace{.3em}(i,j\in\{1,2,3\},\hspace{.3em}i\neq j),\hspace{.3em} 4 \delta_1 \delta_2 \delta_3 + \alpha_{12} \alpha_{13} \alpha_{23} - \delta_3 \alpha_{12}^2 - \delta_{2} \alpha_{13}^2 - \delta_{3} \alpha_{23}^2 \neq 0.$
\end{center}
}
  \label{table:generic constraint r=0}
 \end{center}
\end{table}
\end{landscape}

\begin{table}[tbp]
 \begin{center}
{\renewcommand{\arraystretch}{1.2}
  \begin{tabular}{|c|l|c|c|c|c|} \hline
   type & $h$ &$q$& range &\begin{minipage}[c]{10mm}
\centering
\V{.2em}
$\mathcal{K}[G]$-

det.
\V{.2em}
\end{minipage}& \begin{minipage}[c]{20mm}
\V{.2em}
(stratum)

$\mathcal{K}[G]_e$-cod.
\V{.2em}
\end{minipage}\\ \hline
   $(1,k)$ & $x_1^k + \sum_{j=2}^n \epsilon_{j} x_j^2$ & &$2\leq k \leq 4$ &$k$ & $k$ \\ \cline{1-2}\cline{4-6}
   $(2)$ & $x_2^3 + x_1^2+ \sum_{j=3}^n \epsilon_{j} x_j^2$ &$1$&  &$3$& $4$ \\ \cline{1-2}\cline{4-6}
   $(3,k)$ & $x_2^k +\epsilon_{1} x_1 x_2 + \sum_{j=3}^n \epsilon_{j} x_j^2$ && $3\leq k \leq 4$ &$k$& $k$ \\ \hline
   $(4)$ & $\delta_{1} x_1^2 + \delta_{2} x_2^2 + \alpha x_1 x_2 + \sum_{j=3}^n \epsilon_{j} x_j^2$ &\multirow{4}{*}{$2$}&$\alpha\in \R, \delta_j=\pm 1, (\ast)$ &$2$& $3$ \\ \cline{1-2}\cline{4-6}
   $(5)$ & $x_1^3 +\epsilon_{1} x_2^2 +\epsilon_{2} x_1 x_2 + \sum_{j=3}^n \epsilon_{j} x_j^2$ &&&$3$& $4$ \\ \cline{1-2}\cline{4-6}
   $(6)$ & $\left( x_1 +\epsilon_{1} x_2 \right)^2 +\epsilon_{2} x_2^3 + \sum_{j=3}^n \epsilon_{j} x_j^2$ && &$3$& $4$ \\ \cline{1-2}\cline{4-6}
   $(7)$ & $\begin{matrix}
x_3^3+ \epsilon_{1} x_2 x_3 +\epsilon_{2} x_3 x_1 +\epsilon_{3} x_1 x_2\\ + \sum_{j=4}^n \epsilon_{j} x_j^2
\end{matrix}$ & & &$3$& $4$ \\ \hline
   $(8)$ & $\begin{matrix}
\sum_{j=1}^3 \delta_j x_j^2 + \sum_{1 \le i < j \le 3} \alpha_{ij} x_i x_j \\ +\epsilon_{1} x_1 x_2 x_3 + \sum_{j=4}^n \epsilon_{j} x_j^2\end{matrix}$ &$3$& $\alpha_{ij}\in \R, \delta_j=\pm 1, (\ast\ast)$ &$3$& $4$ \\
\hline
  \end{tabular}
}
\caption{Normal forms $(g_1(x),\ldots, g_q(x),h(x))=\left(x_1,\ldots, x_q,h(x)\right)$ of map-germs with equality/inequality constraints appearing as a full reduction of a generic four-parameter family of constraint mappings, where $\epsilon_j\in \{1,-1\}$, $(\ast)$ and $(\ast\ast)$ are the same conditions as those in Table~\ref{table:generic constraint r=0}.
}
  \label{table:generic constraint q>0 r=1}
 \end{center}
\end{table}

\section{Definitions and properties of constraint qualifications} \label{sec:CQ-definition_and_properties}
In this section, we provide precise definitions of the four classical CQs and prove that these are invariant under $\mathcal{K}[G]$-equivalence (Theorem~\ref{thm:invariance four CQs under K[G]-eq}) and reductions (Theorem~\ref{thm:invariance four CQs under reduction}). 
Although CQs were originally defined for constraint mappings, we will deal with map-germs $(g,h):(\R^n,0)\to \R^{q+r}$ since all the CQs discussed in this paper are local properties.
In what follows, we assume that the set-germ $M(g,h)$ is not empty (i.e.,~$g_j(x)\leq 0$ for any $j\in \{1,\ldots, q\}$).
Let 
\begin{math}
I= \left\{ j \in \left\{ 1, \ldots, q \right\} \middle| g_j \left( 0 \right) = 0 \right\}
\end{math}
be the set of indices of the active inequality constraints and 
\begin{math}
g_{I} = \left( g_j \right)_{j \in I}
\end{math}.

\begin{definition}[LICQ \cite{Hestenes1966}]\label{def:LICQ}
A constraint 
\begin{math}
\left( g, h \right)
\end{math}
satisfies \emph{linear independence constraint qualification (LICQ)} if the Jacobi matrix of 
\begin{math}
\left( g_{I}, h \right)
\end{math}
has corank 
\begin{math}
0
\end{math}.
\end{definition}
\begin{definition}[Mangasarian-Fromovitz, MFCQ \cite{Mangasarian1967}]\label{def:MFCQ}
A constraint 
\begin{math}
\left( g, h \right)
\end{math}
satisfies \emph{Mangasarian-Fromovitz constraint qualification (MFCQ)}
if the Jacobi matrix of 
\begin{math}
h
\end{math}
has corank $0$ and there exists a vector 
\begin{math}
d \in \mathbb{R}^n
\end{math}
such that 
\begin{math}
dg_{j,0} \left( d \right) < 0
\end{math}
holds for all 
\begin{math}
j \in I
\end{math}
and 
\begin{math}
dh_{j,0} \left( d \right) = 0
\end{math}
holds for all 
\begin{math}
j \in \left\{ 1, \ldots, r \right\}
\end{math}. We call such a vector 
\begin{math}
d
\end{math}
an \emph{MF-vector} of the constraint in what follows.
\end{definition}
\noindent
Let 
\begin{equation}
C^+ \left( g, h \right) = \left\{ d \in \mathbb{R}^n \middle| \exists \left\{ x_l \right\}_l \subset M \left( g, h \right), \lim_{l \rightarrow \infty} x_l = 0, \exists \left\{ t_l \right\}_l \subset \mathbb{R}_{>0}, d = \lim_{l \rightarrow \infty} t_l x_l \right\}
\end{equation}
be the tangent cone of the feasible set-germ $M(g,h)$, and 
\begin{equation}
L^+ \left( g, h \right) = \left\{ d \in\mathbb{R}^n \middle| \forall j \in I, dg_{j,0}\left( d \right) \le 0, dh_0 \left( d \right) = 0 \right\}
\end{equation}
be the linearized cone of $(g,h)$.
Note that when we consider a constraint map-germ without (in)equality constraints, we denote its tangent cone by $C^+(g)$ or $C^+(h)$, and the same for the linearized cone.
The reader can refer to \cite{Bazarra1974} for basic properties of tangent/linearized cones.

\begin{definition}[Abadie, ACQ \cite{Abadie1967}] \label{def:ACQ}
A constraint 
\begin{math}
\left( g, h \right)
\end{math}
satisfies \emph{Abadie constraint qualification (ACQ)} if 
\begin{math}
C^+ \left( g, h \right) = L^+ \left( g, h \right)
\end{math}.
\end{definition}
\noindent
For a subset 
\begin{math}
X \subset \mathbb{R}^n
\end{math}, 
let
\begin{math}
X^\circ = \left\{ v \in \mathbb{R}^n\middle| \forall d \in X, v \cdot d \le 0 \right\}
\end{math}, which is called the \emph{polar} of $X$.

\begin{definition}[Guignard, GCQ \cite{Guignard1969}] \label{def:GCQ}
A constraint 
\begin{math}
\left( g, h \right)
\end{math}
satisfies \emph{Guignard constraint qualification (GCQ)} if 
\begin{math}
C^+ \left( g, h \right)^\circ = L^+ \left( g, h \right)^\circ
\end{math}.
\end{definition}

\begin{theorem}\label{thm:invariance four CQs under K[G]-eq}
Above constraint qualifications \textnormal{LICQ, MFCQ, ACQ, GCQ} are invariant under the action of 
\begin{math}
\mathcal{K} \left[ G \right]
\end{math}.
\end{theorem}

\begin{proof}

Let $(g,h)$ be a constraint map-germ, \begin{math}
\phi \colon \left( \mathbb{R}^n, 0 \right) \rightarrow \left( \mathbb{R}^n, 0 \right)
\end{math} be a diffeomorphism-germ, and
\begin{math}
C \colon \left( \mathbb{R}^n, 0 \right) \rightarrow G_{qp}
\end{math}, 
\begin{math}
B \colon \left( \mathbb{R}^n, 0 \right) \rightarrow M_{q,r} \left( \mathbb{R} \right)
\end{math}, and 
\begin{math}
A \colon \left( \mathbb{R}^n, 0 \right) \rightarrow GL \left( r, \mathbb{R} \right)
\end{math} be map-germs. 
We define $(g',h')$ as follows:
\begin{equation}
\begin{pmatrix}
g' \circ \phi \left( x \right) \\
h' \circ \phi \left( x \right)
\end{pmatrix}
= \left(
\begin{array}{c|c}
C \left( x \right) & B \left( x \right) \\
\midrule
O_{r,q} & A \left( x \right)
\end{array}
\right)
\begin{pmatrix}
g \left( x \right) \\
h \left( x \right)
\end{pmatrix}. \label{eq:kgeq}
\end{equation}

\noindent
In what follows, we assume that $(g,h)$ satisfies the CQs and show that so does $(g',h')$.
Since inactive inequality constraints at the origin are irrelevant for the CQs, we assume that all the inequality constraints are active, i.e., 
\begin{math}
g \left( 0 \right) = 0
\end{math}
in this proof. 

\noindent
{\bf LICQ:} Since 
\begin{math}
\mathcal{K} \left[ G \right]
\end{math}
is a subgroup of 
\begin{math}
\mathcal{K}
\end{math} and the action of $\mathcal{K}$ preserves the rank of the Jacobi matrix, it is obvious that $(g',h')$ also satisfies LICQ.

\noindent
{\bf MFCQ:}
By the assumption,  
\begin{math}
h
\end{math}
and 
\begin{math}
h'
\end{math}
are 
\begin{math}
\mathcal{K}
\end{math}-equivalent. Since 
\begin{math}
dh_0
\end{math}
has corank $0$, 
\begin{math}
dh'_0
\end{math}
has corank $0$ as well. Since the constraint 
\begin{math}
\left( g, h \right)
\end{math}
satisfies \textnormal{MFCQ} at the origin, there exists an \textnormal{MF}-vector 
\begin{math}
d \in \mathbb{R}^n
\end{math}
such that 
\begin{math}
dg_0 \left( d \right) < 0
\end{math}
and 
\begin{math}
dh_0 \left( d \right) = 0
\end{math}. 
By differentiating both hand sides of Eq.~\eqref{eq:kgeq} by 
\begin{math}
x
\end{math}
and taking the inner product with the vector 
\begin{math}
d
\end{math}, we obtain
\begin{align}
\left( dg'_0 \circ d\phi_0 \left( d \right), dh'_0 \circ d\phi_0 \left( d \right) \right) &= \left( C \left( 0 \right) dg_0 \left( d \right) + B \left( 0 \right) dh_0 \left( d \right), A \left( 0 \right) dh_0 \left( d \right) \right) \\
&= \left( C \left( 0 \right) dg_0 \left( d \right), 0 \right), 
\end{align}
where we used 
\begin{math}
dh_0 \left( d \right) = 0
\end{math}. In addition, 
\begin{math}
C \left( 0 \right) dg_0 \left( d \right) < 0
\end{math}
holds because 
\begin{math}
dg_0 \left( d \right) < 0
\end{math}
and 
\begin{math}
C \left( 0 \right)
\end{math}
is a generalized permutation matrix. This proves 
\begin{math}
dg'_0 \left( d\phi_0 \left( d \right) \right) < 0
\end{math}
and 
\begin{math}
dh'_0 \left( d\phi_0 \left( d \right) \right) = 0
\end{math}, which implies that 
\begin{math}
d\phi_0 \left( d \right)
\end{math}
is an \textnormal{MF}-vector of the constraint 
\begin{math}
\left( g', h' \right)
\end{math}.
This proves that \textnormal{MFCQ} is invariant under the action of 
\begin{math}
\mathcal{K} \left[ G \right]
\end{math}.

\noindent
{\bf ACQ:} 
Since 
\begin{math}
C^+ \left( g', h' \right) \subset L^+ \left( g', h' \right)
\end{math}
always holds, it is enough to show 
\begin{math}
C^+ \left( g', h' \right) \supset L^+ \left( g', h' \right)
\end{math}. 
If 
\begin{math}
d \in L^+ \left( g', h' \right)
\end{math}, then 
\begin{math}
\left( d\phi_0 \right)^{-1} \left( d \right) \in L^+ \left( g, h \right)
\end{math}
holds by using the similar argument in case of \textnormal{MFCQ}. Since 
\begin{math}
C^+ \left( g, h \right) = L^+ \left( g, h \right)
\end{math}
holds, 
\begin{math}
\left( d\phi_0 \right)^{-1} \left( d \right) \in C^+ \left( g, h \right)
\end{math}
holds as well. By definition, there exists sequences 
\begin{math}
\left\{ x_l \right\}_l \subset M \left( g, h \right)
\end{math}
and 
\begin{math}
\left\{ t_l \right\}_l \subset \mathbb{R}_{>0}
\end{math}
such that 
\begin{math}
\lim_{l \rightarrow \infty} x_l = 0
\end{math}
and 
\begin{math}
\left( d\phi_0 \right)^{-1} \left( d \right) = \lim_{l \rightarrow \infty} t_l x_l
\end{math}
holds. In that case, the sequence 
\begin{math}
\left\{ \phi \left( x_l \right) \right\}_l \subset M \left( g', h' \right)
\end{math}
satisfies 
\begin{math}
\lim_{l \rightarrow \infty} \phi \left( x_l \right) = 0
\end{math}
and 
\begin{align}
\lim_{l \rightarrow \infty} t_l \phi \left( x_l \right) &= \lim_{l \rightarrow \infty} t_l \left( \phi \left( x_l \right) - \phi \left( 0 \right) \right) \\
&= \lim_{l \rightarrow \infty} \left( t_l d\phi_0 \left( x_l \right) + t_l O \left( \left\| x_l \right\|^2 \right) \right)  \\
&= d \phi_0 \left( \lim_{l \rightarrow \infty} t_l x_l \right) = d.
\end{align}
This implies that 
\begin{math}
d \in C^+ \left( g', h' \right)
\end{math}.

\noindent
{\bf GCQ:}
The argument in case of \textnormal{ACQ} shows $d\phi_0 \left( C^+ \left( g, h \right) \right) = C^+ \left( g', h' \right)$ and 
$d\phi_0 \left( L^+ \left( g, h \right) \right) = L^+ \left( g', h' \right)$.
This implies that 
\begin{align}
C^+ \left( g', h' \right)^\circ &= d\phi_0 \left( C^+ \left( g, h \right) \right)^\circ \\
&= \left\{v\in \R^n~|~\forall d\in C^+(g,h),\hspace{.5em}  v\cdot d\phi_0(d)\leq 0\right\}\\
&= \left\{v\in \R^n~|~\forall d\in C^+(g,h),\hspace{.5em}  \left(d\phi_0\right)^\ast (v)\cdot d\leq 0\right\}\\
&=\left(\left( d\phi_0 \right)^*\right)^{-1} \left( C^+ \left( g, h \right)^\circ \right) \\
&= \left(\left( d\phi_0 \right)^*\right)^{-1} \left( L^+ \left( g, h \right)^\circ \right) \\
&= d\phi_0 \left( L^+ \left( g, h \right) \right)^\circ \\
&= L^+ \left( g', h' \right)^\circ,
\end{align}
where 
\begin{math}
\left( d\phi_0 \right)^*
\end{math}
is the adjoint of 
\begin{math}
d\phi_0
\end{math}. 
\end{proof}

We next discuss invariance of CQs under reductions.
Let 
\begin{math}
\left( g, h \right) \colon \left( \mathbb{R}^n, 0 \right) \rightarrow \mathbb{R}^q \times \mathbb{R}^r
\end{math}
be a constraint map-germ, and $\left( g, h \right)_{\iota_{\left( i \right)}, \left( k \right)}$ be a reduction of $(g,h)$ relative to $\left( k \right) = \left( k_1, \ldots, k_{q-s} \right)$ and \begin{math}
\left( i \right) = \left( i_1, \ldots, i_{r-\ell} \right)
\end{math}.
%

\begin{theorem}\label{thm:invariance four CQs under reduction}
A constraint map-germ 
\begin{math}
\left( g, h \right)
\end{math}
satisfies LICQ, MFCQ, ACQ, and GCQ at the origin if and only if its reduction 
\begin{math}
\left( g, h \right)_{\iota_{\left( i \right)}, \left( k \right)}
\end{math}
satisfies LICQ, MFCQ, ACQ, and GCQ, respectively.
\end{theorem}

In order to prove the theorem, we first prepare several lemma. 
Let $MF(g,h)\subset \R^n$ be the set of MF-vectors of $(g,h)$, that is, 
\[
{MF}(g,h) = \left\{d\in \R^n~\middle|~dh_0(d)=0, dg_{j,0}(d)<0 \mbox{ for }\forall j \mbox{ with }g_j(0)=0\right\}.
\]

\begin{lemma} \label{lem:cone_reduction}
The following equalities hold:
\begin{align}
L^+ \left( g, h \right) =& d \left( \iota_{\left( i \right)} \right)_0 \left( L^+ \left( \left( g, h \right)_{\iota_{\left( i \right)}, \left( k \right)} \right) \right) \label{eq:linearized_cone_reduction}\\
MF \left( g, h \right) =& d \left( \iota_{\left( i \right)} \right)_0 \left( MF \left( \left( g, h \right)_{\iota_{\left( i \right)}, \left( k \right)} \right) \right) \label{eq:MFvector_reduction}\\
C^+ \left( g, h \right) =& d \left( \iota_{\left( i \right)} \right)_0 \left( C^+ \left( \left( g, h \right)_{\iota_{\left( i \right)}, \left( k \right)} \right) \right) \label{eq:cone_reduction}
\end{align}
hold. 
\end{lemma}
\begin{proof}
\noindent
\textbf{Proof of Eq.~\eqref{eq:linearized_cone_reduction} ($\supset$):}
Take any 
\begin{math}
d \in L^+ \left( \left( g, h \right)_{\iota_{\left( i \right)}, \left( k \right)} \right)
\end{math}. Then, by definition 
\begin{math}
d \left( g_j \circ \iota_{\left( i \right)} \right)_0 \left( d \right) \le 0
\end{math}
for all 
\begin{math}
j \in \left\{ 1, \overset{\hat{k}}{\ldots}, q \right\}
\end{math}
with 
\begin{math}
g_j \left( 0 \right) = 0
\end{math},
and 
\begin{math}
d \left( h_j \circ \iota_{\left( i \right)} \right)_0 \left( d \right) = 0
\end{math}
for all 
\begin{math}
j \in \left\{ 1, \overset{\hat{i}}{\ldots}, r \right\}
\end{math}. The latter implies that 
\begin{math}
d \left( h_j \right)_0 \left( d \left( \iota_{\left( i \right)} \right)_0 \left( d \right) \right) = 0
\end{math}
holds for
\begin{math}
j
\end{math}
in the same set. Along with the fact that 
\begin{math}
\Im d \left( \iota_{\left( i \right)} \right)_0 = \Ker d \left( h_{i_1}, \ldots, h_{i_{r-\ell}} \right)_0
\end{math}, 
we obtain 
\begin{math}
d \left( h_j \right)_0 \left( d \left( \iota_{\left( i \right)} \right)_0 \left( d \right) \right) = 0
\end{math}
for all 
\begin{math}
j \in \left\{ 1, \ldots, r \right\}
\end{math}. This proves that 
\begin{math}
d \left( \iota_{\left( i \right)} \right)_0 \left( d \right) \in L^+ \left( g, h \right)
\end{math}
and thus 
\begin{math}
L^+ \left( g, h \right) \supset d \left( \iota_{\left( i \right)} \right)_0 \left( L^+ \left( \left( g, h \right)_{\iota_{\left( i \right)}, \left( k \right)} \right) \right)
\end{math}
holds since 
\begin{math}
d
\end{math}
was taken arbitrarily.

\noindent
\textbf{Proof of Eq.~\eqref{eq:linearized_cone_reduction} ($\subset$):}
Take any 
\begin{math}
d \in L^+ \left( g, h \right)
\end{math}. Then, 
\begin{equation}
d \in \Ker d \left( h_1, \ldots, h_r \right) \subset \Ker d \left( h_1, \overset{\hat{i}}{\ldots}, h_r \right) = \Im d \left( \iota_{\left( i \right)} \right)_0
\end{equation}
holds and thus there exists 
\begin{math}
d' \in \mathbb{R}^{n-r-\ell}
\end{math}
such that 
\begin{math}
d = d \left( \iota_{\left( i \right)} \right)_0 \left( d' \right)
\end{math}
holds. Then, 
\begin{math}
0 \ge d \left( g_j \right)_0 \left( d \right) = d \left( g_j \circ \iota_{\left( i \right)} \right)_0 \left( d' \right)
\end{math}
holds for all 
\begin{math}
j \in \left\{ 1, \overset{\hat{k}}{\ldots}, q \right\}
\end{math}
and 
\begin{math}
0 = d \left( h_j \right)_0 \left( d \right) = d \left( h_j \circ \iota_{\left( i \right)} \right)_0 \left( d' \right)
\end{math}
holds for all 
\begin{math}
j \in \left\{ 1, \overset{\hat{i}}{\ldots}, q \right\}
\end{math}. This implies that 
$d'$ is contained in $L^+ \left( \left( g, h \right)_{\iota_{\left( i \right)}, \left( k \right)} \right)$,
and thus 
$d$ is contained in $d \left( \iota_{\left( i \right)} \right)_0 \left( L^+ \left( \left( g, h \right)_{\iota_{\left( i \right)}, \left( k \right)} \right) \right)$.

\noindent
\textbf{Proof of Eq.~\eqref{eq:MFvector_reduction}} is omitted since it is quite similar to that of Eq.~\eqref{eq:linearized_cone_reduction}.

\noindent
\textbf{Proof of Eq.~\eqref{eq:cone_reduction} ($\supset$):} Take any 
\begin{math}
d \in C^+ \left( \left( g, h \right)_{\iota_{\left( i \right)}, \left( k \right)} \right)
\end{math}. By definition, there exist sequences 
\begin{math}
\left\{ x_j \right\}_j
\end{math}
in 
\begin{math}
M \left( \left( g, h \right)_{\iota_{\left( i \right)}, \left( k \right)} \right)
\end{math}
and 
\begin{math}
\left\{ t_j \right\}_j
\end{math}
in 
\begin{math}
\mathbb{R}_{> 0}
\end{math}
such that 
\begin{math}
x_j \rightarrow 0
\end{math}
and 
\begin{math}
t_j x_j \rightarrow d
\end{math}
as 
\begin{math}
j \rightarrow \infty
\end{math}. If 
\begin{math}
j
\end{math}
is sufficiently large, 
\begin{math}
\iota_{\left( i \right)} \left( x_j \right) \in M \left( g, h \right)
\end{math}
holds. In addition, 
\begin{math}
\iota_{\left( i \right)} \left( x_j \right) \rightarrow 0
\end{math}
and 
\begin{math}
t_j \cdot \iota_{\left( i \right)} \left( x_j \right) \rightarrow d \left( \iota_{\left( i \right)} \right)_0 \left( d \right)
\end{math}
holds. The latter holds since 
\begin{align}
t_j \cdot \iota_{\left( i \right)} \left( x_j \right) &= \iota_{\left( i \right)} \left( 0 \right) + t_j \cdot d \left( \iota_{\left( i \right)} \right)_0 \left( x_j \right) + t_j O \left( \left\| x_j \right\|^2 \right), \\
&= d \left( \iota_{\left( i \right)} \right)_0 \left( t_j x_j \right) + \left\| t_j x_j \right\| O \left( \left\| x_j \right\| \right), \\
&\rightarrow d \left( \iota_{\left( i \right)} \right)_0 \left( d \right) \quad \left( j \rightarrow \infty \right).
\end{align}
This proves
\begin{math}
C^+ \left( g, h \right) \supset d \left( \iota_{\left( i \right)} \right)_0 \left( C^+ \left( \left( g, h \right)_{\iota_{\left( i \right)}, \left( k \right)} \right) \right)
\end{math}. 

\noindent
\textbf{Proof of Eq.~\eqref{eq:cone_reduction} ($\subset$):} Take any 
\begin{math}
d \in C^+ \left( g, h \right)
\end{math}. By definition, there exist sequences 
\begin{math}
\left\{ x_j \right\}
\end{math}
in 
\begin{math}
M \left( g, h \right)
\end{math}
and 
\begin{math}
\left\{ t_j \right\}_j
\end{math}
in 
\begin{math}
\mathbb{R}_{> 0}
\end{math}
such that 
\begin{math}
x_j \rightarrow 0
\end{math}
and 
\begin{math}
t_j x_j \rightarrow d
\end{math}
as 
\begin{math}
j \rightarrow \infty
\end{math}. Since 
\begin{equation}
x_j \in M \left( g, h \right) \subset \left( h_{i_1}, \ldots, h_{i_{r-\ell}} \right)^{-1} \left( 0 \right) = \Im \iota_{\left( i \right)}
\end{equation}
holds for each 
\begin{math}
j
\end{math}, there exists a sequence 
\begin{math}
x'_j \in \left( \mathbb{R}^{n-r+\ell}, 0 \right)
\end{math}
such that 
\begin{math}
x_j = \iota_{\left( i \right)} \left( x'_j \right)
\end{math}
holds. Since 
\begin{math}
x_j \rightarrow 0
\end{math}
as 
\begin{math}
j \rightarrow \infty
\end{math}, 
\begin{math}
\iota_{\left( i \right)} \left( x'_j \right) \rightarrow 0
\end{math}
as
\begin{math}
j \rightarrow \infty
\end{math}
holds. Since
\begin{math}
\iota_{\left( i \right)}
\end{math}
is a homeomorphism to its image, 
\begin{math}
x'_j \rightarrow 0
\end{math}
as
\begin{math}
j \rightarrow \infty
\end{math}
follows. Since 
\begin{equation}
t_j x_j = t_j \cdot \iota_{\left( i \right)} \left( x'_j \right) = t_j d \left( \iota_{\left( i \right)} \right)_0 \left( x'_j \right) + t_j O \left( \left\| x_j \right\|^2 \right) =  d \left( \iota_{\left( i \right)} \right)_0 \left( t_j x'_j \right) + \left\| t_j x_j \right\| O \left( \left\| x_j \right\| \right)
\end{equation}
holds for each 
\begin{math}
j
\end{math}, we obtain 
\begin{math}
d = \lim_{j \rightarrow \infty} t_j x_j = \lim_{j \rightarrow \infty} d \left( \iota_{\left( i \right)} \right)_0 \left( t_j x'_j \right)
\end{math}. This implies that 
\begin{align}
\left( d \left( \iota_{\left( i \right)} \right)_0 \right)^{-1} \left( d \right) &= \left( d \left( \iota_{\left( i \right)} \right)_0 \right)^{-1} \left( \lim_{j \rightarrow \infty} d \left( \iota_{\left( i \right)} \right)_0 \left( t_j x'_j \right) \right) \\
&= \lim_{j \rightarrow \infty} \left( d \left( \iota_{\left( i \right)} \right)_0 \right)^{-1} \left( d \left( \iota_{\left( i \right)} \right)_0 \left( t_j x'_j \right) \right) \\
&= \lim_{j \rightarrow \infty} t_j x'_j
\end{align}
holds. By definition, 
\begin{math}
d' = \lim_{j \rightarrow \infty} t_j x'_j \in C^+ \left( \left( g, h \right)_{\iota_{\left( i \right)}, \left( k \right)} \right)
\end{math}
holds, this proves
\begin{math}
C^+ \left( g, h \right) \subset d \left( \iota_{\left( i \right)} \right)_0 \left( C^+ \left( \left( g, h \right)_{\iota_{\left( i \right)}, \left( k \right)} \right) \right)
\end{math}. 
\end{proof}

\begin{lemma} \label{lem:polar_reduction}
The following equalities hold: 
\begin{align*}
\left( d \left( \iota_{\left( i \right)} \right)_0 \right)^* \left( L^+ \left( g, h \right)^\circ \right) =& L^+ \left( \left( g, h \right)_{\iota_{\left( i \right)}, \left( k \right)} \right)^\circ,\\
\left( d \left( \iota_{\left( i \right)} \right)_0 \right)^* \left( C^+ \left( g, h \right)^\circ \right) =& C^+ \left( \left( g, h \right)_{\iota_{\left( i \right)}, \left( k \right)} \right)^\circ.
\end{align*}
\end{lemma}

\begin{proof}
The first equality follows from the following equalities:
\begin{align*}
\left( d \left( \iota_{\left( i \right)} \right)_0 \right)^* \left( L^+ \left( g, h \right)^\circ \right) =& \left( d \left( \iota_{\left( i \right)} \right)_0 \right)^* \left(\left(d \left( \iota_{\left( i \right)} \right)_0 \left( L^+ \left( \left( g, h \right)_{\iota_{\left( i \right)}, \left( k \right)} \right) \right)\right)^\circ \right) & (\because \mbox{Eq.~}\eqref{eq:linearized_cone_reduction})\\
=& \left( d \left( \iota_{\left( i \right)} \right)_0 \right)^* \left(\left(\left(d \left( \iota_{\left( i \right)} \right)_0\right)^\ast\right)^{-1} \left( L^+ \left( \left( g, h \right)_{\iota_{\left( i \right)}, \left( k \right)} \right)^\circ \right) \right)\\
=& L^+ \left( \left( g, h \right)_{\iota_{\left( i \right)}, \left( k \right)} \right)^\circ,
\end{align*}
where the second equality can be shown in the same way as that in the proof of Theorem~\ref{thm:invariance four CQs under K[G]-eq} for GCQ, while the third one holds since $\left( d \left( \iota_{\left( i \right)} \right)_0 \right)^*$ is surjective.
%
%
%
%
\end{proof}

\begin{proof}[Proof of Theorem~\ref{thm:invariance four CQs under reduction}]
\noindent
{\bf LICQ:}
We first observe that 
\begin{math}
\left( d \left( g_{I}, h \right)_{\iota_{\left( i \right)}, \left( k \right)} \right)_0
\end{math}
is the composition of the injection 
\begin{math}
\left( d \iota_{\left( i \right)} \right)_0
\end{math}
and the restriction 
\begin{math}
\left. d \left( g_{I}, h \right) \right|_{\Ker d \left( h_{i_1}, \ldots, ,h_{i_{r-\ell}} \right)_0}
\end{math}. Since 
\begin{math}
\Ker d \left( g_{I}, h \right)
\end{math}
is contained in
\begin{math}
\Ker d \left( h_{i_1}, \ldots, ,h_{i_{r-\ell}} \right)_0
\end{math}, we obtain 
\begin{align}
\rank \left( d \left( g_{I}, h \right)_{\iota_{\left( i \right)}, \left( k \right)} \right)_0 &= \left( n - r + \ell \right) - \dim \Ker \left. d \left( g_{I}, h \right)_0 \right|_{\Ker d \left( h_{i_1}, \ldots, ,h_{i_{r-\ell}} \right)_0} \\
&= \left( n - r + \ell \right) - \dim \Ker d \left( g_{I}, h \right)_0\\
&= \left( n - r + \ell \right) - \left( n - \rank d \left( g_{I}, h \right)_0 \right) \\
&= \rank d \left( g_{I}, h \right)_0 - r + \ell.
\end{align}
Since the germ 
\begin{math}
\left( g, h \right)
\end{math}
(resp.~\begin{math}
\left( g_{I}, h \right)_{\iota_{\left( i \right)}, \left( k \right)}
\end{math})
satisfies LICQ if and only if 
\begin{math}
\rank d \left( g_{I}, h \right)_0 = s + r
\end{math}
(resp.~\begin{math}
\rank \left( d \left( g_{I}, h \right)_{\iota_{\left( i \right)}, \left( k \right)} \right)_0 = s + \ell
\end{math}), this proves the claim. 

\noindent
{\bf MFCQ:}
The Jacobi matrix of 
\begin{math}
h
\end{math}
has corank 
\begin{math}
0
\end{math}
if and only if that of its reduction 
\begin{math}
h_{\iota_{\left( i \right)}}
\end{math}
has corank 
\begin{math}
0
\end{math} (Lemma~3.1 in \cite{HamadaHayanoTeramoto2025_1}). 
By Eq.~\eqref{eq:MFvector_reduction} in Lemma~\ref{lem:cone_reduction}, $(g,h)$ has an MF-vector if and only if its reduction has. 

\noindent
{\bf ACQ:}
The claim immediately follows from Lemma~\ref{lem:cone_reduction}.

\noindent
{\bf GCQ:}
Lemma~\ref{lem:polar_reduction} implies that if 
\begin{math}
\left( g, h \right)
\end{math}
satisfies GCQ, its reduction 
\begin{math}
\left( g, h \right)_{\iota_{\left( i \right)}, \left( k \right)}
\end{math}
satisfies GCQ. Therefore, in what follows, we show ``only if'' part. 

Suppose the reduction
\begin{math}
\left( g, h \right)_{\iota_{\left( i \right)}, \left( k \right)}
\end{math}
satisfies GCQ. 
By Lemma~\ref{lem:polar_reduction}, the following equality holds:
\[
L^+ \left( g, h \right)^\circ + \Ker \left( d \left( \iota_{\left( i \right)} \right)_0 \right)^*=C^+ \left( g, h \right)^\circ + \Ker \left( d \left( \iota_{\left( i \right)} \right)_0 \right)^*.
\]
Since 
\begin{math}
\Ker \left( d \left( \iota_{\left( i \right)} \right)_0 \right)^* = \left( \Im d \left( \iota_{\left( i \right)} \right)_0 \right)^\perp
\end{math}
and 
\begin{math}
\Im d \left( \iota_{\left( i \right)} \right)_0 = \Ker d \left( h_{i_1}, \ldots, h_{i_{r-\ell}} \right)_0 \supset L^+ \left( g, h \right) \supset C^+ \left( g, h \right)
\end{math}, 
\[
\Ker \left( d \left( \iota_{\left( i \right)} \right)_0 \right)^* \subset \left( \Im d \left( \iota_{\left( i \right)} \right)_0 \right)^\circ \subset L^+ \left( g, h \right)^\circ\subset C^+ \left( g, h \right)^\circ
\]
hold. 
We thus obtain:
\[
L^+ \left( g, h \right)^\circ=L^+ \left( g, h \right)^\circ + \Ker \left( d \left( \iota_{\left( i \right)} \right)_0 \right)^*=C^+ \left( g, h \right)^\circ + \Ker \left( d \left( \iota_{\left( i \right)} \right)_0 \right)^*=C^+ \left( g, h \right)^\circ.
\]
This completes the proof of ``only if'' part.
\end{proof}

\section{Verification of constraint qualifications for $\mathcal{K}[G]$-classes of generic constraints
} \label{sec:verification_of_CQ_for_each_class}

Recall that the purpose of this paper is to determine, for a \emph{generic} constraint map(-germ), when each of the four classical CQs (LICQ, MFCQ, ACQ, GCQ) hold. 
By the results of the previous section (Theorems~\ref{thm:invariance four CQs under K[G]-eq} and \ref{thm:invariance four CQs under reduction}), these four CQs are invariant under $\mathcal{K}[G]$-equivalence and reduction.
Combining this with the classification in \cite{HamadaHayanoTeramoto2025_1}---which lists the \emph{full reductions} of generic constraint map-germs up to $\mathcal{K}[G]$-equivalence---our task reduces to checking CQ-validity for the normal forms displayed in Tables~\ref{table:generic constraint q=0}--\ref{table:generic constraint q>0 r=1}.
We now carry out this verification.
Note that the implication in Eq.~\eqref{eq:strict_hierarchy} is known for these constraint qualifications. 

\subsection{LICQ and MFCQ}

First, among the normal forms, only the regular class (i.e.,~$g(x)=(x_1,\ldots, x_q)$ and $h(x)=(x_{q+1},\ldots, x_{q+r})$) satisfies LICQ.
In what follows, we discuss which classes satisfy MFCQ.
By definition, MFCQ is violated if the Jacobi matrix of an equality constraint has positive corank. 
In particular no classes in Tables~\ref{table:generic constraint q=0} and \ref{table:generic constraint q>0 r=1} satisfy MFCQ.
Therefore, we consider germs in Table~\ref{table:generic constraint r=0}.
\textnormal{MFCQ}-validity depends only on the $1$-jet, and each germ in Table~\ref{table:generic constraint r=0} has the following $1$-jet for some $l_1\in \left\{ 0, 1, \ldots, \lceil \frac{l}{2} \rceil \right\}$ and $l \in \left\{ 0, \dots, q-1 \right\}$:
\begin{equation}
j^1g(0) =\left( x_1, \ldots, x_{q-1}, \sum_{j=1}^{l_1} x_j - \sum_{j=l_1+1}^l x_j \right) \label{eq:q0_1jet}.
\end{equation}
If 
\begin{math}
0 < l_1
\end{math}
holds, the constraint map-germ $g$ satisfies \textnormal{MFCQ}. This can be shown as follows: Suppose that is the case. Set 
\begin{equation}
d = \left( \underbrace{-l, -1, \ldots, -1, 0}_q, \underbrace{0, \ldots, 0}_{n-q} \right) \in \mathbb{R}^n.
\end{equation}
Then, 
\begin{math}
d g_{i,0} \left( d \right) = -1 < 0
\end{math}
for 
\begin{math}
i \in \left\{ 2, \ldots, q-1 \right\}
\end{math}
and 
\begin{math}
d  g_{q,0} \left( d \right) \le \left( l-1 \right) - l = -1 < 0
\end{math}
hold. This proves the claim. 
If \begin{math}
l_1 = 0
\end{math} and $d\in \R^n$ satisfies $dg_{1,0}(d),\ldots,  dg_{l,0}(d) <0$, then $d_1,\ldots, d_l <0$ and thus $dg_{q,0}(d) = -\sum_{i=1}^{l}d_i >0$. 
Thus, $g$ does not satisfy \textnormal{MFCQ}. 
In summary, we obtain the following theorem. 

\begin{theorem} \label{thm:MFCQ}

No germ in Tables~\ref{table:generic constraint q=0} and \ref{table:generic constraint q>0 r=1} satisfies MFCQ. 
A germ in Table~\ref{table:generic constraint r=0} satisfies MFCQ if and only if $l_1>0$, where $l_1$ is the parameter in the normal form in the caption of the table (or Eq.~\eqref{eq:q0_1jet}).

\end{theorem}

\begin{remark}

For a general constraint map-germ $(g,h)$ with $\corank dh_0=0$, suppose that the $1$-jet of a full reduction of $(g,h)$ is $\mathcal{K}[G]^1$-equivalent to that in Eq.~\eqref{eq:q0_1jet} for some $l,l_1$.
Then, $(g,h)$ satisfies MFCQ if and only if $l_1>0$.

\end{remark}

\subsection{\textnormal{ACQ} and \textnormal{GCQ}}

Next, we compute tangent cones of feasible-set germs in Table~\ref{table:generic constraint q=0}, Table~\ref{table:generic constraint r=0} (violating \textnormal{MFCQ}), and Table~\ref{table:generic constraint q>0 r=1} and confirms if ACQ and GCQ hold for each class. 
Note that we assume
\begin{math}
n
\end{math}
is sufficiently large so that each normal form has a quadratic part (cf.~Theorem~\ref{thm:5.1}). 

In order to determine tangent cones, we need the following lemmas.

\begin{lemma}\label{lem:tangent cone for r=1}

Let $g = (x_1,\ldots, x_q)$ and $h(x) = Q(x) + R(x)$, where $Q(x)$ is a quadratic polynomial and $R(x)$ is a polynomial consisting of terms with degree larger than $2$. 

\begin{enumerate}

\item 
$C^+(g,h)$ is contained in $\overline{C}(q,Q) := \{d\in \R^n ~|~ d_1,\ldots, d_q\leq 0, Q(d) = 0\}$.

\item 
Let $R_r$ be the degree $r$ homogeneous part of $R$. 
The element $d\in \overline{C}(q,Q)$ is contained in $C^+(g,h)$ if either $R_r(d)=0$ for any $r \geq 3$, or there exists $v\in \R^n$ satisfying the following conditions:

\begin{itemize}

\item 
$v_j\leq 0$ for any $j\in \{1,\ldots, q\}$ with $d_j=0$, 

\item 
if the inner product $v \cdot \nabla Q(d)$ is not $0$, its sign is opposite to that of $R_{r_0}(d)$, where $r_0 = \min \{r\geq 0~|~R_r(d)\neq 0\}$,

\item 
if $v \cdot \nabla Q(d)=0$, the product ${}^t v \mathrm{Hess}(Q) v$ is not $0$ and its sign is opposite to that of $R_{r_0}(d)$, where $\mathrm{Hess}(Q) = \left(\frac{\Pa^2 Q}{\Pa x_i\Pa x_j}\right)_{i,j}$ is the Hessian matrix of $Q$.

\end{itemize}

\end{enumerate}

\end{lemma}

\noindent
We denote by $\underline{C}(q,Q,R)$ the set of $d\in \overline{C}(q,Q)$ satisfying the condition in the second statement.

\begin{proof}
Take $d\in C^+(g,h)$. 
There exist
\begin{math}
\left\{ x^{\left( m \right)} \right\}_{m \in \mathbb{N}} \subset M \left(g, h \right)
\end{math}
and
\begin{math}
\left\{ t_m \right\}_{m \in \mathbb{N}} \subset \mathbb{R}_{>0}
\end{math}
such that 
\begin{math}
\lim_{m \rightarrow \infty} x^{\left( m \right)} = 0
\end{math}
and 
\begin{math}
d = \lim_{m \rightarrow \infty} t_m x^{\left( m \right)}
\end{math}
hold. 
Since the $j$-th component $x_j^{(m)}$ of $x^{(m)}$ is less than or equal to $0$ for $j=1,\ldots, q$, so is the limit $d_j = \lim_{m\to\infty} t_m x_j^{(m)}$. 
Let $e\geq 3$ be the lowest order of the term of $R$.
Since 
$x^{\left( m \right)}$ is contained in $M \left( g,h \right)$,
$h \left( x^{\left( m \right)} \right) = Q \left( x^{(m)} \right)+R(x^{(m)})$ is equal to $0$. 
We thus obtain:
\[
Q(t_m x^{(m)}) + t_m^2 R(x^{(m)}) =0. 
\]
By taking the limit 
\begin{math}
m \rightarrow \infty
\end{math}, we obtain 
\begin{math}
Q \left( d \right) = 0
\end{math}
since 
\begin{equation}
\left| t_m^2 R\left(x^{(m)}\right)\right| = O \left( \left| t_m x^{\left( m \right)} \right|^2 \right) \cdot O \left( \left| x^{\left( m \right)} \right|^{e-2} \right)
\end{equation}
holds and 
\begin{math}
\lim_{m \rightarrow \infty} \left| t_m x^{\left( m \right)} \right|^2 = \left| d \right|^2
\end{math}
whereas
\begin{math}
\lim_{m \rightarrow \infty} \left| x^{\left( m \right)} \right|^{e-2} = 0
\end{math}
holds.
Thus, $d$ is contained in $\overline{C}(q,Q)$. 

Take $d\in \overline{C}(q,Q)$. 
Suppose that $R_r(d)=0$ for any $r\geq 3$. 
The following holds for any $m\gg 0$:
\[
h(m^{-1}d) = Q(m^{-1}d) + \sum_{r\geq 3}R_r(m^{-1}d) = m^{-2}Q(d) + \sum_{r\geq 3}m^{-r}R_r(d) =0. 
\]
Thus, $d$ is contained in $C^+(g,h)$ since $\lim_{m\to\infty} m^{-1}d=0$ and $d = \lim_{m\to\infty} m \cdot m^{-1}d$. 
In what follows, we assume that $R_r(d)\neq 0$ for some $r\geq 3$. 
We take a vector $v\in \R^n$ satisfying the conditions in the second statement for $d$. 
By the Taylor's theorem, the following equalities hold:
\begin{align*}
&h\left(m^{-1}d + m^{-5/4}v\right) \\
=&h \left(m^{-1}d\right) + \left( v\cdot \nabla h\left(m^{-1}d\right)\right)m^{-5/4} \\
&+ \frac{1}{2}\left({}^tv \mathrm{Hess}(h)(m^{-1}d)v\right)m^{-5/2}+O(m^{-15/4}) & (m\to \infty)\\
=&h \left(m^{-1}d\right) + \left( v\cdot \nabla Q\left(d\right)\right)m^{-9/4} \\
&+ \frac{1}{2}\left({}^tv\mathrm{Hess}(Q)v\right)m^{-5/2}+O(m^{-3}) & (m\to \infty).
\end{align*}
If $v\cdot \nabla Q(d)$ is not $0$, its sign is opposite to that of $R_{r_0}(d)$, and thus the sign of $h\left(m^{-1}d + m^{-5/4}v\right)$ is opposite to that of $h(m^{-1}d) = m^{-r_0}R_{r_0}(d) + O(m^{-r_0-1})$ for $m\gg 0$. 
By the intermediate value theorem, there exists $\theta_m\in (0,1)$ with $h\left(m^{-1}d + m^{-5/4}\theta_mv\right)=0$. 
Let $x^{(m)} = m^{-1}d + m^{-5/4}\theta_mv$. 
Since $v_j \leq 0$ for $j\in \{1,\ldots, q\}$ with $d_j=0$, $x^{(m)}$ is contained in $M(g,h)$ for $m\gg 0$. 
It is easy to check that $\lim_{m\to \infty} x^{(m)} = 0$ and $\lim_{m\to \infty} m x^{(m)} = d$.
Hence, $d$ is contained in $C^+(g,h)$. 
If $v\cdot \nabla Q (d)=0$, ${}^t v\mathrm{Hess}(Q)v$ is not zero and its sign is opposite to that of $R_{r_0}(d)$. 
We can thus deduce that $d$ is contained in $C^+(g,h)$ in the same way. 
\end{proof}

In what follows, for $x = (x_1,\ldots, x_n)\in \R^n$, we denote $x' = (x_{l+1},\ldots, x_n)\in \R^{n-l+1}$. 

\begin{lemma}\label{lem:tangent cone for r=0}

Let $g(x)=(x_1,\ldots, x_{q-1},\widetilde{g}(x)=-\sum_{j=1}^{l}x_j+Q(x')+R(x'))$, where $Q(x')$ is a quadratic polynomial and $R(x')$ is a polynomial consisting of terms with degree larger than $2$. 

\begin{enumerate}

\item 
$C^+(g)$ is contained in $\overline{D}(l,Q)$, where 
\[
\overline{D}(l,Q)= \{d\in \R^n~|~d_1=\cdots =d_l=0, d_{l+1}\leq 0 ,\ldots, d_{q-1}\leq 0, Q(d')\leq 0\}.
\]

\item 
The element $d\in \overline{D}(l,Q)$ is contained in $C^+(g)$ if one of the following conditions holds:
\begin{itemize}

\item 
$Q(d')<0$,

\item 
$R_r(d')=0$ for any $r\geq 3$, where $R_r$ is the degree $r$ homogeneous part of $R$,

\item 
$R_{r_0}(d') <0$, where $r_0 = \min \{r\geq 0~|~R_r(d')\neq 0\}$,

\item 
there exists $v=(0,\ldots,0,v_{l+1},\ldots, v_n)\in \R^{n}$ satisfying the following conditions:

\begin{enumerate}[(a)]

\item 
$v_j\leq 0$ for any $j\in \{l+1,\ldots,q-1\}$ with $d_j=0$, 

\item 
if the inner product $v'\cdot \nabla Q(d')$ is not $0$, it is less than $0$, 

\item 
if $v'\cdot \nabla Q(d')=0$, the product ${}^tv'\mathrm{Hess}(Q)v'$ is less than $0$. 

\end{enumerate}

\end{itemize}  

\end{enumerate}

\end{lemma}

\begin{proof}
Let $d\in C^+(g)$. 
Since $C^+(g)$ is contained in $L^+(g)$, the (in)equalities $d_1=\cdots =d_l=0$ and $d_{l+1},\ldots, d_{q-1}\leq 0$ hold. 
Take \begin{math}
\left\{ x^{\left( m \right)} \right\}_{m \in \mathbb{N}} \subset M \left( g \right)
\end{math}
and 
\begin{math}
\left\{ t_m \right\}_{m \in \mathbb{N}} \subset \mathbb{R}_{> 0}
\end{math}
so that 
\begin{math}
\lim_{m \rightarrow \infty} x^{\left( m \right)} = 0
\end{math}
and 
\begin{math}
d = \lim_{m \rightarrow \infty} t_m x^{\left( m \right)}
\end{math}
hold. 
Since 
\begin{math}
x^{\left( m \right)} \in M \left( g \right)
\end{math}, 
the following inequality holds:
\[
0\geq \widetilde{g}\left(x^{(m)}\right)\geq Q\left( x^{\left( m \right)}_{l+1}, \ldots, x^{\left( m \right)}_n \right) +R\left( x^{\left( m \right)}_{l+1}, \ldots, x^{\left( m \right)}_n \right).
\]
By multiplying 
\begin{math}
t_m^2
\end{math}
to this inequality and taking the limit 
\begin{math}
m \rightarrow \infty
\end{math}, we obtain 
\begin{math}
Q \left(d' \right) \le 0
\end{math}. 
Thus, $d$ is contained in $\overline{D}(l,Q)$.

Let $d\in \overline{D}(l,Q)$. 
If $Q(d')<0$, $m^{-1}d$ is contained in $M(g)$ for $m\gg 0$ since 
\[
\widetilde{g}(m^{-1}d) = Q(m^{-1}d') + R(m^{-1}d') = m^{-2}Q(d')+O(m^{-3})<0~(m\to \infty). 
\]
Since $\lim_{m\to \infty}m^{-1}d =0$ and $\lim_{m\to\infty}m\cdot m^{-1}d=d$, $d$ is contained in $C^+(g)$. 
We can also deduce that $d\in C^+(g)$ if either $R_{r_0}(d')<0$ or $R_r(d')=0$ for any $r\geq 3$ in the same manner. 
In what follows, we assume that $Q(d')=0$ and there exists $v\in \R^{n-l}$ satisfying the conditions in Lemma~\ref{lem:tangent cone for r=0}. 
We will show that $m^{-1}d+m^{-5/4}v$ is contained in $M(g)$ for $m\gg 0$. 
For $j\leq l$, the $j$-th component of $m^{-1}d+m^{-5/4}v$ (which is $g_j(m^{-1}d+m^{-5/4}v)$) is equal to $0$ since $d_j=v_j=0$.
For $l+1\leq j\leq q-1$, the $j$-th component of $m^{-1}d+m^{-5/4}v$ (which is $g_j(m^{-1}d+m^{-5/4}v)$) is less than $0$ for $m\gg0$ if $d_j<0$. 
If $d_j=0$, $v_j$ is less than or equal to $0$ by the assumption on $v$, and thus the $j$-th component of $m^{-1}d+m^{-5/4}v$ is also less than or equal to $0$ for $m\gg 0$.
We can obtain the following equality in the same way as in the proof of Lemma~\ref{lem:tangent cone for r=1}:
\[
\widetilde{g}\left(m^{-1}d + m^{-5/4}v\right) =(v'\cdot \nabla Q(d'))m^{-9/4} + \frac{1}{2}\left({}^tv'\mathrm{Hess}(Q)v'\right)m^{-5/2} +O(m^{-3}) ~(m\to \infty). 
\]
If $v'\cdot \nabla Q(d')$ is not $0$, it is less than $0$ by the assumption, and thus $\widetilde{g}\left(m^{-1}d + m^{-5/4}v\right)$ is also less than $0$ for $m\gg0$ . 
If $v'\cdot \nabla Q(d')=0$, ${}^tv'\mathrm{Hess}(Q)v'$ is less than $0$ by the assumption, and thus $\widetilde{g}\left(m^{-1}d + m^{-5/4}v\right)$ is also less than $0$ for $m\gg0$ . 
We can eventually deduce that $m^{-1}d+m^{-5/4}v$ is contained in $M(g)$ for $m\gg 0$. 
Hence $d$ is contained in $C^+(g)$ since $\lim_{m\to \infty}m^{-1}d+m^{-5/4}v=0$ and $\lim_{m\to\infty}m\left(m^{-1}d+m^{-5/4}v\right)=d$. 
\end{proof}

\begin{corollary}\label{cor:tangent cone for r=0}

Let $g,\tilde{g},Q,R$ be the same as those in Lemma~\ref{lem:tangent cone for r=0}. 
Suppose that $Q(x')$ is equal to $P(x') + \sum_{j=l+s+1}^{n}\epsilon_jx_j^2$ for some quadratic polynomial $P$ with variables $x_{l+1},\ldots, x_{l+s}$ and $\epsilon_j\in \{1,-1\}$, and $R(x')$ is a homogeneous polynomial with variables $x_{l+1},\ldots, x_{l+s}$. 
Then, $\overline{D}(l,Q)\setminus \{(0,\ldots, 0,d_{l+1},\ldots, d_{l+s},0,\ldots,0)~|~R(d')>0\}$ is contained in $C^+(g)$, and $C^+(g)=\overline{D}(l,Q)$ if $\epsilon_j=-1$ for some $j\geq l+s+1$. 

\end{corollary}

\begin{proof}
The tangent cone $C^+(g)$ is contained in $\overline{D}(l,Q)$ by Lemma~\ref{lem:tangent cone for r=0}. 
Let $d\in \overline{D}(l,Q)$. 
Since $R(x')=R_{r_0}(x')$ by the assumption on $R$, $d$ is contained in $C^+(g)$ if $R(d')\leq 0$ by Lemma~\ref{lem:tangent cone for r=0}. 
If $d_j\neq 0$ for some $j\geq l+s+1$, the vector $v= -\epsilon_jd_je_j$ satisfies the conditions in Lemma~\ref{lem:tangent cone for r=0}. 
Indeed, $v_1=\cdots =v_{q-1}=0$ (in particular $v$ satisfies the condition (a) in Lemma~\ref{lem:tangent cone for r=0}), $v'\cdot \nabla Q(d') =-2 \epsilon_j^2 d_j^2 <0$.
We thus obtain $\overline{D}(l,Q)\setminus \{(0,\ldots, 0,d_{l+1},\ldots, d_{l+s},0,\ldots,0)~|~R(d')>0\}\subset C^+(g)$. 
If $d_{l+s+1}=\cdots = d_n=0$ and $\epsilon_j=-1$ for some $j\geq l+s+1$, the vector $v=e_j$ satisfies the conditions in Lemma~\ref{lem:tangent cone for r=0}.
Indeed, $v_1=\cdots =v_{q-1}=0$, $v'\cdot \nabla Q(d')=v'\cdot \nabla P(d')=0$ since $P(x')$ is a polynomial with variables $x_{l+1},\ldots, x_{l+s}$, and ${}^tv'\mathrm{Hess}(Q)v' = 2\epsilon_j=-2 <0$. 
Hence $C^+(g)$ is equal to $\overline{D}(l,Q)$ if $\epsilon_j=-1$ for some $j\geq l+s+1$.
\end{proof}

\subsubsection{ACQ and GCQ in Table~\ref{table:generic constraint q=0}}

In this case, the linearized cones of normal forms in Table~\ref{table:generic constraint q=0} are 
\begin{math}
\mathbb{R}^n
\end{math}
since the gradients of the normal forms are zero. 

\begin{proposition}\label{prop:ACQ for r=0}

ACQ does not hold for any germ in Table~\ref{table:generic constraint q=0}.

\end{proposition}

\begin{proof}
Let $h$ be any germ in Table~\ref{table:generic constraint q=0}.
By Lemma~\ref{lem:tangent cone for r=1}, the tangent cone $C^+(h)$ is contained in $\overline{C}(0,Q) = \{d\in \R^n~|~Q(d)=0\}$, where $Q$ is the quadratic part of $h$. 
One can easily check that $Q(e_n)\neq 0$, in particular $C^+(h)\subsetneq \R^n=L^+(h)$ for any $h$ in Table~\ref{table:generic constraint q=0}. 
\end{proof}

\subsubsection*{The germ of type $(1,k)$ in Table~\ref{table:generic constraint q=0}}
Let $k\geq 2$, $h = x_1^k+\sum_{j=2}^{n}\epsilon_jx_j^2$, where $\epsilon_j\in \{1,-1\}$. For 
\begin{math}
k \ge 3
\end{math}, let 
\begin{math}
Q = \sum_{j=2}^{n}\epsilon_jx_j^2
\end{math}
and 
\begin{math}
R = x_1^k
\end{math}.
\begin{proposition} \label{prop:table1_(1,k)_k2_cone}
If
\begin{math}
k = 2
\end{math}, the tangent cone 
\begin{math}
C^+ \left( h \right)
\end{math}
is equal to 
\begin{math}
 \overline{C}(0,Q)
\end{math}.
\end{proposition}
\begin{proof}
We can put 
\begin{math}
R \left( x \right) = 0
\end{math}
in this case, and thus the proposition holds by Lemma~\ref{lem:tangent cone for r=1}.
\end{proof}

\begin{proposition} \label{prop:table1_(1,k)_k3g_cone}
Assume $k\geq 3$. 
The tangent cone 
\begin{math}
C^+ \left( h \right)
\end{math}
is equal to 
\begin{math}
\left\{ 0 \right\}
\end{math}
if 
\begin{math}
k
\end{math}
is even and all the 
\begin{math}
\epsilon_j
\end{math}s are 
\begin{math}
1
\end{math}. If 
\begin{math}
k
\end{math}
is odd and all the 
\begin{math}
\epsilon_j
\end{math}s have the same sign
\begin{math}
\delta
\end{math}, 
\begin{math}
C^+ \left( h \right)
\end{math}
is equal to 
\begin{math}
\overline{C} \left( 0, Q \right) \setminus \left\{ d \in \mathbb{R}^n \middle| \delta d_1> 0, d_2 = \cdots = d_n = 0 \right\}
\end{math}. 
\begin{math}
C^+ \left( h \right)
\end{math}
is equal to 
\begin{math}
\overline{C} \left( 0, Q \right)
\end{math}
in all the other cases.
\end{proposition}
\begin{proof}
In the first case, it is easy to see that 
\begin{math}
M \left( g, h \right)
\end{math}
is equal to 
\begin{math}
\left\{ 0 \right\}
\end{math}, and thus its tangent cone is also 
\begin{math}
\left\{ 0 \right\}
\end{math}.

In all the cases, 
\begin{math}
C^+ \left( h \right)
\end{math}
is contained in 
\begin{math}
\overline{C} \left( 0, Q \right)
\end{math}
by Lemma~\ref{lem:tangent cone for r=1}. The element 
\begin{math}
d \in \overline{C} \left( 0, Q \right)
\end{math}
is contained in 
\begin{math}
C^+ \left( h \right)
\end{math}
if 
\begin{math}
d_1 = 0
\end{math}
holds since 
\begin{math}
R \left( d \right) = d^k_1 = 0
\end{math}
in this case. Therefore, we consider the case 
\begin{math}
d_1 \neq 0
\end{math}. If 
\begin{math}
\nabla Q \left( d \right) \neq 0
\end{math}, 
we can choose 
\begin{math}
v = -\mathrm{sign} \left( R \left( d \right) \right) \cdot \nabla Q \left( d \right)
\end{math}
so that the sign of 
\begin{math}
v \cdot \nabla Q \left( d \right)
\end{math}
is opposite to that of 
\begin{math}
R \left( d \right)
\end{math}
and thus 
\begin{math}
d
\end{math}
is contained in 
\begin{math}
C^+ \left( h \right)
\end{math}
by Lemma~\ref{lem:tangent cone for r=1}. Therefore, we obtain 
\begin{equation}
\overline{C} \left( 0, Q \right) \setminus \left\{ d \in \mathbb{R}^n \middle| d_1\neq 0, d_2 = \cdots = d_n = 0 \right\} \subset C^+ \left( h \right).
\end{equation}

Suppose
\begin{math}
k
\end{math}
is odd and all the 
\begin{math}
\epsilon_j
\end{math}s have the same sign 
\begin{math}
\delta
\end{math}. Take any 
\begin{math}
d \in \mathbb{R}^n
\end{math}
such that 
\begin{math}
d_1 \neq 0, d_2 = \cdots = d_n = 0
\end{math}. Then,  
\begin{math}
\nabla Q \left( d \right) = 0
\end{math}
holds. In what follows, we will show that
\begin{math}
d
\end{math}
is contained in 
\begin{math}
C^+ \left( h \right)
\end{math}
if and only if
\begin{math}
\delta d_1 < 0
\end{math}
holds. First, suppose 
\begin{math}
\delta d_1 < 0
\end{math}. Then,  
the sign of \begin{math}
\;^t e_2 \mathrm{Hess} \left( Q \right) e_2 = 2 \delta
\end{math}
is opposite to that of 
\begin{math}
R \left( d \right)=d_1^k
\end{math}. 
Therefore, 
\begin{math}
d
\end{math}
is contained in 
\begin{math}
C^+ \left( h \right)
\end{math}
by Lemma~\ref{lem:tangent cone for r=1}. Conversely, for any 
\begin{math}
d \in C^+ \left( h \right)
\end{math}, there exist sequences 
\begin{math}
\left\{ t_m \right\} \subset \mathbb{R}_{> 0}
\end{math}
and 
\begin{math}
\left\{ x^{\left( m \right)} \right\} \subset M \left( h \right)
\end{math}
such that 
\begin{math}
d = \lim_{m \rightarrow \infty} t_m x^{\left( m \right)}
\end{math}
holds. Since 
\begin{math}
\delta t_m^k\left( x^{\left( m \right)}_1 \right)^k = - t_m^k\sum_{j=2}^n \left( x^{\left( m \right)} \right)^2 \le 0
\end{math}
holds, by taking 
\begin{math}
m \rightarrow \infty
\end{math}
in the both sides of the inequality implies that 
\begin{math}
\delta d_1^k \le 0
\end{math}. 
Therefore, we obtain
\begin{math}
C^+ \left( h \right) = \overline{C} \left( 0, Q \right) \setminus \left\{ d \in \mathbb{R}^n \middle| \delta d_1> 0, d_2 = \cdots = d_n = 0 \right\}
\end{math}
in case if 
\begin{math}
k
\end{math}
is odd and all the signs of 
\begin{math}
\epsilon_j
\end{math}
is 
\begin{math}
\delta
\end{math}.

Suppose 
\begin{math}
k
\end{math}
is even and all the 
\begin{math}
\epsilon_j
\end{math}s are 
\begin{math}
-1
\end{math}. Take any 
\begin{math}
d \in \mathbb{R}^n
\end{math}
such that 
\begin{math}
d_1 \neq 0, d_2 = \cdots = d_n = 0
\end{math}. Then,  
\begin{math}
\nabla Q \left( d \right) = 0
\end{math}
holds. In that case,
\begin{math}
d
\end{math}
is contained in 
\begin{math}
C^+ \left( h \right)
\end{math}
since the sign of 
\begin{math}
\;^t e_2 \mathrm{Hess} \left( Q \right) e_2 = - 2 
\end{math}
is opposite to that of 
\begin{math}
R \left( d \right) = d_1^k
\end{math}. 
Therefore, 
\begin{math}
C^+ \left( h \right) = \overline{C} \left( 0, Q \right)
\end{math}
holds in this case. 

Suppose 
\begin{math}
\left\{ \epsilon_2, \ldots, \epsilon_n \right\} = \left\{ 1, -1 \right\}
\end{math}. In this case, we can suppose 
\begin{math}
\epsilon_2 = 1
\end{math}
and 
\begin{math}
\epsilon_3 = -1
\end{math}
without loss of generality. Take any 
\begin{math}
d \in \mathbb{R}^n
\end{math}
such that 
\begin{math}
d_1 \neq 0, d_2 = \cdots = d_n = 0
\end{math}. Then,  
\begin{math}
\nabla Q \left( d \right) = 0
\end{math}
holds. In that case,
\begin{math}
d
\end{math}
is contained in 
\begin{math}
C^+ \left( h \right)
\end{math}
by Lemma~\ref{lem:tangent cone for r=1} since 
\begin{math}
\;^t e_2 \mathrm{Hess} \left( Q \right) e_2 = 2
\end{math}
and 
\begin{math}
\;^t e_3 \mathrm{Hess} \left( Q \right) e_3 = -2
\end{math}
hold and thus
\begin{math}
v
\end{math}
can be chosen to 
\begin{math}
e_2
\end{math}
or 
\begin{math}
e_3
\end{math}
so that 
\begin{math}
\;^t v \mathrm{Hess} \left( Q \right) v
\end{math}
has the opposite sign to 
\begin{math}
R \left( d \right)
\end{math}. Therefore, 
\begin{math}
C^+ \left( h \right) = \overline{C} \left( 0, Q \right)
\end{math}
holds in this case. 
\end{proof}

\begin{proposition}

GCQ holds for $h$ if and only if
\begin{math}
k = 2
\end{math}
and one of 
\begin{math}
\epsilon_j
\end{math}s is
\begin{math}
-1
\end{math}
or 
\begin{math}
k \ge 3
\end{math}
and 
\begin{math}
\left\{ \epsilon_2, \ldots, \epsilon_n \right\} = \left\{ 1, -1 \right\}
\end{math}.

\end{proposition}

\begin{proof}
It is easy to check that 
\begin{math}
L^+ \left( h \right)^\circ
\end{math}
is equal to
\begin{math}
\left\{0 \right\}
\end{math}. 

\noindent
\textbf{Proof of ``if'' part:} First, suppose 
\begin{math}
k = 2
\end{math}
and
\begin{math}
\epsilon_2 = -1
\end{math}. 
In this case, $C^+(h)$ is equal to $\overline{C}(0,Q)$ by Proposition~\ref{prop:table1_(1,k)_k2_cone}.
Take any 
\begin{math}
w \in C^+ \left( h \right)^\circ = \overline{C} \left( 0, Q \right)^\circ
\end{math}. 
For 
\begin{math}
j \in \left\{ 1, \ldots, n \right\}
\end{math}, either 
\begin{math}
\pm e_1 \pm e_j
\end{math}
or 
\begin{math}
\pm e_2 \pm e_j
\end{math}
is contained in 
\begin{math}
\overline{C} \left( 0, Q \right)
\end{math}. Therefore, 
\begin{math}
w \cdot \left( \pm e_1 \pm e_j \right) = \pm w_1 \pm w_j \le 0
\end{math}
or 
\begin{math}
w \cdot \left( \pm e_2 \pm e_j \right) = \pm w_2 \pm w_j \le 0
\end{math}
hold for 
\begin{math}
j \in \left\{ 1, \ldots, n \right\}
\end{math}. This implies that 
\begin{math}
w_j = 0
\end{math}
for all 
\begin{math}
j \in \left\{ 1, \ldots, n \right\}
\end{math}. This proves
\begin{math}
w \in L^+ \left( h \right)^\circ
\end{math}
and thus 
\begin{math}
C^+ \left( h \right)^\circ \subset L^+ \left( h \right)^\circ
\end{math}. Since 
\begin{math}
C^+ \left( h \right)^\circ \supset L^+ \left( h \right)^\circ
\end{math}
always holds, this proves that GCQ holds in this case. 

Second, suppose 
\begin{math}
k \ge 3
\end{math}
and 
\begin{math}
\left\{ \epsilon_2, \ldots, \epsilon_n \right\} = \left\{ 1, -1 \right\}
\end{math}. 
In this case, $C^+(h)$ is equal to $\overline{C}(0,Q)$ by Proposition~\ref{prop:table1_(1,k)_k3g_cone}.
Without loss of generality, we can assume 
\begin{math}
\epsilon_2 = 1
\end{math}
and 
\begin{math}
\epsilon_3 = -1
\end{math}. 
Take any 
\begin{math}
w \in C^+ \left( h \right)^\circ = \overline{C} \left( 0, Q \right)^\circ
\end{math}. 
Since 
\begin{math}
\pm e_1 \in \overline{C} \left( 0, Q \right)
\end{math}
holds, 
\begin{math}
w_1 = 0
\end{math}
holds. For 
\begin{math}
j \in \left\{ 2, \ldots, n \right\}
\end{math}, either 
\begin{math}
\pm e_2 \pm e_j
\end{math}
or 
\begin{math}
\pm e_3 \pm e_j
\end{math}
is contained in 
\begin{math}
\overline{C} \left( 0, Q \right)
\end{math}. Therefore, 
\begin{math}
w \cdot \left( \pm e_2 \pm e_j \right) = \pm w_2 \pm w_j \le 0
\end{math}
or 
\begin{math}
w \cdot \left( \pm e_3 \pm e_j \right) = \pm w_3 \pm w_j \le 0
\end{math}
hold for 
\begin{math}
j \in \left\{ 2, \ldots, n \right\}
\end{math}. This implies that 
\begin{math}
w_j = 0
\end{math}
for all 
\begin{math}
j \in \left\{ 2, \ldots, n \right\}
\end{math}. This proves
\begin{math}
w \in L^+ \left( h \right)^\circ
\end{math}
and thus 
\begin{math}
C^+ \left( h \right)^\circ \subset L^+ \left( h \right)^\circ
\end{math}. Since 
\begin{math}
C^+ \left( h \right)^\circ \supset L^+ \left( h \right)^\circ
\end{math}
always holds, this proves that GCQ holds in this case. 

\noindent
\textbf{Proof of ``only if'' part:} 
We prove this by proving the contraposition. 
First, suppose 
\begin{math}
k = 2
\end{math}
and all of 
\begin{math}
\epsilon_j
\end{math}s are 
\begin{math}
1
\end{math}. In this case, 
\begin{math}
M \left( h \right) = \left\{ 0 \right\}
\end{math}
and thus 
\begin{math}
C^+ \left( h \right) = \left\{ 0 \right\}
\end{math}
holds. Therefore, 
\begin{math}
C^+ \left( h \right)^\circ = \mathbb{R}^n \neq L^+ \left( h \right)^\circ
\end{math}
and thus GCQ is violated. 

Second, suppose 
\begin{math}
k \ge 3
\end{math}
and all of 
\begin{math}
\epsilon_j
\end{math}s have the same sign 
\begin{math}
\delta
\end{math}. In this case, 
\begin{math}
\overline{C} \left( 0, Q \right) = \left\{ d \in \mathbb{R}^n \middle| d_2 = \cdots = d_n = 0 \right\}
\end{math}
holds and thus 
\begin{math}
\{0\}=L^+ \left( h \right)^\circ\subsetneq \overline{C} \left( 0, Q \right)^\circ = \left\{ w \in \mathbb{R}^n \middle| w_1 = 0 \right\}
\end{math}. 
By Lemma~\ref{lem:tangent cone for r=1}, $C^+ \left( h \right)$ is contained in $\overline{C} \left( 0, Q \right)$, and thus $\overline{C}(0,Q)^\circ \subset C^+(h)^\circ$. 
Therefore, GCQ is violated. 
\end{proof}

\subsubsection*{The germ of type $(2)$ in Table~\ref{table:generic constraint q=0}}
Let $h = x_1^3+\epsilon_2 x_1 x_2^2 + x_3^2 + \sum_{j=4}^{n}\epsilon_jx_j^2$, where $\epsilon_j\in \{1,-1\}$, 
\begin{math}
Q = x_3^2 + \sum_{j=4}^{n}\epsilon_jx_j^2
\end{math}
and 
\begin{math}
R = x_1^3+\epsilon_2 x_1 x_2^2
\end{math}.
\begin{proposition} \label{prop:table1_(2)_cone}
The tangent cone 
\begin{math}
C^+ \left( h \right)
\end{math}
is equal to 
\begin{equation}
\overline{C} \left( 0, Q \right) \setminus \left\{ d\in \R^n \middle| d_1^3 + \epsilon_2 d_1 d_2^2 > 0, d_3 = \cdots = d_n = 0 \right\}
\end{equation}
if all of 
\begin{math}
\epsilon_4, \ldots, \epsilon_n
\end{math}
are
\begin{math}
1
\end{math}, and 
\begin{math}
C^+ \left( h \right) = \overline{C} \left( 0, Q \right)
\end{math}
otherwise.
\end{proposition}
\begin{proof}
In all the cases, the element 
\begin{math}
d \in \overline{C} \left( 0, Q \right)
\end{math}
is contained in 
\begin{math}
C^+ \left( h \right)
\end{math}
if 
\begin{math}
d_1^3 + \epsilon_2 d_1 d_2^2 = 0
\end{math}
holds since 
\begin{math}
R \left( d \right) = 0
\end{math}
in this case. Therefore, we consider the case 
\begin{math}
d_1^3 + \epsilon_2 d_1 d_2^2 \neq 0
\end{math}. If 
\begin{math}
\nabla Q \left( d \right) \neq 0
\end{math}, 
we can choose 
\begin{math}
v = -\mathrm{sign} \left( R \left( d \right) \right) \cdot \nabla Q \left( d \right)
\end{math}
so that the sign of 
\begin{math}
v \cdot \nabla Q \left( d \right)
\end{math}
is opposite to that of 
\begin{math}
R \left( d \right)
\end{math}
and thus 
\begin{math}
d
\end{math}
is contained in 
\begin{math}
C^+ \left( h \right)
\end{math}
by Lemma~\ref{lem:tangent cone for r=1}. Therefore, we obtain 
\begin{equation}
\overline{C} \left( 0, Q \right) \setminus \left\{ d \in \mathbb{R}^n \middle| d_1^3 + \epsilon_2 d_1 d_2^2 \neq 0, d_3 = \cdots = d_n = 0 \right\} \subset C^+ \left( h \right).
\end{equation}

Suppose all of 
\begin{math}
\epsilon_4, \ldots, \epsilon_n
\end{math}
are
\begin{math}
1
\end{math}. Take any 
\begin{math}
d \in \mathbb{R}^n
\end{math}
such that 
\begin{math}
d_1^3 + \epsilon_2 d_1 d_2^2 \neq 0, d_3 = \cdots = d_n = 0
\end{math}. Then,  
\begin{math}
\nabla Q \left( d \right) = 0
\end{math}
holds. In that case,
\begin{math}
d
\end{math}
is contained in 
\begin{math}
C^+ \left( h \right)
\end{math}
if and only if
\begin{math}
d_1^3 + \epsilon_2 d_1 d_2^2 < 0
\end{math}
holds. First, suppose 
\begin{math}
d_1^3 + \epsilon_2 d_1 d_2^2 < 0
\end{math}
holds. Then, the sign of 
\begin{math}
\;^t e_3 \mathrm{Hess} \left( Q \right) e_3 = 2
\end{math}
is opposite to that of 
\begin{math}
R \left( d \right)
\end{math}. 
Therefore, 
\begin{math}
d
\end{math}
is contained in 
\begin{math}
C^+ \left( h \right) 
\end{math}
by Lemma~\ref{lem:tangent cone for r=1}. Conversely, for any 
\begin{math}
d \in C^+ \left( h \right)
\end{math}, there exist sequences 
\begin{math}
\left\{ t_m \right\} \subset \mathbb{R}_{> 0}
\end{math}
and 
\begin{math}
\left\{ x^{\left( m \right)} \right\} \subset M \left( h \right)
\end{math}
such that 
\begin{math}
d = \lim_{m \rightarrow \infty} t_m x^{\left( m \right)}
\end{math}
holds. Since 
\begin{equation}
t_m^3\left( x^{\left( m \right)}_1 \right)^3 + \epsilon_2 t_m\left( x^{\left( m \right)}_1 \right)\cdot t_m^2 \left( x^{\left( m \right)}_2 \right)^2 = - t_m^3\sum_{j=2}^n \left( x^{\left( m \right)} \right)^2 \le 0
\end{equation}
holds, by taking 
\begin{math}
m \rightarrow \infty
\end{math}
in the both sides of the inequality implies that 
\begin{math}
d_1^3 + \epsilon_2 d_1 d_2^2 \le 0
\end{math}. This proves 
\begin{equation}
C^+ \left( h \right) = \overline{C} \left( 0, Q \right) \setminus \left\{ d \in \mathbb{R}^n \middle| d_1^3 + \epsilon_2 d_1 d_2^2 > 0, d_3 = \cdots = d_n = 0 \right\}
\end{equation}
in case all of 
\begin{math}
\epsilon_4, \ldots, \epsilon_n
\end{math}
are 
\begin{math}
1
\end{math}. 

Suppose one of 
\begin{math}
\epsilon_4, \ldots, \epsilon_n
\end{math}
is 
\begin{math}
-1
\end{math}. Without loss of generality, we can assume 
\begin{math}
\epsilon_4 = -1
\end{math}. Take any 
\begin{math}
d \in \mathbb{R}^n
\end{math}
such that 
\begin{math}
d_1^3 + \epsilon_2 d_1 d_2^2 \neq 0, d_3 = \cdots = d_n = 0
\end{math}. Then,  
\begin{math}
\nabla Q \left( d \right) = 0
\end{math}
holds. In that case,
\begin{math}
d
\end{math}
is contained in 
\begin{math}
C^+ \left( h \right)
\end{math}
since 
\begin{math}
\;^t e_3 \mathrm{Hess} \left( Q \right) e_3 = 2
\end{math}
and
\begin{math}
\;^t e_4 \mathrm{Hess} \left( Q \right) e_4 = -2
\end{math}
hold and 
\begin{math}
v
\end{math}
can be chosen to 
\begin{math}
e_3
\end{math}
or 
\begin{math}
e_4
\end{math}
so that 
\begin{math}
\;^t v \mathrm{Hess} \left( Q \right) v
\end{math}
has the opposite sign to that of 
\begin{math}
R \left( d \right)
\end{math}. Therefore, 
\begin{math}
C^+ \left( h \right) = \overline{C} \left( 0, Q \right)
\end{math}
holds in this case. 
\end{proof}

\begin{proposition}

GCQ holds for $h$ if and only if one of 
\begin{math}
\epsilon_j
\end{math}s is 
\begin{math}
-1
\end{math}.

\end{proposition}

\begin{proof} It is easy to check that 
\begin{math}
L^+ \left( h \right)^\circ
\end{math}
is equal to
\begin{math}
\left\{ 0 \right\}
\end{math}.

\noindent
\textbf{Proof of ``if'' part:} Without loss of generality, we can assume 
\begin{math}
\epsilon_4 = -1
\end{math}. Take any 
\begin{math}
w \in \overline{C} \left( 0, Q \right)^\circ
\end{math}. Since 
\begin{math}
\pm e_1, \pm e_2
\end{math}
is contained in 
\begin{math}
\overline{C} \left( 0, Q \right)
\end{math}, 
\begin{math}
w_1 = w_2 = 0
\end{math}
holds. For 
\begin{math}
j \in \left\{ 3, \ldots, n \right\}
\end{math}, either 
\begin{math}
\pm e_3 \pm e_j
\end{math}
or 
\begin{math}
\pm e_4 \pm e_j
\end{math}
is contained in 
\begin{math}
\overline{C} \left( 0, Q \right)
\end{math}
and thus 
\begin{math}
w_3 = \cdots = w_n = 0
\end{math}
holds as well. This proves 
\begin{math}
\overline{C} \left( 0, Q \right)^\circ =\{0\}=L^+ \left( h \right)^\circ
\end{math}. 
By Proposition~\ref{prop:table1_(2)_cone}, $C^+(h)$ is equal to $\overline{C}(0,Q)$, and thus GCQ holds for $h$.

\noindent
\textbf{Proof of ``only if'' part:} We prove this by proving the contraposition. Suppose all of 
\begin{math}
\epsilon_4, \ldots, \epsilon_n
\end{math}
are
\begin{math}
1
\end{math}. Then, 
\begin{math}
\overline{C} \left( 0, Q \right) = \left\{ d \in \mathbb{R}^n \middle| d_3 = \cdots = d_n = 0 \right\}
\end{math}. Since 
\begin{math}
C^+ \left( h \right) \subset \overline{C} \left( 0, Q \right)
\end{math}
holds, 
\begin{math}
\overline{C} \left( 0, Q \right)^\circ \subset C^+ \left( h \right)^\circ
\end{math}
holds. Since 
\begin{math}
e_3
\end{math}
is contained in 
\begin{math}
\overline{C} \left( 0, Q \right)^\circ
\end{math}
but not contained in 
\begin{math}
L^+ \left( h \right)^\circ
\end{math}, GCQ is violated. 
\end{proof}

\begin{example}
Suppose 
\begin{math}
h \left( x \right) = x_1^3 + x_2^2
\end{math}. In this case, the constraint 
\begin{math}
h \left( x \right) = 0
\end{math}
does not satisfy \textnormal{GCQ} at the origin. The theorem by Gould and Tolle \cite{01bf7f42-b4ce-3a05-b3a1-87856b488f37} implies that there exists an objective function 
\begin{math}
f
\end{math}
such that the KKT condition does not hold at a local minimum of 
\begin{math}
f
\end{math}
subject to 
\begin{math}
h = 0
\end{math}. In fact, if we set
\begin{math}
f \left( x \right) = - x_1
\end{math}, then, 
\begin{math}
x = \left( 0, 0 \right)
\end{math}
is the minimum of the above optimization problem but the KKT condition, i.e., there is no constant 
\begin{math}
u \in \mathbb{R}
\end{math}
such that 
\begin{math}
df_0 = u \cdot dh_0
\end{math}
holds. 
\end{example}

{
\subsubsection{ACQ and GCQ in Table~\ref{table:generic constraint r=0}}

First of all, ACQ and GCQ hold for a germ in Table~\ref{table:generic constraint r=0} with $l_1>0$ since MFCQ holds for such a germ.
For this reason, we will discuss germs with $l_1=0$ below.
The linearized cone $L^+(g)$ is equal to $\{d\in \R^n~|~d_1=\cdots =d_{l}=0, d_{l+1},\ldots, d_{q-1}\leq 0\}$ for a germ $g$ in Table~\ref{table:generic constraint r=0}. 
In what follows, for $x = (x_1,\ldots, x_n)\in \R^n$, we denote $x' = (x_{l+1},\ldots, x_n)\in \R^{n-l+1}$. 
Note that the integer $l$ may vary depending on the context.
Nevertheless, we adopt a unified notation $x'$ for simplicity, with the understanding that $l$ is determined by the context in each case.

\subsubsection*{The germ of type $(1,k)$ in Table~\ref{table:generic constraint r=0}}

Let $k\geq 2$, $g = \left(x_1,\ldots, x_{q-1},-\sum_{j=1}^{q-1}x_j+\epsilon_qx_q^k+\sum_{j=q+1}^{n}\epsilon_jx_j^2\right)$, where $\epsilon_j\in \{1,-1\}$. 

\begin{proposition}\label{prop:tangent cone type (1,k) k=2 r=0}

If $k=2$, the tangent cone $C^+(g)$ is equal to $\overline{D}(q-1,Q)$. 

\end{proposition}

\begin{proof}
We can put $R(x')=0$ in this case, and thus the proposition holds by Lemma~\ref{lem:tangent cone for r=0}. 
\end{proof}

\begin{proposition}\label{prop:tangent cone type (1,k) k geq 3 r=0}

Suppose that $k$ is larger than or equal to $3$.
The tangent cone $C^+(g)$ is equal to $\overline{D}(q-1,Q)\setminus \{(0,\ldots,0,d_q,0,\ldots, 0)\in \R^n~|~\epsilon_qd_q^k>0\}$ if $\epsilon_{q+1}=\cdots = \epsilon_n=1$, and $C^+(g)=\overline{D}(q-1,Q)$ otherwise. 

\end{proposition}

\begin{proof}
We can put $Q =\sum_{j=q+1}^{n} \epsilon_j x_j^2$ and $R=\epsilon_qx_q^k$ in this case. 
By Corollary~\ref{cor:tangent cone for r=0} (with $P(x')=0$ and $s=1$), $\overline{D}(q-1,Q)\setminus \{(0,\ldots,0,d_q,0,\ldots,0)~|~\epsilon_qd_q^k> 0\}$ is contained in $C^+(g)$, and $C^+(g)=\overline{D}(q-1,Q)$ if $\epsilon_j=-1$ for some $j\geq q+1$. 
%
If $\epsilon_{q+1}=\cdots =\epsilon_n = +1$, any $x\in M(g)$ satisfies the following inequality:
\[
0\geq -\sum_{j=1}^{q-1}x_j+\epsilon_qx_q^k+\sum_{j=q+1}^{n}x_j^2\Leftrightarrow \epsilon_qx_q^k\leq \sum_{j=1}^{q-1}x_j-\sum_{j=q+1}^{n}x_j^2\leq 0. 
\]
Thus, any $d\in C^+(g)$ also satisfies the inequality $\epsilon_qd_q^k\leq 0$ when $\epsilon_{q+1}=\cdots =\epsilon_n=1$. 
\end{proof}

\begin{proposition}\label{prop:ACQ type (1,k) r=0}

ACQ holds for $g$ if and only if one of the following conditions holds:

\begin{enumerate}

\item 
$k=2$ and $\epsilon_q=\cdots =\epsilon_n=-1$, 

\item 
$k\geq 3$ and $\epsilon_{q+1}=\cdots =\epsilon_n=-1$.

\end{enumerate}

\end{proposition}

\begin{proof}
Suppose that the condition in the proposition does not hold, say $\epsilon_j=1$ for $j\geq q$ (resp.~$j\geq q+1$) if $k=2$ (resp.~$k\geq 3$). 
It is easy to check that $e_j$ is contained in $L^+(g)$ but not in $\overline{D}(q-1,Q)$, in particular $C^+(g)\subset \overline{D}(q-1,Q)\subsetneq L^+(g)$.  
Suppose conversely that the condition in the proposition hold. 
By Propositions \ref{prop:tangent cone type (1,k) k=2 r=0} and \ref{prop:tangent cone type (1,k) k geq 3 r=0}, the tangent cone $C^+(g)$ is equal to $\overline{D}(q-1,Q)$, which is equal to $\{d\in \R^n~|~d_1=\cdots =d_{q-1}=0\}=L^+(g)$ by the assumption. 
\end{proof}

\begin{proposition}\label{prop:GCQ type (1,k) r=0}

GCQ holds for $g$ if and only if one of the following conditions holds:

\begin{enumerate}
\item 
$k=2$ and one of $\epsilon_q,\ldots, \epsilon_n$ is $-1$, 

\item 
$k\geq 3$ and one of $\epsilon_{q+1},\ldots, \epsilon_n$ is $-1$, 

\end{enumerate}

\end{proposition}
\begin{proof}
It is easy to check that $L^+ \left(g\right)^\circ$ is equal to $\left\{ w \in \mathbb{R}^n \middle| w_q=\cdots = w_n=0\right\}$.

\noindent
\textbf{Proof of ``if'' part:} By Propositions~\ref{prop:tangent cone type (1,k) k=2 r=0} and \ref{prop:tangent cone type (1,k) k geq 3 r=0}, the conditions implies that 
\begin{math}
C^+ \left( g \right)
\end{math}
is equal to 
\begin{math}
\overline{D} \left( q-1, Q \right)
\end{math}. 
Since 
\begin{math}
C^+ \left( g\right) \subset L^+ \left( g\right)
\end{math}
always holds, it is enough to show 
\begin{math}
\overline{D} \left( q-1, Q \right)^\circ \subset L^+ \left( g\right)^\circ
\end{math}.

Suppose that the condition 1. in Proposition~\ref{prop:GCQ type (1,k) r=0} holds.
Without loss of generality, we can assume $\epsilon_q = -1$. 
Take any 
\begin{math}
w \in \overline{D} \left( q-1, Q \right)^\circ
\end{math}. 
Since the vector $\pm e_q$ is contained in $\overline{D}(q-1,Q)$, $w\cdot (\pm e_q) = \pm w_q$ is less than or equal to $0$, implying $w_q =0$. 
Since the vector $\pm e_q \pm e_j$ is contained in $\overline{D}(q-1,Q)$ for any $j\in \{q+1,\ldots,n\}$, $w\cdot(\pm e_q \pm e_j)=\pm w_j$ is less than or equal to $0$, implying $w_j=0$. 
We thus obtain $w\in L^+(g)^\circ$.

Suppose that the condition 2. in Proposition~\ref{prop:GCQ type (1,k) r=0} holds.
Without loss of generality, we can assume $\epsilon_{q+1}=-1$.
Take any 
\begin{math}
w \in \overline{D} \left( q-1, Q \right)^\circ
\end{math}. 
Since $\pm e_q$ is contained in $\overline{D}(q-1,Q)$, $w\cdot (\pm e_q)=\pm w_q$ is less than or equal to $0$, implying $w_q=0$. 
Since the vector $\pm e_{q+1}$ is contained in $\overline{D}(q-1,Q)$, $w\cdot (\pm e_{q+1}) = \pm w_{q+1}$ is less than or equal to $0$, implying $w_{q+1}=0$. 
Since the vector $\pm e_{q+1} \pm e_j$ is contained in $\overline{D}(q-1,Q)$ for any $j\in \{q+2,\ldots,n\}$, $w\cdot(\pm e_{q+1} \pm e_j)=\pm w_j$ is less than or equal to $0$, implying $w_j=0$. 
We thus obtain $w\in L^+(g)^\circ$.

\noindent
\textbf{Proof of ``only if'' part:}
We show the contraposition of the statement.
%
%
If $k=2$ and
\begin{math}
\epsilon_q = \cdots = \epsilon_n = 1
\end{math}, Proposition~\ref{prop:tangent cone type (1,k) k=2 r=0} implies that 
\begin{math}
C^+ \left( g\right)
\end{math}
is equal to 
\[
\overline{D}(q-1,Q) = \left\{d\in \R^n~\middle|~ d_1=\cdots =d_{q-1}=0, \sum_{j=q}^{n}d_j^2 = 0\right\}=\left\{0\right\}. 
\]
In particular, $C^+(g)^\circ = \R^n \neq L^+(g)^\circ$.
If $k\geq 3$ and  
\begin{math}
\epsilon_{q+1} = \cdots = \epsilon_n = 1
\end{math}, Proposition~\ref{prop:tangent cone type (1,k) k geq 3 r=0} implies that $C^+(g)$ is equal to 
\begin{align*}
&\overline{D}(q-1,Q)\setminus \{(0,\ldots,0,d_q,0,\ldots, 0)\in \R^n ~|~ \epsilon_q d_q^k >0\}\\
=& \{(0,\ldots,0,d_q,0,\ldots, 0)\in \R^n ~|~ \epsilon_q d_q^k \leq 0\}.
\end{align*}
In particular, \begin{math}
C^+ \left( g\right)^\circ
\end{math} contains $e_{q+1}$, which is not contained in $L^+(g)^\circ$.  
\end{proof}

\subsubsection*{The germ of type $(2)$ in Table~\ref{table:generic constraint r=0}}

Let $g = \left(x_1,\ldots, x_{q-1},-\sum_{j=1}^{q-1}x_j+x_q^3+\epsilon_{q+1}x_qx_{q+1}^2+\sum_{j=q+2}^{n}\epsilon_jx_j^2\right)$, where $\epsilon_j\in \{1,-1\}$, $Q =\sum_{j=q+2}^{n} \epsilon_j x_j^2$, and $R=x_q^3+\epsilon_{q+1}x_qx_{q+1}^2$. 

\begin{proposition}\label{prop:tangent cone type (2) r=0}

The tangent cone $C^+(g)$ is equal to $\overline{D}(q-1,Q)\setminus \{(0,\ldots,0,d_q,d_{q+1},0,\ldots, 0)\in \R^n~|~d_q^3+\epsilon_{q+1}d_qd_{q+1}^2>0\}$ if $\epsilon_{q+2}=\cdots = \epsilon_n=1$, and $C^+(g)=\overline{D}(q-1,Q)$ otherwise. 

\end{proposition}

\begin{proof}
By Corollary~\ref{cor:tangent cone for r=0} (with $P(x')=0$ and $s=2$), $\overline{D}(q-1,Q)\setminus \{(0,\ldots,0,d_q,d_{q+1},0,\ldots,0)~|~d_q^3+\epsilon_{q+1}d_qd_{q+1}^2> 0\}$ is contained in $C^+(g)$, and $C^+(g)=\overline{D}(q-1,Q)$ if $\epsilon_j=-1$ for some $j\geq q+2$. 
%
%
If $\epsilon_{q+2}=\cdots =\epsilon_n = +1$, any $x\in M(g)$ satisfies the following inequality:
\[
0\geq -\sum_{j=1}^{q-1}x_j+x_q^3+\epsilon_{q+1}x_qx_{q+1}^2+\sum_{j=q+2}^{n}x_j^2\Leftrightarrow x_q^3+\epsilon_{q+1}x_qx_{q+1}^2\leq \sum_{j=1}^{q-1}x_j-\sum_{j=q+2}^{n}x_j^2\leq 0. 
\]
Thus, any $d\in C^+(g)$ also satisfies the inequality $d_q^3+\epsilon_{q+1}d_qd_{q+1}^2\leq 0$ when $\epsilon_{q+2}=\cdots =\epsilon_n=1$. 
\end{proof}

\begin{proposition}\label{prop:ACQ type (2) r=0}

ACQ holds for $g$ if and only if $\epsilon_{q+2}=\cdots =\epsilon_n=-1$, 

\end{proposition}

\begin{proof}
Suppose that the condition in the proposition does not hold, say $\epsilon_j=1$ for $j\geq q+2$. 
It is easy to check that $e_j$ is contained in $L^+(g)$ but not in $\overline{D}(q-1,Q)$, in particular $C^+(g)\subset \overline{D}(q-1,Q)\subsetneq L^+(g)$.  
Suppose conversely that the condition in the proposition hold. 
By Proposition \ref{prop:tangent cone type (2) r=0}, the tangent cone $C^+(g)$ is equal to $\overline{D}(q-1,Q)$, which is equal to $\{d\in \R^n~|~d_1=\cdots =d_{q-1}=0\}=L^+(g)$ by the assumption. 
\end{proof}

\begin{proposition}

GCQ holds for $g$ if and only if one of 
\begin{math}
\epsilon_{q+2}, \ldots, \epsilon_n
\end{math}
is $-1$.

\end{proposition}

\begin{proof}
It is easy to check that $L^+ \left(g\right)^\circ$ is equal to $\left\{ w \in \mathbb{R}^n \middle| w_q=\cdots = w_n=0\right\}$.

\noindent
\textbf{Proof of ``if'' part:} Without loss of generality, we can assume 
\begin{math}
\epsilon_{q+2} = -1
\end{math}. By Propositions~\ref{prop:tangent cone type (2) r=0}, the condition implies that 
\begin{math}
C^+ \left( g \right)
\end{math}
is equal to 
\begin{math}
\overline{D} \left( q-1, Q \right)
\end{math}. 
Since 
\begin{math}
C^+ \left( g\right) \subset L^+ \left( g\right)
\end{math}
always holds, it is enough to show 
\begin{math}
\overline{D} \left( q-1, Q \right)^\circ \subset L^+ \left( g\right)^\circ
\end{math}. Take any 
\begin{math}
w \in \overline{D} \left( q-1, Q \right)^\circ
\end{math}. 
Since the vector $\pm e_q, \pm e_{q+1}$ is contained in $\overline{D}(q-1,Q)$, $w_q =0$ and $w_{q+1} = 0$ hold. Since the vector 
\begin{math}
\pm e_{q+2}
\end{math}
is contained in 
\begin{math}
\overline{D} \left( q-1, Q \right)
\end{math}, 
\begin{math}
w_{q+2} = 0
\end{math}
holds. 
Since the vector $\pm e_{q+2} \pm e_j$ is contained in $\overline{D}(q-1,Q)$ for any $j\in \{q+3,\ldots,n\}$, $w\cdot(\pm e_{q+2} \pm e_j)=\pm w_j$ is less than or equal to $0$, implying $w_j=0$. 
We thus obtain $w\in L^+(g)^\circ$. 

\noindent
\textbf{Proof of ``only if'' part:}
We show the contraposition of the statement, that is, if all of 
\begin{math}
\epsilon_{q+2}, \ldots, \epsilon_n
\end{math}
are 
\begin{math}
1
\end{math}, GCQ is violated. Under the assumption,
\[
\overline{D}(q-1,Q) = \left\{d\in \R^n~\middle|~ d_1=\cdots =d_{q-1}=0, \sum_{j=q+2}^{n}d_j^2=0\right\}
\]
and thus 
\begin{math}
\overline{D}(q-1,Q)^\circ = \left\{ w \in \mathbb{R}^n \middle| w_q = w_{q+1} = 0 \right\}
\end{math}
holds. Therefore, $C^+(g)^\circ \supset \overline{D}(q-1,Q)^\circ \supsetneq L^+(g)^\circ$ holds and thus GCQ is violated. 
\end{proof}

\subsubsection*{The germ of type $(3,k)$ in Table~\ref{table:generic constraint r=0}}

Let $k\geq 2$, $g = \left(x_1,\ldots, x_{q-1},-\sum_{j=1}^{q-2}x_j+\epsilon_{q-1}x_{q-1}^k+\sum_{j=q}^{n}\epsilon_jx_j^2\right)$, where $\epsilon_j\in \{1,-1\}$. 

\begin{proposition}\label{prop:tangent cone type (3,k) k=2 r=0}

If $k=2$, the tangent cone $C^+(g)$ is equal to $\overline{D}(q-2,Q)$. 

\end{proposition}

\begin{proof}
We can put $R(x')=0$ in this case, and thus the proposition holds by Lemma~\ref{lem:tangent cone for r=0}. 
\end{proof}

\begin{proposition}\label{prop:tangent cone type (3,k) k geq 3 r=0}

Suppose that $k$ is larger than or equal to $3$.
The tangent cone $C^+(g)$ is equal to $\overline{D}(q-2,Q)\setminus \{(0,\ldots,0,d_{q-1},0,\ldots, 0)\in \R^n~|~\epsilon_{q-1}d_{q-1}^k>0\}$ if $\epsilon_{q}=\cdots = \epsilon_n=1$, and $C^+(g)=\overline{D}(q-2,Q)$ otherwise. 

\end{proposition}

\begin{proof}
We can put $Q =\sum_{j=q}^{n} \epsilon_j x_j^2$ and $R=\epsilon_{q-1}x_{q-1}^k$ in this case. 
By Corollary~\ref{cor:tangent cone for r=0} (with $P(x')=0$ and $s=1$), $\overline{D}(q-2,Q)\setminus \{(0,\ldots,0,d_{q-1},0,\ldots,0)~|~\epsilon_{q-1}d_{q-1}^k> 0\}$ is contained in $C^+(g)$, and $C^+(g)=\overline{D}(q-2,Q)$ if $\epsilon_j=-1$ for some $j\geq q$. 
If $\epsilon_{q}=\cdots =\epsilon_n = +1$, any $x\in M(g)$ satisfies the following inequality:
\[
0\geq -\sum_{j=1}^{q-2}x_j+\epsilon_{q-1}x_{q-1}^k+\sum_{j=q}^{n}x_j^2\Leftrightarrow \epsilon_{q-1}x_{q-1}^k\leq \sum_{j=1}^{q-2}x_j-\sum_{j=q}^{n}x_j^2\leq 0. 
\]
Thus, any $d\in C^+(g)$ also satisfies the inequality $\epsilon_{q-1}d_{q-1}^k\leq 0$ when $\epsilon_{q}=\cdots =\epsilon_n=1$. 
\end{proof}

\begin{proposition}\label{prop:ACQ type (3) r=0}

ACQ holds for $g$ if and only if one of the following conditions holds:

\begin{enumerate}

\item 
$k=2$ and $\epsilon_{q-1}=\cdots =\epsilon_n=-1$, 

\item 
$k\geq 3$ and $\epsilon_{q}=\cdots =\epsilon_n=-1$.

\end{enumerate}

\end{proposition}

\begin{proof}
Suppose that the condition in the proposition does not hold, say $\epsilon_j=1$ for $j\geq q-1$ (resp.~$j\geq q$) if $k=2$ (resp.~$k\geq 3$). 
It is easy to check that $-e_j$ is contained in $L^+(g)$ but not in $\overline{D}(q-2,Q)$, in particular $C^+(g)\subset \overline{D}(q-2,Q)\subsetneq L^+(g)$.  
Suppose conversely that the condition in the proposition hold. 
By Propositions \ref{prop:tangent cone type (3,k) k=2 r=0} and \ref{prop:tangent cone type (3,k) k geq 3 r=0}, the tangent cone $C^+(g)$ is equal to $\overline{D}(q-2,Q)$, which is equal to $\{d\in \R^n~|~d_1=\cdots =d_{q-2}=0,d_{q-1}\leq 0\}=L^+(g)$ by the assumption. 
\end{proof}

\begin{proposition}

GCQ holds for $g$ if and only if one of 
\begin{math}
\epsilon_q, \ldots, \epsilon_n
\end{math}
is $-1$.

\end{proposition}

\begin{proof}
It is easy to check that $L^+ \left(g\right)^\circ$ is equal to $\left\{ w \in \mathbb{R}^n \middle| w_{q-1} \ge 0, w_q =\cdots = w_n=0\right\}$.

\noindent
\textbf{Proof of ``if'' part:} Without loss of generality, we can assume 
\begin{math}
\epsilon_q = -1
\end{math}. 
By Propositions~\ref{prop:tangent cone type (3,k) k=2 r=0} and \ref{prop:tangent cone type (3,k) k geq 3 r=0}, $C^+(g)$ is equal to $\overline{D}(q-2,Q)$. 
Take any 
\begin{math}
w \in \overline{D}(q-2,Q)^\circ
\end{math}. Since 
\begin{math}
\pm e_q
\end{math}
is contained in 
\begin{math}
\overline{D} \left( q-2, Q \right)
\end{math}, 
\begin{math}
w_q = 0
\end{math}
holds.  Since 
\begin{math}
- e_{q-1} \pm e_q
\end{math}
is contained in 
\begin{math}
\overline{D} \left( q-2, Q \right)
\end{math}, 
\begin{math}
w_{q-1} \ge 0
\end{math}
holds. Since 
\begin{math}
\pm e_q \pm e_j
\end{math}
is contained in 
\begin{math}
\overline{D} \left( q-2, Q \right)
\end{math}
for
\begin{math}
j \in \left\{ q+1, \ldots, n \right\}
\end{math}, 
\begin{math}
w_j = 0
\end{math}
holds. This proves 
\begin{math}
w \in L^+ \left( g \right)^\circ
\end{math}.

\noindent
\textbf{Proof of ``only if'' part:}
We show the contraposition of the statement, that is, if all of 
\begin{math}
\epsilon_{q+2}, \ldots, \epsilon_n
\end{math}
are 
\begin{math}
1
\end{math}, GCQ is violated. Under the assumption and $k \ge 3$,
\[
\overline{D}(q-2,Q) = \left\{d\in \R^n~\middle|~ d_1=\cdots =d_{q-2}=0, d_{q-1} \le 0, \sum_{j=q}^{n}d_j^2=0\right\}
\]
and thus 
\begin{math}
\overline{D}(q-2,Q)^\circ = \left\{ w \in \mathbb{R}^n \middle| w_{q-1} \ge 0 \right\}
\end{math}
holds. Therefore, $C^+(g)^\circ \supset \overline{D}(q-2,Q)^\circ \supsetneq L^+(g)^\circ$ holds and thus GCQ is violated. 

Under the assumption and $k = 2$,
\[
\overline{D}(q-2,Q) = \left\{d\in \R^n~\middle|~ d_1=\cdots =d_{q-2}=0, d_{q-1} \le 0, \epsilon_{q-1} d_{q-1}^2 + \sum_{j=q}^{n}d_j^2 \le 0\right\}
\]
holds. 
Take any 
\begin{math}
d \in \overline{D}(q-2,Q)
\end{math}. Then, 
\begin{math}
d_q^2 \le \sum_{j=q}^{n}d_j^2 \le \left| d_{q-1}^2 \right|
\end{math}
and thus 
\begin{math}
\left| d_q \right| \le \left| d_{q-1} \right|
\end{math}
and
\begin{math}
(e_{q-1} + e_q)\cdot d = d_{q-1} + d_q \le 0
\end{math}
hold. 
Thus, \begin{math}
w = e_{q-1} + e_q
\end{math}
is contained in 
\begin{math}
\overline{D}(q-2,Q)^\circ
\end{math}. 
Since 
\begin{math}
w
\end{math}
is not contained in 
\begin{math}
L^+ \left( g \right)^\circ
\end{math}, GCQ is violated. 
\end{proof}

\subsubsection*{The germ of type $(4,k)$ in Table~\ref{table:generic constraint r=0}}

Let $k\geq 3$, $g = \left(x_1,\ldots, x_{q-1},-\sum_{j=1}^{q-2}x_j+\epsilon_q x_q^k+x_{q-1}x_q+\sum_{j=q+1}^{n}\epsilon_jx_j^2\right)$, where $\epsilon_j\in \{1,-1\}$, $Q =x_{q-1}x_q+\sum_{j=q+1}^{n} \epsilon_j x_j^2$, and $R=\epsilon_q x_q^k$. 

\begin{proposition}\label{prop:tangent cone type (4,k) r=0}

The tangent cone $C^+(g)$ is equal to $\overline{D}(q-2,Q)\setminus \{(0,\ldots,0,d_{q},0,\ldots, 0)\in \R^n~|~d_q<0,\epsilon_qd_q^k>0\}$ if $\epsilon_{q+1}=\cdots = \epsilon_n=1$, and $C^+(g)=\overline{D}(q-2,Q)$ otherwise. 

\end{proposition}

\begin{proof}
By Corollary~\ref{cor:tangent cone for r=0} (with $P(x')=x_{q-1}x_q$ and $s=2$), the set 
\[
\overline{D}(q-2,Q)\setminus \{(0,\ldots,0,d_{q-1},d_{q},0,\ldots, 0)\in \R^n~|~\epsilon_qd_q^k>0\}
\]
is contained in $C^+(g)$, and $C^+(g)=\overline{D}(q-2,Q)$ if $\epsilon_j=-1$ for some $j\geq q+1$. 

In what follows, we assume $\epsilon_{q+1}=\cdots =\epsilon_n = +1$. 
Let $d=(0,\ldots,0,d_{q-1},d_q,0,\ldots, 0)\in \overline{D}(q-2,Q)$ with $\epsilon_qd_q^k>0$. 
Since $d\in \overline{D}(q-2,Q)$, $d_{q-1}$ and $Q(d') = d_{q-1}d_q$ are less than or equal to $0$. 
If $d_{q-1}<0$, $d_{q-1}d_q$ is also less than $0$ since $d_q$ is not $0$.
The vector $m^{-1}d$ is contained in $M(g)$ for $m\gg 0$ since $g_j(m^{-1}d) = m^{-1}d_j \leq 0$ for $j=1,\ldots, q-1$ and 
\[
\widetilde{g}(m^{-1}d) = \epsilon_qm^{-k}d_q^k + m^{-2}d_{q-1}d_q <0.
\]
Hence $d =\lim_{m\to \infty} m\cdot m^{-1}d$ is contained in $C^+(g)$. 
If $d_{q-1}=0$ and $d_q>0$, the vector $x^{(m)}:=m^{-1}d_qe_q -m^{-3/2}d_qe_{q-1}$ is contained in $M(g)$ for $m\gg 0$ since 
\[
g_j\left(x^{(m)}\right) = \begin{cases}
0 & (j\leq q-2)\\
-m^{-3/2}d_q<0 & (j=q-1),
\end{cases}
\]
and $\widetilde{g}\left(x^{(m)}\right) = m^{-k}\epsilon_qd_q^k-m^{-5/2}d_q^2 <0$ for $m\gg 0$. 
Thus, $d = \lim_{m\to \infty} m\cdot x^{(m)}$ is contained in $C^+(g)$. 
We eventually obtain:
\[
\overline{D}(q-2,Q)\setminus \{(0,\ldots,0,d_{q},0,\ldots, 0)\in \R^n~|~d_q<0,\epsilon_qd_q^k>0\}\subset C^+(g). 
\]

Suppose that $d = (0,\ldots, 0,d_{q-1},d_q,0,\ldots,0)$ with $d_q<0$ and $\epsilon_qd_q^k>0$ is contained in $C^+(g)$. 
Take $x^{(m)}\in M(g)$ and $t_m>0$ so that $\lim_{m\to\infty}x^{(m)}=0$ and $\lim_{m\to \infty}t_m x^{(m)}=d$. 
The following inequality holds:
\begin{align*}
&0 \geq -\sum_{j=1}^{q-2}x_j^{(m)}+\epsilon_q\left(x_q^{(m)}\right)^k+x_{q-1}^{(m)}x_{q}^{(m)}+\sum_{j=q+1}^{n}(x_j^{(m)})^2\\
\Leftrightarrow&\epsilon_q\left(x_q^{(m)}\right)^k+x_{q-1}^{(m)}x_{q}^{(m)}\leq \sum_{j=1}^{q-2}x_j^{(m)}-\sum_{j=q+1}^{n}(x_j^{(m)})^2\leq 0.
\end{align*}
Since $\lim_{m\to \infty}t_m x^{(m)}=d$, $d_q<0$ and $\epsilon_qd_q^k>0$, $\epsilon_q\left(t_mx_q^{(m)}\right)^k$ is larger than $0$ and $\left(t_mx_{q-1}^{(m)}\right)\left(t_mx_{q}^{(m)}\right)\geq 0$ for $m\gg 0$. 
Thus, $\epsilon_q\left(x_q^{(m)}\right)^k+x_{q-1}^{(m)}x_{q}^{(m)}$ is larger than $0$, contradicting the inequality above. 
Hence we obtain:
\[
\overline{D}(q-2,Q)\setminus \{(0,\ldots,0,d_{q},0,\ldots, 0)\in \R^n~|~d_q<0,\epsilon_qd_q^k>0\}=C^+(g). \qedhere
\]
\end{proof}

\begin{proposition}\label{prop:ACQ type (4) r=0}

ACQ does not hold for $g$.

\end{proposition}

\begin{proof}
The vector $-e_{q-1}-\epsilon_q e_q$ is contained in $L^+(g)$ but not in $\overline{D}(q-2,Q)$, in particular $C^+(g)\subset \overline{D}(q-1,Q)\subsetneq L^+(g)$.  
\end{proof}

\begin{proposition}

GCQ holds for $g$ if and only if either of the following holds:
\begin{enumerate}
\item one of 
\begin{math}
\epsilon_{q+1}, \ldots, \epsilon_n
\end{math}
is 
\begin{math}
-1
\end{math}.
\item 
\begin{math}
\epsilon_{q+1} = \cdots = \epsilon_n = 1
\end{math}
and \begin{math}
\epsilon_q = (-1)^{k+1}
\end{math}.
\end{enumerate}

\end{proposition}

\begin{proof}
It is easy to check that $L^+ \left(g\right)^\circ$ is equal to 
\begin{math}
\left\{ w \in \mathbb{R}^n \middle| w_{q-1} \ge 0, w_q =\cdots = w_n=0\right\}
\end{math}.
In what follows, we will show that 
\begin{math}
\overline{D} \left( q-2, Q \right)^\circ \subset L^+ \left( g \right)^\circ
\end{math}
holds. Take any 
\begin{math}
w \in \overline{D} \left( q-2, Q \right)^\circ
\end{math}. 
Since 
\begin{math}
-e_{q-1}
\end{math}
is contained in 
\begin{math}
\overline{D} \left( q-2, Q \right)
\end{math}, 
\begin{math}
w_{q-1} \ge 0
\end{math}
holds. Since 
\begin{math}
\pm e_q
\end{math}
is contained in 
\begin{math}
\overline{D} \left( q-2, Q \right)
\end{math}, 
\begin{math}
w_q = 0
\end{math}
holds. For 
\begin{math}
j \in \left\{ q+1, \ldots, n \right\}
\end{math}
and 
\begin{math}
s > 0
\end{math}, 
\begin{math}
\displaystyle - \frac{1}{s} e_{q-1} + \epsilon_j s e_q \pm e_j
\end{math}
is contained in 
\begin{math}
\overline{D} \left( q-2, Q \right)
\end{math}, 
\begin{math}
\displaystyle - \frac{1}{s} w_{q-1} \pm w_j \le 0
\end{math}
holds. This 
\begin{math}
s
\end{math}
can be chosen arbitrarily large and thus 
\begin{math}
w_j = 0
\end{math}
holds for 
\begin{math}
j \in \left\{ q+1, \ldots, n \right\}
\end{math}. This proves 
\begin{math}
w \in L^+ \left( g \right)^\circ
\end{math}. 

If one of
\begin{math}
\epsilon_{q+1}, \ldots, \epsilon_n
\end{math}
is 
\begin{math}
-1
\end{math},
\begin{math}
C^+ \left( g \right) = \overline{D} \left( q-2, Q \right)
\end{math}
holds and GCQ hold in that case. 

If 
\begin{math}
\epsilon_{q+1} = \cdots =\epsilon_n = 1
\end{math}
holds,
\begin{equation}
C^+ \left( g \right) = \overline{D}(q-2,Q)\setminus \{(0,\ldots,0,d_{q},0,\ldots, 0)\in \R^n~|~d_q<0,\epsilon_qd_q^k>0\}
\end{equation}
holds. If 
\begin{math}
\epsilon_q = (-1)^{k+1}
\end{math}, 
\begin{math}
C^+ \left( g \right) = \overline{D} \left( q-2, Q \right)
\end{math}
holds and GCQ holds in that case. If 
\begin{math}
\epsilon_q = (-1)^{k}
\end{math}, 
\begin{equation}
C^+ \left( g \right) = \overline{D}(q-2,Q)\setminus \{(0,\ldots,0,d_{q},0,\ldots, 0)\in \R^n~|~d_q<0\}
\end{equation}
holds. 
In what follows, we will show that
\begin{math}
w = -  e_q
\end{math}
is contained in 
\begin{math}
C^+ \left( g \right)^\circ
\end{math}. Take any 
\begin{math}
d \in C^+ \left( g \right)
\end{math}. 
If one of 
\begin{math}
d_{q+1}, \ldots, d_n
\end{math}
is non-zero, 
\begin{equation}
d_{q-1} d_q \le - \sum_{j=q+1}^n d_j^2 < 0
\end{equation}
holds. This implies that 
\begin{math}
d_{q-1} < 0
\end{math}
and thus 
\begin{math}
d_q > 0
\end{math}
holds. Therefore, 
\begin{math}
w \cdot d = - d_q < 0
\end{math}
holds. If 
\begin{math}
d_{q+1} = \cdots = d_n = 0
\end{math}
holds, 
\begin{math}
 d_q \ge 0
\end{math}
and thus 
\begin{math}
w \cdot d = - d_q \le 0
\end{math}. This proves 
\begin{math}
w \in C^+ \left( g \right)^\circ
\end{math}. Since 
\begin{math}
w
\end{math}
is not contained in 
\begin{math}
L^+ \left( g \right)^\circ
\end{math}, GCQ is violated in this case. 
\end{proof}

\subsubsection*{The germ of type $(5)$ in Table~\ref{table:generic constraint r=0}}

Let $g = \left(x_1,\ldots, x_{q-1},-\sum_{j=1}^{q-2}x_j+\epsilon_{q-1} x_{q-1}^2+x_q^3+\sum_{j=q+1}^{n}\epsilon_jx_j^2\right)$, where $\epsilon_j\in \{1,-1\}$, $Q =\epsilon_{q-1}x_{q-1}^2+\sum_{j=q+1}^{n} \epsilon_j x_j^2$, and $R=x_q^3$. 

\begin{proposition}\label{prop:tangent cone type (5) r=0}

The tangent cone $C^+(g)$ is equal to $\overline{D}(q-2,Q)\setminus \{(0,\ldots,0,d_{q},0,\ldots, 0)\in \R^n~|~d_q^3>0\}$ if $\epsilon_{q-1}=\epsilon_{q+2}=\cdots = \epsilon_n=1$, and $C^+(g)=\overline{D}(q-2,Q)$ otherwise. 

\end{proposition}

\begin{proof}
By Corollary~\ref{cor:tangent cone for r=0} (with $P(x')=\epsilon_{q-1}x_{q-1}^2$ and $s=2$), the set
\[
\overline{D}(q-2,Q)\setminus \{(0,\ldots,0,d_{q-1},d_{q},0,\ldots, 0)\in \R^n~|~d_q^3>0\}
\]
is contained in $C^+(g)$, and $C^+(g)=\overline{D}(q-2,Q)$ if $\epsilon_j=-1$ for some $j\geq q+1$. 

In what follows, we assume $\epsilon_{q+1}=\cdots =\epsilon_n = +1$. 
Let $d=(0,\ldots,0,d_{q-1},d_q,0,\ldots, 0)\in \overline{D}(q-2,Q)$ with $d_q^3>0\Leftrightarrow d_q>0$. 
Since $d\in \overline{D}(q-2,Q)$, $d_{q-1}\leq 0$ and $Q(d')= \epsilon_{q-1}d_{q-1}^2\leq 0$. 
If $\epsilon_{q-1}=-1$ and $d_{q-1} <0$, the vector $v= d_{q-1}e_{q-1}$ satisfies the conditions in Lemma~\ref{lem:tangent cone for r=0}. 
Indeed, $v_1=\cdots =v_{q-2}=0$, $v_{q-1}=d_{q-1}<0$ (in particular $v$ satisfies the condition (a) in Lemma~\ref{lem:tangent cone for r=0}), and $v'\cdot \nabla Q(d') =-2 d_{q-1}^2 <0$.
If $\epsilon_{q-1}=-1$ and $d_{q-1}=0$, the vector $v=-e_{q-1}$ satisfies the conditions in Lemma~\ref{lem:tangent cone for r=0}.
Indeed, $v_1=\cdots =v_{q-2}=0$, $v_{q-1}=-1<0$, $v'\cdot \nabla Q(d')=0$, and ${}^tv'\mathrm{Hess}(Q)v' =-2 <0$. 
We thus obtain $C^+(g)=\overline{D}(q-2,Q)$ if $\epsilon_{q-1}=-1$. 

If $\epsilon_{q-1}=1$, any $x\in M(g)$ satisfies the following inequality: 
\[
0\geq -\sum_{j=1}^{q-2}x_j+x_{q-1}^2+x_q^3+\sum_{j=q+1}^{n}x_j^2\Leftrightarrow x_q^3\leq \sum_{j=1}^{q-1}x_j-x_{q-1}^2-\sum_{j=q+1}^{n}x_j^2\leq 0. 
\]
Thus, any $d\in C^+(g)$ also satisfies the inequality $d_q^3\leq 0$. 
\end{proof}

\begin{proposition}\label{prop:ACQ type (5) r=0}

ACQ holds for $g$ if and only if $\epsilon_{q-1}=\epsilon_{q+1}=\cdots =\epsilon_n =-1$.

\end{proposition}

\begin{proof}
Suppose that the condition in the proposition does not hold, say $\epsilon_j=1$ for $j=q-1$ or $j\geq q+1$. 
It is easy to check that $-e_j$ is contained in $L^+(g)$ but not in $\overline{D}(q-2,Q)$, in particular $C^+(g)\subset \overline{D}(q-2,Q)\subsetneq L^+(g)$.  
Suppose conversely that the condition in the proposition hold. 
By Proposition \ref{prop:tangent cone type (5) r=0}, the tangent cone $C^+(g)$ is equal to $\overline{D}(q-2,Q)$, which is equal to $\{d\in \R^n~|~d_1=\cdots =d_{q-2}=0,d_{q-1}\leq 0\}=L^+(g)$ by the assumption. 
\end{proof}

\begin{proposition}

GCQ holds for $g$ if and only if one of 
\begin{math}
\epsilon_{q+1}, \ldots, \epsilon_n
\end{math}
is 
\begin{math}
-1
\end{math}.

\end{proposition}

\begin{proof}
It is easy to check that $L^+ \left(g\right)^\circ$ is equal to $\left\{ w \in \mathbb{R}^n \middle| w_{q-1} \ge 0, w_q =\cdots = w_n=0\right\}$.

\noindent
\textbf{Proof of ``if'' part:}
Without loss of generality, we can assume 
\begin{math}
\epsilon_{q+1} = -1
\end{math}. 
By Proposition~\ref{prop:tangent cone type (5) r=0}, $C^+(g)$ is equal to $\overline{D}(q-2,Q)$.
Take any 
\begin{math}
w \in C^+ \left( g \right)^\circ
\end{math}. Since 
\begin{math}
\pm e_q
\end{math}
is contained in 
\begin{math}
\overline{D} \left( q-2, Q \right)
\end{math}, 
\begin{math}
w_q = 0
\end{math}
holds. Since 
\begin{math}
\pm e_{q+1}
\end{math}
is contained in 
\begin{math}
\overline{D} \left( q-2, Q \right)
\end{math}, 
\begin{math}
w_{q+1} = 0
\end{math}
holds. Since
\begin{math}
\pm e_{q+1} \pm e_j
\end{math}
is contained in 
\begin{math}
\overline{D} \left( q-2, Q \right)
\end{math}
for
\begin{math}
j \in \left\{ q+2, \ldots, n \right\}
\end{math}, 
\begin{math}
w_j = 0
\end{math}
holds. 
Since 
\begin{math}
- e_{q-1} \pm e_{q+1}
\end{math}
is contained in 
\begin{math}
\overline{D} \left( q-2, Q \right)
\end{math}, 
\begin{math}
w_{q-1} \ge 0
\end{math}
holds. This proves 
\begin{math}
w \in L^+ \left( g \right)^\circ
\end{math}. 

\noindent
\textbf{Proof of ``only if'' part:}
We show the contraposition of the statement, that is, if all of 
\begin{math}
\epsilon_{q+1}, \ldots, \epsilon_n
\end{math}
are 
\begin{math}
1
\end{math}, GCQ is violated. Under the assumption,
\[
\overline{D}(q-2,Q) = \left\{d\in \R^n~\middle|~ d_1=\cdots =d_{q-2}=0, d_{q-1} \le 0, \epsilon_{q-1} d_{q-1}^2 + \sum_{j=q}^{n}d_j^2\le0\right\}
\]
holds. 
Take any 
\begin{math}
d \in \overline{D}(q-2,Q)
\end{math}. Then, 
\begin{math}
d_q^2 \le \sum_{j=q}^{n}d_j^2 \le \left| d_{q-1}^2 \right|
\end{math}, in particular
\begin{math}
\left| d_q \right| \le \left| d_{q-1} \right|
\end{math},
and thus
\begin{math}
(e_{q-1}+e_q) \cdot d = d_{q-1} + d_q \le 0
\end{math}
hold. 
Thus, 
\begin{math}
w = e_{q-1} + e_q
\end{math}
is contained in 
\begin{math}
\overline{D}(q-2,Q)^\circ
\end{math}. 
Since 
\begin{math}
w
\end{math}
is not contained in 
\begin{math}
L^+ \left( g \right)^\circ
\end{math}, GCQ is violated.
\end{proof}

\subsubsection*{The germ of type $(6)$ in Table~\ref{table:generic constraint r=0}}

Let $g = \left(x_1,\ldots, x_{q-1},-\sum_{j=1}^{q-3}x_j+\sum_{j=1}^{2}\delta_jx_{q-j}^2+\alpha x_{q-2}x_{q-1}+\sum_{j=q}^{n}\epsilon_jx_j^2\right)$, where $\epsilon_j,\delta_j\in \{1,-1\}$ and $\alpha \in \R$ satisfying $4\delta_1\delta_2-\alpha^2\neq 0$. 

\begin{proposition}\label{prop:tangent cone type (6) r=0}

The tangent cone $C^+(g)$ is equal to $\overline{D}(q-3,Q)$. 

\end{proposition}

\begin{proof}
We can put $R(x')=0$ in this case, and thus the proposition holds by Lemma~\ref{lem:tangent cone for r=0}. 
\end{proof}

\begin{proposition}\label{prop:ACQ type (6) r=0}

ACQ holds for $g$ if and only if $\delta_1=\delta_2=-1$, $\alpha <2$, and $\epsilon_{q}=\cdots =\epsilon_n =-1$.

\end{proposition}

\begin{proof}
If $\delta_j=1$ for $j=1,2$, $-e_{q-j}$ is contained in $L^+(g)$ but not in $\overline{D}(q-3,Q)$.
If $\delta_1=\delta_2=-1$ and $\alpha>2$, $-e_{q-1}-e_q$ is contained in $L^+(g)$ but not in $\overline{D}(q-3,Q)$.
If $\epsilon_j=1$ for $j\geq q$, $-e_j$ is contained in $L^+(g)$ but not in $\overline{D}(q-3,Q)$. 
In each case, $C^+(g)=\overline{D}(q-3,Q)$ is a proper subset of $L^+(g)$.

Suppose that the condition in the proposition hold. 
For $d\in L^+(g) = \{d\in \R^n~|~d_1=\cdots =d_{q-3}=0,d_{q-2},d_{q-1}\leq 0\}$, the value $Q(d)$ is estimated as follows:
\[
Q(d) = -(d_{q-1}-d_{q-2})^2 +(\alpha-2)d_{q-1}d_{q-2} -\sum_{j=q}^{n}d_j^2 \leq 0. 
\]
Thus, $d$ is contained in $\overline{D}(q-3,Q)=C^+(g)$. 
\end{proof}

\begin{proposition}\label{prop:GCQ type (6) r=0}

GCQ holds for $g$ if and only if one of 
\begin{math}
\epsilon_q, \ldots, \epsilon_n
\end{math}
is 
\begin{math}
-1
\end{math}.

\end{proposition}

\begin{proof}
It is easy to check that $L^+ \left(g\right)^\circ$ is equal to 
\begin{equation}
\left\{ w \in \mathbb{R}^n \middle| w_{q-2} \ge 0, w_{q-1} \ge 0, w_q =\cdots = w_n=0\right\}.
\end{equation}

\noindent
\textbf{Proof of ``if'' part:} Without loss of generality, we can assume 
\begin{math}
\epsilon_q = -1
\end{math}. Take any 
\begin{math}
w \in C^+ \left( g \right)^\circ
\end{math}. Since 
\begin{math}
\pm e_q \in C^+ \left( g \right)
\end{math}, 
\begin{math}
w_q = 0
\end{math}
holds. Since 
\begin{math}
\pm e_q \pm e_j \in C^+ \left( g \right)
\end{math}
holds for 
\begin{math}
j \in \left\{ q+1, \ldots, n \right\}
\end{math}, 
\begin{math}
w_j = 0
\end{math}
holds. Since 
\begin{math}
- e_j \pm e_q \in C^+ \left( g \right)
\end{math}
holds for 
\begin{math}
j \in \left\{ q-2, q-1 \right\}
\end{math}, 
\begin{math}
w_j \ge 0
\end{math}
holds for 
\begin{math}
j \in \left\{ q-2, q-1 \right\}
\end{math}. This proves 
\begin{math}
w \in L^+ \left( g \right)^\circ
\end{math}. 

\noindent
\textbf{Proof of ``only if'' part:} 
We show the contraposition of the statement, that is, if all of 
\begin{math}
\epsilon_q, \ldots, \epsilon_n
\end{math}
are 
\begin{math}
1
\end{math}, GCQ is violated. 
We take $P\in O(2)$ so that the following equality holds: 
\[
\delta_1x_1^2 + \delta_2x_2^2 +\alpha x_1x_2 = (x_1,x_2){}^tP\begin{pmatrix}
\lambda_1&0\\
0&\lambda_2
\end{pmatrix}P\begin{pmatrix}
x_1\\x_2
\end{pmatrix},
\]
where $\lambda_1,\lambda_2\neq 0$ is the eigenvalues of $\begin{pmatrix}
\delta_1&\alpha/2\\
\alpha/2&\delta_2
\end{pmatrix}$.
The following then holds for $d\in C^+(g)$: 
{\allowdisplaybreaks
\begin{align*}
& \delta_1d_{q-1}^2 + \delta_2d_{q-2}^2 +\alpha d_{q-1}d_{q-2} +\sum_{j=q}^{n} d_j^2 \leq 0\\
\Rightarrow &  \sum_{j=q}^{n} d_j^2 \leq \left|(d_{q-1},d_{q-2}){}^tP\begin{pmatrix}
\lambda_1&0\\
0&\lambda_2
\end{pmatrix}P\begin{pmatrix}
d_{q-1}\\d_{q-2}
\end{pmatrix}\right|\\
\Rightarrow&  \sum_{j=q}^{n} d_j^2 \leq \max\{|\lambda_1|,|\lambda_2|\} \left\| \left(d_{q-1},d_{q-2}\right)\right\|^2.
\end{align*}
}%
Let $R= \sqrt{\max\{|\lambda_1|,|\lambda_2|\}}$. 
For any $d\in C^+(g)$, the inner product $(R(e_{q-1}+e_{q-2})+e_q)\cdot d$ is estimated as follows:
{\allowdisplaybreaks
\begin{align*}
& (R(e_{q-1}+e_{q-2})+e_q)\cdot d\\
=& R(d_{q-1}+d_{q-2}) + d_q \\
\leq & -R(|d_{q-1}|+|d_{q-2}|) +|d_q| & (\because d_{q-1},d_{q-2}\leq 0)\\
\leq & -R\left\|\left(d_{q-1},d_{q-2}\right)\right\| + |d_q| & \left(\because \left\|\left(d_{q-1},d_{q-2}\right)\right\| = \sqrt{d_{q-1}^2+d_{q-2}^2} \leq |d_{q-1}|+|d_{q-2}|\right)\\
\leq & -R \sqrt{\dfrac{\sum_{j=q}^{n} d_j^2}{\max\{|\lambda_1|,|\lambda_2|\}}}+|d_q| \\
\leq & \left(-\frac{R}{\sqrt{\max\{|\lambda_1|,|\lambda_2|\}}}+1\right)|d_q|= 0.
\end{align*}
}%
Thus, $R(e_{q-1}+e_{q-2})+e_q$ is contained in $C^+(g)^\circ$. 
However, it is not in $L^+(g)^\circ$, and thus GCQ is violated. 
\end{proof}

\subsubsection*{The germ of type $(7)$ in Table~\ref{table:generic constraint r=0}}

Let $g = \left(x_1,\ldots, x_{q-1},-\sum_{j=1}^{q-3}x_j+\epsilon_{q-2}(x_{q-2}+\epsilon_{q-1}'x_{q-1})^2 +\epsilon_{q-1}x_{q-1}^3+\sum_{j=q}^{n}\epsilon_jx_j^2\right)$, where $\epsilon_j,\epsilon_{q-1}'\in \{1,-1\}$, $Q =\epsilon_{q-2}(x_{q-2}+\epsilon_{q-1}'x_{q-1})^2+\sum_{j=q}^{n} \epsilon_j x_j^2$, and $R=\epsilon_{q-1}x_{q-1}^3$. 

\begin{proposition}\label{prop:tangent cone type (7) r=0}

The tangent cone $C^+(g)$ is equal to $\overline{D}(q-3,Q)\setminus \{(0,\ldots,0,d_{q-1},0,\ldots, 0)\in \R^n~|~d_{q-1}< 0\}$ if $\epsilon_{q-1}=-1$ and $\epsilon_{q-2}=\epsilon_{q}=\cdots = \epsilon_n=1$, and $C^+(g)=\overline{D}(q-3,Q)$ otherwise. 

\end{proposition}

\begin{proof}
By Corollary~\ref{cor:tangent cone for r=0} (with $P(x')=\epsilon_{q-2}(x_{q-2}+\epsilon_{q-1}'x_{q-1})^2$ and $s=2$), $\overline{D}(q-3,Q)\setminus \{(0,\ldots,0,d_{q-2},d_{q-1},0,\ldots, 0)\in \R^n~|~\epsilon_{q-1}d_{q-1}^3>0\}$ is contained in $C^+(g)$, and $C^+(g)=\overline{D}(q-3,Q)$ if $\epsilon_j=-1$ for some $j\geq q$. 
In particular, $C^+(g)=\overline{D}(q-3,Q)$ if $\epsilon_{q-1}=1$ since any $d\in \overline{D}(q-3,Q)$ satisfies $d_{q-1}\leq 0$.

In what follows, we assume $\epsilon_{q-1}=-1$ and $\epsilon_{q}=\cdots =\epsilon_n = +1$. 
Let $d=d_{q-2}e_{q-2}+d_{q-1}e_{q-1}\in \overline{D}(q-3,Q)$ with $\epsilon_{q-1}d_{q-1}^3>0\Leftrightarrow d_{q-1}<0$. 
If $\epsilon_{q-2}=-1$ and $d_{q-2}+\epsilon_{q-1}'d_{q-1}\neq 0$, $Q(d')$ is equal to $\epsilon_{q-2}(d_{q-2}+\epsilon_{q-1}'d_{q-1})^2 <0$, and thus $d$ in contained in $C^+(g)$ by Lemma~\ref{lem:tangent cone for r=0}. 
If $\epsilon_{q-2}=-1$ and $d_{q-2}+\epsilon_{q-1}'d_{q-1}=0$, the vector $v=-e_{q-1}$ satisfies the conditions in Lemma~\ref{lem:tangent cone for r=0}.
Indeed, $v_1=\cdots =v_{q-2}=0$, $v_{q-1}=-1<0$, $v'\cdot \nabla Q(d')=0$, and ${}^tv'\mathrm{Hess}(Q)v' =-2 <0$. 
We thus obtain $C^+(g)=\overline{D}(q-2,Q)$ if $\epsilon_{q-2}=-1$. 

If $\epsilon_{q-2}=1$, any $x\in M(g)$ satisfies the following inequality: 
\begin{align*}
&0\geq -\sum_{j=1}^{q-3}x_j+(x_{q-2}+\epsilon_{q-1}'x_{q-1})^2-x_{q-1}^3+\sum_{j=q}^{n}x_j^2\\
\Leftrightarrow& x_{q-1}^3\geq -\sum_{j=1}^{q-1}x_j+(x_{q-2}+\epsilon_{q-1}'x_{q-1})^2+\sum_{j=q}^{n}x_j^2\geq 0. 
\end{align*}
Thus, any $d\in C^+(g)$ also satisfies the inequality $d_{q-1}^3\geq 0\Leftrightarrow d_{q-1}\geq 0$. 
\end{proof}

\begin{proposition}\label{prop:ACQ type (7) r=0}

ACQ holds for $g$ if and only if $\epsilon_{q-2}=\epsilon_q=\cdots =\epsilon_n=-1$, 

\end{proposition}

\begin{proof}
Suppose that the condition in the proposition does not hold, say $\epsilon_j=1$ for $j=q-2$ or $j\geq q$. 
It is easy to check that $-e_j$ is contained in $L^+(g)$ but not in $\overline{D}(q-3,Q)$, in particular $C^+(g)\subset \overline{D}(q-3,Q)\subsetneq L^+(g)$.  
Suppose conversely that the condition in the proposition hold. 
By Proposition \ref{prop:tangent cone type (7) r=0}, the tangent cone $C^+(g)$ is equal to $\overline{D}(q-3,Q)$, which is equal to $\{d\in \R^n~|~d_1=\cdots =d_{q-3}=0,d_{q-2},d_{q-1}\leq 0\}=L^+(g)$ by the assumption. 
\end{proof}

\begin{proposition}

GCQ holds for $g$ if and only if one of 
\begin{math}
\epsilon_q, \ldots, \epsilon_n
\end{math}
is 
\begin{math}
-1
\end{math}.

\end{proposition}

\begin{proof}
It is easy to check that $L^+ \left(g\right)^\circ$ is equal to 
\begin{equation}
\left\{ w \in \mathbb{R}^n \middle| w_{q-2} \ge 0, w_{q-1} \ge 0, w_q =\cdots = w_n=0\right\}.
\end{equation}

\noindent
\textbf{Proof of ``if'' part:} Without loss of generality, we can assume 
\begin{math}
\epsilon_q = -1
\end{math}. Take any 
\begin{math}
w \in C^+ \left( g \right)^\circ
\end{math}. Since 
\begin{math}
\pm e_q \in C^+ \left( g \right)
\end{math}, 
\begin{math}
w_q = 0
\end{math}
holds. Since 
\begin{math}
\pm e_q \pm e_j \in C^+ \left( g \right)
\end{math}
holds for 
\begin{math}
j \in \left\{ q+1, \ldots, n \right\}
\end{math}, 
\begin{math}
w_j = 0
\end{math}
holds. Since 
\begin{math}
- e_j \pm e_q \in C^+ \left( g \right)
\end{math}
holds for 
\begin{math}
j \in \left\{ q-2, q-1 \right\}
\end{math}, 
\begin{math}
w_j \ge 0
\end{math}
holds for 
\begin{math}
j \in \left\{ q-2, q-1 \right\}
\end{math}. This proves 
\begin{math}
w \in L^+ \left( g \right)^\circ
\end{math}. 

\noindent
\textbf{Proof of ``only if'' part:} We show the contraposition of the statement, that is, if all of 
\begin{math}
\epsilon_q, \ldots, \epsilon_n
\end{math}
are 
\begin{math}
1
\end{math}, GCQ is violated. 
Since the eigenvalues of 
\begin{math}
\begin{pmatrix}
\epsilon_{q-2} & \epsilon_{q-2} \epsilon_{q-1}' \\
\epsilon_{q-2} \epsilon_{q-1}' & \epsilon_{q-2}
\end{pmatrix}
\end{math}
are $0$ and $2\epsilon_{q-2}$, we can obtain $\sum_{j=q}^n d_j^2 \le 2\|(d_{q-2},d_{q-1}) \|^2$ for $d\in C^+(g)$ in the same way as in the proof of Proposition~\ref{prop:GCQ type (6) r=0}.
Let $w = \sqrt{2}(e_{q-2}+e_{q-1})+e_q$.
As in the proof of Proposition~\ref{prop:GCQ type (6) r=0}, we obtain the following inequality for any $d\in C^+(g)$:
\begin{equation}
w \cdot d = \sqrt{2} \left( d_{q-2} + d_{q-1} \right) + d_q \leq \left(-\frac{\sqrt{2}}{\sqrt{|2\epsilon_{q-2}|}}+1\right)|d_q| = 0.
\end{equation}
This proves 
\begin{math}
w \in C^+ \left( g \right)^\circ
\end{math}. Since 
\begin{math}
w
\end{math}
is not contained in 
\begin{math}
L^+ \left( g \right)^\circ
\end{math}, GCQ is violated. 
\end{proof}

\subsubsection*{The germ of type $(8)$ in Table~\ref{table:generic constraint r=0}}

Let $g = \left(x_1,\ldots, x_{q-1},-\sum_{j=1}^{q-3}x_j+\epsilon_{q-2}x_{q-2}^3+\epsilon_{q-1}x_{q-1}^2 +\epsilon_{q-1}'x_{q-2}x_{q-1}+\sum_{j=q}^{n}\epsilon_jx_j^2\right)$, where $\epsilon_j,\epsilon_{q-1}'\in \{1,-1\}$, $Q =\epsilon_{q-1}x_{q-1}^2 +\epsilon_{q-1}'x_{q-2}x_{q-1}+\sum_{j=q}^{n} \epsilon_j x_j^2$, and $R=\epsilon_{q-2}x_{q-2}^3$. 

\begin{proposition}\label{prop:tangent cone type (8) r=0}

The tangent cone $C^+(g)$ is equal to $\overline{D}(q-3,Q)\setminus \{(0,\ldots,0,d_{q-2},0,\ldots, 0)\in \R^n~|~d_{q-2}< 0\}$ if $\epsilon_{q-2}=-1$ and $\epsilon_{q-1}'=\epsilon_{q}=\cdots = \epsilon_n=1$, and $C^+(g)=\overline{D}(q-3,Q)$ otherwise. 

\end{proposition}

\begin{proof}
By Corollary~\ref{cor:tangent cone for r=0} (with $P(x')=\epsilon_{q-1}x_{q-1}^2 +\epsilon_{q-1}'x_{q-2}x_{q-1}$ and $s=2$), $\overline{D}(q-3,Q)\setminus \{(0,\ldots,0,d_{q-2},d_{q-1},0,\ldots, 0)\in \R^n~|~\epsilon_{q-2}d_{q-2}^3>0\}$ is contained in $C^+(g)$, and $C^+(g)=\overline{D}(q-3,Q)$ if $\epsilon_j=-1$ for some $j\geq q$. 
In particular, $C^+(g)=\overline{D}(q-3,Q)$ if $\epsilon_{q-2}=1$ since any $d\in \overline{D}(q-3,Q)$ satisfies $d_{q-2}\leq 0$.

In what follows, we assume $\epsilon_{q-2}=-1$ and $\epsilon_{q}=\cdots =\epsilon_n = +1$. 
Let $d=d_{q-2}e_{q-2}+d_{q-1}e_{q-1}\in \overline{D}(q-3,Q)$ with $\epsilon_{q-2}d_{q-2}^3>0\Leftrightarrow d_{q-2}<0$. 
Since $d\in \overline{D}(q-3,Q)$, $Q(d') = d_{q-1}(\epsilon_{q-1}d_{q-1}+\epsilon_{q-1}'d_{q-2})$ is less than or equal to $0$. 
By Lemma~\ref{lem:tangent cone for r=0}, $d$ is contained in $C^+(g)$ if $d_{q-1}$ and $\epsilon_{q-1}d_{q-1}+\epsilon_{q-1}'d_{q-2}$ are not equal to $0$. 
If $\epsilon_{q-1}d_{q-1}+\epsilon_{q-1}'d_{q-2}=0$, $d_{q-1}$ is not $0$ since $d_{q-2}<0$, and $\epsilon_{q-1}$ is not equal to $\epsilon_{q-1}'$ since $d_{q-2},d_{q-1}<0$. 
The gradient $\nabla Q (d')$ is calculated as follows: 
\[
\nabla Q (d') = \left(\epsilon_{q-1}'d_{q-1}, \epsilon_{q-1}'d_{q-2}+2\epsilon_{q-1}d_{q-1},0,\ldots, 0\right) = \left(\epsilon_{q-1}'d_{q-1}, \epsilon_{q-1}d_{q-1},0,\ldots, 0\right).
\]
In particular, both the 1st and 2nd components of $\nabla Q(d')$ are not $0$, and their signs are opposite.
Thus, there exists $v\in \R^n$ with $v_j\leq 0$ for $j\leq q-1$ and $v'\cdot \nabla Q(d') <0$.  
By Lemma~\ref{lem:tangent cone for r=0}, $d$ is contained in $C^+(g)$. 
If $d_{q-1}=0$ and $\epsilon_{q-1}'=-1$, the vector $v=-e_{q-1}$ satisfies the conditions in Lemma~\ref{lem:tangent cone for r=0}.
Indeed, $v_1=\cdots =v_{q-2}=0$, $v_{q-1}=-1<0$, $v'\cdot \nabla Q(d')=d_{q-2}<0$. 

So far we have shown that
\[
\overline{D}(q-3,Q)\setminus \{(0,\ldots,0,d_{q-2},0,\ldots, 0)\in \R^n~|~d_{q-2}< 0\} \subset C^+(g),
\]
and $C^+(g) = \overline{D}(q-3,Q)$ if $\epsilon_{q-1}'=-1$. 
We will show that $d= (0,\ldots,0,d_{q-2},0,\ldots, 0)$ ($d_{q-2}<0$) is not contained in $C^+(g)$ if $\epsilon_{q-1}'=1$. 
Suppose contrary that $d$ is in $C^+(g)$. 
Take $t_m>0$ and $x^{(m)}\in M(g)$ so that $\lim_{m\to \infty} x^{(m)}=0$ and $\lim_{m\to \infty}t_m\cdot x^{(m)}=d$.
The following inequality holds: 
\begin{align*}
&0 \geq -\sum_{j=1}^{q-3}x_j^{(m)}-\left(x_{q-2}^{(m)}\right)^3+\epsilon_{q-1}\left(x_{q-1}^{(m)}\right)^2+x_{q-2}^{(m)}x_{q-1}^{(m)}+\sum_{j=q}^{n}(x_j^{(m)})^2\\
\Leftrightarrow&\left(x_{q-2}^{(m)}\right)^3\geq -\sum_{j=1}^{q-3}x_j^{(m)}+\frac{x_{q-1}^{(m)}}{t_m}\left(\epsilon_{q-1}t_mx_{q-1}^{(m)}+t_mx_{q-2}^{(m)}\right)+\sum_{j=q}^{n}(x_j^{(m)})^2.
\end{align*}
Since $\lim_{m\to \infty}t_m\cdot x^{(m)}=d$, $t_mx_{q-1}^{(m)}$ and $t_mx_{q-2}^{(m)}$ tend to $0$ and $d_{q-2}$, respectively, as $m\to \infty$.  
Hence $\epsilon_{q-1}t_mx_{q-1}^{(m)}+t_mx_{q-2}^{(m)}$ is less than $0$ for $m\gg 0$. 
We thus obtain $x_{q-2}^{(m)}\geq 0$, contradicting $0>d_{q-2}=\lim_{m\to \infty}t_m\cdot x_{q-2}^{(m)}$. 
We eventually obtain:
\[
\overline{D}(q-3,Q)\setminus \{(0,\ldots,0,d_{q-2},0,\ldots, 0)\in \R^n~|~d_{q-2}< 0\} =C^+(g)
\]
if $\epsilon_{q	-1}'=1$.
\end{proof}

\begin{proposition}\label{prop:ACQ type (8) r=0}

ACQ holds for $g$ if and only if $\epsilon_{q-1}=\epsilon_{q-1}'=\epsilon_q=\cdots =\epsilon_n=-1$, 

\end{proposition}

\begin{proof}
If $\epsilon_j=1$ for $j\geq q-1$, the vector $-e_j$ is contained in $L^+(g)$ but not in $\overline{D}(q-3,Q)$.
If $\epsilon_{q-1}'=-1$, the vector $-2e_{q-2}-e_{q-1}$ is contained in $L^+(g)$ but not in $\overline{D}(q-3,Q)$.
In each case, $C^+(g)\subset \overline{D}(q-3,Q)$ is a proper subset of $L^+(g)$.  

Suppose that the condition in the proposition hold. 
By Proposition \ref{prop:tangent cone type (8) r=0}, the tangent cone $C^+(g)$ is equal to $\overline{D}(q-3,Q)$, which is equal to $\{d\in \R^n~|~d_1=\cdots =d_{q-3}=0,d_{q-2},d_{q-1}\leq 0\}=L^+(g)$ by the assumption. 
\end{proof}

\begin{proposition}

GCQ holds for $g$ if and only if one of 
\begin{math}
\epsilon_q, \ldots, \epsilon_n
\end{math}
is
\begin{math}
-1
\end{math}.

\end{proposition}

\begin{proof}
It is easy to check that $L^+ \left(g\right)^\circ$ is equal to 
\begin{equation}
\left\{ w \in \mathbb{R}^n \middle| w_{q-2} \ge 0, w_{q-1} \ge 0, w_q =\cdots = w_n=0\right\}.
\end{equation}

\noindent
\textbf{Proof of ``if'' part:} Without loss of generality, we can assume 
\begin{math}
\epsilon_q = -1
\end{math}. Take any 
\begin{math}
w \in C^+ \left( g \right)^\circ
\end{math}. Since 
\begin{math}
\pm e_q \in C^+ \left( g \right)
\end{math}, 
\begin{math}
w_q = 0
\end{math}
holds. Since 
\begin{math}
\pm e_q \pm e_j \in C^+ \left( g \right)
\end{math}
holds for 
\begin{math}
j \in \left\{ q+1, \ldots, n \right\}
\end{math}, 
\begin{math}
w_j = 0
\end{math}
holds. Since 
\begin{math}
- e_j \pm e_q \in C^+ \left( g \right)
\end{math}
holds for 
\begin{math}
j \in \left\{ q-2, q-1 \right\}
\end{math}, 
\begin{math}
w_j \ge 0
\end{math}
holds for 
\begin{math}
j \in \left\{ q-2, q-1 \right\}
\end{math}. This proves 
\begin{math}
w \in L^+ \left( g \right)^\circ
\end{math}. 

\noindent
\textbf{Proof of ``only if'' part:} We show the contraposition of the statement, that is, if all of 
\begin{math}
\epsilon_q, \ldots, \epsilon_n
\end{math}
are 
\begin{math}
1
\end{math}, GCQ is violated. Let 
\begin{math}
\lambda_1, \lambda_2 \neq 0
\end{math}
be the eigenvalues of 
\begin{math}
\begin{pmatrix}
0 & \epsilon_{q-1}'/2 \\
\epsilon_{q-1}'/2 & \epsilon_{q-1}
\end{pmatrix}
\end{math}.
We can show that \begin{math}
w = R\left( e_{q-2} + e_{q-1} \right) + e_q
\end{math}
is contained in 
\begin{math}
C^+ \left( g \right)^\circ
\end{math}
for 
\begin{math}
R = \sqrt{\max \left\{ \left| \lambda_1 \right|, \left| \lambda_2 \right| \right\}}
\end{math} in the same way as in the proof of Proposition~\ref{prop:GCQ type (6) r=0}. 
Since 
\begin{math}
w
\end{math}
is not contained in 
\begin{math}
L^+ \left( g \right)^\circ
\end{math}, GCQ is violated. 
\end{proof}

\subsubsection*{The germ of type $(9)$ in Table~\ref{table:generic constraint r=0}}

Let $g = \left(x_1,\ldots, x_{q-1},-\sum_{j=1}^{q-3}x_j+x_q^3+\epsilon_{01}x_qx_{q-1}+\epsilon_{02}x_qx_{q-2}+\epsilon_{12}x_{q-1}x_{q-2}+\sum_{j=q+1}^{n}\epsilon_jx_j^2\right)$, where $\epsilon_j,\epsilon_{ij}\in \{1,-1\}$, $Q =\epsilon_{01}x_qx_{q-1}+\epsilon_{02}x_qx_{q-2}+\epsilon_{12}x_{q-1}x_{q-2}+\sum_{j=q+1}^{n}\epsilon_jx_j^2$, and $R=x_{q}^3$. 

\begin{proposition}\label{prop:tangent cone type (9) r=0}

The tangent cone $C^+(g)$ is equal to $\overline{D}(q-3,Q)\setminus \{(0,\ldots,0,d_{q},0,\ldots, 0)\in \R^n~|~d_{q}> 0\}$ if $\epsilon_{01}=\epsilon_{02}=-1$ and $\epsilon_{q+1}=\cdots = \epsilon_n=1$, and $C^+(g)=\overline{D}(q-3,Q)$ otherwise. 

\end{proposition}

\begin{proof}
By Corollary~\ref{cor:tangent cone for r=0} (with $P(x')=\epsilon_{01}x_qx_{q-1}+\epsilon_{02}x_qx_{q-2}+\epsilon_{12}x_{q-1}x_{q-2}$ and $s=3$), $\overline{D}(q-3,Q)\setminus \{(0,\ldots,0,d_{q-2},d_{q-1},d_q,0,\ldots, 0)\in \R^n~|~d_{q}^3>0\}$ is contained in $C^+(g)$, and $C^+(g)=\overline{D}(q-3,Q)$ if $\epsilon_j=-1$ for some $j\geq q+1$.

In what follows, we assume $\epsilon_{q+1}=\cdots =\epsilon_n = +1$. 
Let $d=d_{q-2}e_{q-2}+d_{q-1}e_{q-1}+d_qe_q\in \overline{D}(q-3,Q)$ with $d_{q}^3>0\Leftrightarrow d_{q}>0$. 
Since $d\in \overline{D}(q-3,Q)$, $Q(d')$ is less than or equal to $0$. 
By Lemma~\ref{lem:tangent cone for r=0}, $d$ is contained in $C^+(g)$ if $Q(d')<0$. 
Suppose that $Q(d') = \frac{1}{2}({}^td' \mathrm{Hess}(Q) d')$ is equal to $0$.
If $(\nabla Q(d'))_3 =\epsilon_{01}d_{q-1}+\epsilon_{02}d_{q-2}$ is not $0$, the vector $v = d - \mathrm{sign}(\epsilon_{01}d_{q-1}+\epsilon_{02}d_{q-2})e_q$ satisfies the conditions in Lemma~\ref{lem:tangent cone for r=0}, where $ \mathrm{sign}(\epsilon_{01}d_{q-1}+\epsilon_{02}d_{q-2})$ is $+1$ (resp.~$-1$) if $\epsilon_{01}d_{q-1}+\epsilon_{02}d_{q-2}$ is positive (resp.~negative). 
Indeed, $v_1=\cdots =v_{q-1}=0$, and 
\[
v'\cdot \nabla Q(d')={}^t d'\mathrm{Hess}(Q)d'- \mathrm{sign}(\epsilon_{01}d_{q-1}+\epsilon_{02}d_{q-2}) (\nabla Q(d'))_3<0. 
\]
If $(\nabla Q(d'))_3=\epsilon_{01}d_{q-1}+\epsilon_{02}d_{q-2}=0$, either $d_{q-1}=d_{q-2}=0$ or $\epsilon_{01}$ is not equal to $\epsilon_{02}$.
If $(d_{q-2},d_{q-1})\neq (0,0)$ and the latter condition holds, $d_{q-1}$ is equal to $d_{q-2}$, in particular both are not equal to $0$. 
Since $\mathrm{Hess}(Q)$ is regular and $d'\neq 0$, either the 1st or the 2nd component of $\nabla(Q)(d')=\mathrm{Hess}(Q)d'$ is not $0$. 
If $(\nabla Q(d'))_1\neq 0$, the vector $v = d - \varepsilon\mathrm{sign}((\nabla Q(d'))_1)e_{q-2}$ satisfies the conditions in Lemma~\ref{lem:tangent cone for r=0} for $0<\varepsilon \ll 1$. 
Indeed, $v_1=\cdots =v_{q-3}=0$, $v_{q-1}=d_{q-1}< 0$, $v_{q-2} = d_{q-2}-\varepsilon\mathrm{sign}((\nabla Q(d'))_1) <0$ for $\varepsilon \ll 1$, and 
\[
v'\cdot \nabla Q(d')={}^t d'\mathrm{Hess}(Q)d'- \varepsilon\mathrm{sign}((\nabla Q(d'))_1) (\nabla Q(d'))_1<0. 
\]
By Lemma~\ref{lem:tangent cone for r=0}, $d$ is contained in $C^+(g)$.
One can also show that $d\in C^+(g)$ if  $(\nabla Q(d'))_2\neq 0$ in the same manner. 

So far, we have shown the following inclusion:
\[
\overline{D}(q-3,Q)\setminus \{(0,\ldots,0,d_q,0,\ldots, 0)\in \R^n~|~d_{q}> 0\}\subset C^+(g). 
\]
Let $d=(0,\ldots,0,d_q,0,\ldots, 0)$ with $d_q>0$. 
Since $\nabla Q(d') = (\epsilon_{02}d_q,\epsilon_{01}d_q,0,\ldots, 0)$, there exists $v\in \R^n$ with $v_j\leq 0$ for $j\leq q-1$ and $v'\cdot \nabla Q(d')<0$ if either $\epsilon_{02}$ or $\epsilon_{01}$ is $1$. 
Thus, $C^+(g)=\overline{D}(q-3,Q)$ unless $\epsilon_{02}=\epsilon_{01}=-1$. 
Suppose that $\epsilon_{02}=\epsilon_{01}=-1$ and $d$ is contained in $C^+(g)$. 
Take $t_m>0$ and $x^{(m)}\in M(g)$ so that $\lim_{m\to \infty} x^{(m)}=0$ and $\lim_{m\to \infty}t_m\cdot x^{(m)}=d$.
The following inequality holds: 
\begin{align*}
&0 \geq -\sum_{j=1}^{q-3}x_j^{(m)}+\left(x_{q}^{(m)}\right)^3-x_{q}^{(m)}x_{q-2}^{(m)}-x_{q}^{(m)}x_{q-1}^{(m)}+\epsilon_{12}x_{q-1}^{(m)}x_{q-2}^{(m)}+\sum_{j=q+1}^{n}(x_j^{(m)})^2\\
\Leftrightarrow&
\left(x_{q}^{(m)}\right)^3\leq\sum_{j=1}^{q-3}x_j^{(m)}+\frac{x_{q-2}^{(m)}}{t_m}\left(t_mx_{q}^{(m)}-\epsilon_{12}t_mx_{q-1}^{(m)}\right)+x_{q}^{(m)}x_{q-1}^{(m)}-\sum_{j=q+1}^{n}(x_j^{(m)})^2. 
\end{align*}
Since $\lim_{m\to \infty}t_m\cdot x^{(m)}=d$, $t_mx_{q-1}^{(m)}$ and $t_mx_{q}^{(m)}$ tend to $0$ and $d_{q}$, respectively, as $m\to \infty$.  
Hence both $t_mx_{q}^{(m)}-\epsilon_{12}t_mx_{q-1}^{(m)}$ and $x_q^{(m)}$ are larger than $0$ for $m\gg 0$. 
We thus obtain $x_{q}^{(m)}\leq 0$ when $m\gg 0$, contradicting the fact $x_{q}^{(m)}>0$. 
We eventually obtain:
\[
\overline{D}(q-3,Q)\setminus \{(0,\ldots,0,d_{q},0,\ldots, 0)\in \R^n~|~d_{q}> 0\} =C^+(g)
\]
if $\epsilon_{01}=\epsilon_{02}=-1$.
\end{proof}

\begin{proposition}\label{prop:ACQ type (9) r=0}

ACQ does not hold for $g$. 

\end{proposition}

\begin{proof}
The vector $\epsilon_{01}e_q - e_{q-1}$ is contained in $L^+(g)$ but not in $\overline{D}(q-3,Q)$, in particular $C^+(g)\subset \overline{D}(q-3,Q)\subsetneq L^+(g)$.  
\end{proof}

\begin{proposition}

GCQ holds for $g$ if and only if either
\begin{enumerate}
\item one of 
\begin{math}
\epsilon_{q+1}, \ldots, \epsilon_n
\end{math}
is 
\begin{math}
-1
\end{math},
or 
\item 
\begin{math}
\epsilon_{01}
\end{math}
or 
\begin{math}
\epsilon_{02}
\end{math}
is 
\begin{math}
1
\end{math}.
\end{enumerate}

\end{proposition}

\begin{proof}
It is easy to check that $L^+ \left(g\right)^\circ$ is equal to 
\begin{equation}
\left\{ w \in \mathbb{R}^n \middle| w_{q-2} \ge 0, w_{q-1} \ge 0, w_q =\cdots = w_n=0\right\}.
\end{equation}
First we show that 
\begin{math}
\overline{D} \left( q-3, Q \right)^\circ \subset L^+ \left( g \right)^\circ
\end{math}. Take any 
\begin{math}
w \in \overline{D} \left( q-3, Q \right)^\circ
\end{math}. Since 
\begin{math}
\pm e_q \in \overline{D} \left( q-3, Q \right)^\circ
\end{math}, 
\begin{math}
w_q = 0
\end{math}
holds. Since 
\begin{math}
- e_j \in \overline{D} \left( q-3, Q \right)^\circ
\end{math}
holds for 
\begin{math}
j \in \left\{ q-2, q-1 \right\}
\end{math}, 
\begin{math}
w_j \ge 0
\end{math}
holds for 
\begin{math}
j \in \left\{ q-2, q-1 \right\}
\end{math}.
Since 
\begin{math}
\displaystyle - \frac{1}{s} e_{q-1} + \epsilon_{01} \epsilon_j s e_q \pm e_j \in \overline{D} \left( q-3, Q \right)^\circ
\end{math}
holds for 
\begin{math}
s > 0
\end{math}
and
\begin{math}
j \in \left\{ q+1, \ldots, n \right\}
\end{math}, 
\begin{math}
\displaystyle - \frac{1}{s} w_{q-1} \pm w_j \le 0
\end{math}
holds. Since this holds for arbitrarily large 
\begin{math}
s > 0
\end{math}, we obtain 
\begin{math}
w_j = 0
\end{math}
for 
\begin{math}
j \in \left\{ q+1, \ldots, n \right\}
\end{math}. This proves
\begin{math}
w \in L^+ \left( g \right)^\circ
\end{math}.

Therefore, if 
\begin{math}
C^+ \left( g \right) = \overline{D} \left( q-3, Q \right)
\end{math}
holds, GCQ holds. By Proposition~\ref{prop:tangent cone type (9) r=0}, this holds if either one of $\epsilon_{01}, \epsilon_{02}$ is $1$, or one of $\epsilon_{q+1}, \ldots, \epsilon_n$ is $-1$.
In what follows, we consider the case that 
\begin{math}
\epsilon_{01} = \epsilon_{02} = -1
\end{math}
and 
\begin{math}
\epsilon_{q+1} = \cdots = \epsilon_n = 1
\end{math}
and prove that GCQ is violated in this case. 

Suppose
\begin{math}
\epsilon_{12} = 1
\end{math}. We claim that 
\begin{math}
w = e_q
\end{math}
is contained in 
\begin{math}
C^+ \left( g \right)^\circ
\end{math}. Take any 
\begin{math}
d \in C^+ \left( g \right)
\end{math}. Then, 
\begin{equation}
- d_qd_{q-1} - d_qd_{q-2}+d_{q-1}d_{q-2}+\sum_{j=q+1}^{n} d_j^2 \le 0
\end{equation}
holds. If 
\begin{math}
d_{q-1} + d_{q-2} = 0
\end{math}, 
\begin{math}
d_{q+1} = \cdots =d_n = 0
\end{math}
holds and thus 
\begin{math}
d_q < 0
\end{math}
holds by Proposition~\ref{prop:tangent cone type (9) r=0}. In this case, 
\begin{math}
w \cdot d = d_q < 0
\end{math}
holds. If 
\begin{math}
d_{q-1} + d_{q-2} < 0
\end{math}, 
\begin{equation}
d_q \le \frac{- d_{q-1}d_{q-2} - \sum_{j=q+1}^{n} d_j^2}{-d_{q-1}-d_{q-2}} < 0
\end{equation}
holds and thus 
\begin{math}
w \cdot d = d_q < 0
\end{math}
holds. Since 
\begin{math}
w
\end{math}
is not contained in 
\begin{math}
L^+ \left( g \right)^\circ
\end{math}, this proves that GCQ is violated in this case.

Suppose
\begin{math}
\epsilon_{12} = -1
\end{math}. We claim that 
\begin{math}
w = e_{q-2} + e_{q-1} + e_q
\end{math}
is contained in 
\begin{math}
C^+ \left( g \right)^\circ
\end{math}.
Take any 
\begin{math}
d \in C^+ \left( g \right)
\end{math}. Then, 
\begin{equation}
- d_qd_{q-1} - d_qd_{q-2}-d_{q-1}d_{q-2}+\sum_{j=q+1}^{n} d_j^2 \le 0
\end{equation}
holds. If 
\begin{math}
d_{q-1} + d_{q-2} = 0
\end{math}, 
\begin{math}
d_{q+1} = \cdots =d_n = 0
\end{math}
holds and thus 
\begin{math}
d_q < 0
\end{math}
holds by Proposition~\ref{prop:tangent cone type (9) r=0}. In this case, 
\begin{math}
w \cdot d = d_q < 0
\end{math}
holds. If 
\begin{math}
d_{q-1} + d_{q-2} < 0
\end{math}, 
\begin{equation}
d_{q-2} + d_{q-1} + d_q \le \frac{-d_{q-2}^2 - d_{q-1}^2 - d_{q-1}d_{q-2} - \sum_{j=q+1}^{n} d_j^2}{-d_{q-1}-d_{q-2}} < 0
\end{equation}
holds and thus 
\begin{math}
w \cdot d = d_{q-2} + d_{q-1} + d_q < 0
\end{math}
holds. Since 
\begin{math}
w
\end{math}
is not contained in 
\begin{math}
L^+ \left( g \right)^\circ
\end{math}, this proves that GCQ is violated in this case. 
\end{proof}

\subsubsection*{The germ of type $(10)$ in Table~\ref{table:generic constraint r=0}}

We put 
\begin{align*}
g =& \left(x_1,\ldots, x_{q-1},-\sum_{j=1}^{q-4}x_j+\sum_{j=1}^{3}\delta_jx_{q-4+j}^2+\sum_{1\leq i<j\leq 3}\alpha_{ij}x_{q-4+i}x_{q-4+j}\right. \\
&\left. \hspace{2em}+\epsilon_{0}x_{q-3}x_{q-2}x_{q-1}+\sum_{j=q}^{n}\epsilon_jx_j^2\right),
\end{align*}
where $\alpha\in \R$ and $\delta_j\in \{1,-1\}$ satisfy the condition ($\ast\ast$) in Table~\ref{table:generic constraint r=0} and $\epsilon_j\in \{1,-1\}$, $Q =+\sum_{j=1}^{3}\delta_jx_{q-4+j}^2+\sum_{1\leq i<j\leq 3}\alpha_{ij}x_{q-4+i}x_{q-4+j}+\sum_{j=q}^{n}\epsilon_jx_j^2$, and $R=\epsilon_{0}x_{q-3}x_{q-2}x_{q-1}$.

\begin{proposition}\label{prop:tangent cone type (10) r=0}

The tangent cone $C^+(g)$ is equal to $\overline{D}(q-4,Q)$. 

\end{proposition}

\begin{proof}
By Corollary~\ref{cor:tangent cone for r=0} (with $P(x')=\sum_{j=1}^{3}\delta_jx_{q-4+j}^2+\sum_{1\leq i<j\leq 3}\alpha_{ij}x_{q-4+i}x_{q-4+j}$ and $s=3$), $\overline{D}(q-4,Q)\setminus \{(0,\ldots,0,d_{q-3},d_{q-2},d_{q-1},0,\ldots, 0)\in \R^n~|~\epsilon_0d_{q-3}d_{q-2}d_{q-1}>0\}$ is contained in $C^+(g)$, and $C^+(g)=\overline{D}(q-4,Q)$ if $\epsilon_j=-1$ for some $j\geq q$. 
Since $d_{q-3},d_{q-2},d_{q-1}\leq 0$ for any $d\in \overline{D}(q-4,Q)$, $C^+(g)=\overline{D}(q-4,Q)$ if $\epsilon_0=1$.

In what follows, we assume $\epsilon_0=-1$ and $\epsilon_{q}=\cdots =\epsilon_n = +1$. 
Let $d=d_{q-3}e_{q-3}+d_{q-2}e_{q-2}+d_{q-1}e_{q-1}\in \overline{D}(q-4,Q)$ with $d_{q-3}d_{q-2}d_{q-1}<0\Leftrightarrow d_{q-3},d_{q-2},d_{q-1}\neq 0$. 
Since $d\in \overline{D}(q-4,Q)$, $Q(d')$ is less than or equal to $0$. 
By Lemma~\ref{lem:tangent cone for r=0}, $d$ is contained in $C^+(g)$ if $Q(d')<0$. 
Suppose that $Q(d') = \frac{1}{2}({}^td' \mathrm{Hess}(Q) d')$ is equal to $0$.
The matrix $\mathrm{Hess}(Q)$ is regular since $\delta_j$ and $\alpha_{ij}$ satisfy the condition ($\ast\ast$) is Table~\ref{table:generic constraint r=0}.
Since $d'\neq 0$, one of the components of $\nabla(Q)(d')=\mathrm{Hess}(Q)d'$, say the $\ell$-th component ($\ell\in \{1,2,3\}$), is not $0$. 
The vector $v = d - \varepsilon\mathrm{sign}((\nabla Q(d'))_\ell)e_{q-4+\ell}$ satisfies the conditions in Lemma~\ref{lem:tangent cone for r=0} for $0<\varepsilon \ll 1$. 
Indeed, $v_1=\cdots =v_{q-4}=0$, $v_{q-4+j}=d_{q-4+j}< 0$ for $j\neq \ell$, $v_{q-4+\ell} = d_{q-4+\ell}-\varepsilon\mathrm{sign}((\nabla Q(d'))_\ell) <0$ for $0<\varepsilon \ll 1$, and 
\[
v'\cdot \nabla Q(d')={}^t d'\mathrm{Hess}(Q)d'- \varepsilon\mathrm{sign}((\nabla Q(d'))_\ell) (\nabla Q(d'))_\ell<0. 
\]
By Lemma~\ref{lem:tangent cone for r=0}, $d$ is contained in $C^+(g)$.
\end{proof}

\begin{proposition}\label{prop:ACQ type (10) r=0}

ACQ holds for $g$ if and only if 
\begin{math}
\delta_1 = \delta_2 = \delta_3 = -1
\end{math}, 
\begin{math}
\epsilon_q = \cdots = \epsilon_n = -1
\end{math}
and there exist distinct indices $i,j,k \in \{1,2,3\}$ such that either
\begin{enumerate}
\item $\alpha_{ij}\le 0,\ \alpha_{ik}\le 0,\ \alpha_{jk}<2$, \quad or
\item $0<\alpha_{ij}<2,\ 0<\alpha_{ik}<2,\ 
      \displaystyle \alpha_{jk}+\frac{\alpha_{ij}\alpha_{ik}}{2}
      < 2\sqrt{\Bigl(1-\frac{\alpha_{ij}^2}{4}\Bigr)\Bigl(1-\frac{\alpha_{ik}^2}{4}\Bigr)}$,
\end{enumerate}
%
where we regard that $\alpha_{ij}$ is attached to the unordered pair $\{i,j\}$ (that is, we assume $\alpha_{ji}=\alpha_{ij}$).

\end{proposition}
\begin{proof}
Note that, if
\begin{math}
\delta_1 = \delta_2 = \delta_3 = -1
\end{math}
and
\begin{math}
\epsilon_q = \cdots \epsilon_n = -1
\end{math}
hold, ACQ holds if and only if 
\begin{equation}
Q' \left( d_{q-3}, d_{q-2}, d_{q-1} \right) = - \sum_{j=1}^{3} d_{q-4+j}^2+\sum_{1\leq i<j\leq 3}\alpha_{ij}d_{q-4+i}d_{q-4+j} \le 0
\end{equation}
holds for all 
\begin{math}
d_{q-3}, d_{q-2}, d_{q-1} \le 0
\end{math}.

\noindent
\textbf{Proof of ``if'' part:} 
First of all, we show that ACQ holds for the case 1. 
Without loss of generality, we can assume $\alpha_{12},\alpha_{13}\leq 0$ and $\alpha_{23}<2$.
Take any 
\begin{math}
d_{q-3}, d_{q-2}, d_{q-1} \le 0
\end{math}. Then, 
\begin{align}
Q' &= \left( -d_{q-3}^2 + \alpha_{12} d_{q-3} d_{q-2} + \alpha_{13} d_{q-3} d_{q-2} \right) + \left( -d_{q-2}^2-d_{q-1}^2+\alpha_{23} d_{q-2}d_{q-1} \right) \\
&\le -d_{q-2}^2-d_{q-1}^2+\alpha_{23} d_{q-2}d_{q-1} =- \left( d_{q-2} - d_{q-1} \right)^2 + \left( \alpha_{23} - 2 \right) d_{q-2}d_{q-1} \le 0
\end{align}
holds. Therefore, ACQ holds in the case 1. 

Next, we show ACQ holds in the case 2. 
Without loss of generality, we can assume the inequalities in the case 2. hold for $i=1,j=2,k=3$. 
Take any 
\begin{math}
d_{q-3}, d_{q-2}, d_{q-1} \le 0
\end{math}. Then, 
\begin{multline}
Q' = - \left( d_{q-3} - \frac{\alpha_{12}}{2} d_{q-2} - \frac{\alpha_{13}}{2} d_{q-1} \right)^2 \\
- \left( 1-\frac{\alpha_{12}^2}{4} \right) d_{q-2}^2 - \left( 1-\frac{\alpha_{13}^2}{4} \right) d_{q-1}^2 + \left( \alpha_{23} + \frac{\alpha_{12} \alpha_{13}}{2} \right) d_{q-2} d_{q-1} \label{eq:Qtwopos}
\end{multline}
holds. Since the second line of the equation is non-positive for all 
\begin{math}
d_{q-2}, d_{q-1} \le 0
\end{math}
by the condition, ACQ holds.

\noindent
\textbf{Proof of ``only if'' part:} 
The condition
\begin{math}
\delta_1 = \delta_2 = \delta_3 = -1
\end{math}
is necessary for ACQ to hold since if one of 
\begin{math}
\delta_1, \delta_2, \delta_3
\end{math}
is 
\begin{math}
1
\end{math}, say 
\begin{math}
\delta_1 = 1
\end{math}, then, 
\begin{math}
e_{q-3} \in L^+ \left( g \right)
\end{math}
whereas 
\begin{math}
Q \left( e_{q-3} \right) > 0
\end{math}
and thus
\begin{math}
e_{q-3} \notin C^+ \left( g \right)
\end{math}. The condition 
\begin{math}
\epsilon_q = \cdots = \epsilon_n = -1
\end{math}
is also necessary for ACQ to hold since if one of 
\begin{math}
\epsilon_q, \ldots, \epsilon_n
\end{math}
is 
\begin{math}
1
\end{math}, say 
\begin{math}
\epsilon_n = 1
\end{math}, then, 
\begin{math}
e_n \in L^+ \left( g \right)
\end{math}
whereas 
\begin{math}
Q \left( e_n \right) > 0
\end{math}
and thus 
\begin{math}
e_n \notin C^+ \left( g \right)
\end{math}. Therefore, in what follows, we assume 
\begin{math}
\delta_1 = \delta_2 = \delta_3 = -1
\end{math}
and
\begin{math}
\epsilon_q = \cdots = \epsilon_n = -1
\end{math}. 

We show the contraposition of the claim. 
Suppose that the conditions 1. and 2. do not hold for any distinct $i,j,k\in \{1,2,3\}$. 
First, we show that ACQ does not hold if two of 
\begin{math}
\alpha_{12}, \alpha_{13}, \alpha_{23}
\end{math}
is non-positive. 
By symmetry, we can assume $\alpha_{12},\alpha_{13}\leq 0$. 
Since the condition 1. does not hold for $i=1,j=2,k=3$, $\alpha_{23}$ is greater than $2$.
In this case, 
\begin{math}
Q' \left( 0, -1, -1 \right) = \alpha_{23} - 2 > 0
\end{math}
holds and thus ACQ does not hold. In what follows, we assume more than one of 
\begin{math}
\alpha_{12}, \alpha_{13}, \alpha_{23}
\end{math}
is positive. By symmetry, we can assume 
\begin{math}
\alpha_{12}, \alpha_{13} > 0
\end{math}
without loss of generality. Then, Eq.~\eqref{eq:Qtwopos}, as a function of 
\begin{math}
d_{q-3}
\end{math}, attains the maximum 
\begin{equation}
Q'' = - \left( 1-\frac{\alpha_{12}^2}{4} \right) d_{q-2}^2 - \left( 1-\frac{\alpha_{13}^2}{4} \right) d_{q-1}^2 + \left( \alpha_{23} + \frac{\alpha_{12} \alpha_{13}}{2} \right) d_{q-2} d_{q-1}.
\end{equation}
at
\begin{math}
\displaystyle d_{q-3} = \frac{\alpha_{12}}{2} d_{q-2} + \frac{\alpha_{13}}{2} d_{q-3} \le 0
\end{math}. By noting that 
\begin{equation}
\alpha_{12} \neq 2, \alpha_{13} \neq 2, \alpha_{23} + \frac{\alpha_{12} \alpha_{13}}{2} \neq 2 \sqrt{\left( 1-\frac{\alpha_{12}^2}{4} \right) \left( 1-\frac{\alpha_{13}^2}{4} \right)}
\end{equation}
holds, 
\begin{math}
Q'' \le 0
\end{math}
holds for all 
\begin{math}
d_{q-2}, d_{q-1} \le 0
\end{math}
if and only if the condition 2. (for $i=1,j=2,k=3$) hold (cf.~the proof of Proposition~\ref{prop:ACQ type (6) r=0}). 
\end{proof}

\begin{proposition}

GCQ holds for $g$ if and only if one of 
\begin{math}
\epsilon_q, \ldots, \epsilon_n
\end{math}
is 
\begin{math}
-1
\end{math}.

\end{proposition}

\begin{proof}
It is easy to check that $L^+ \left(g\right)^\circ$ is equal to 
\begin{equation}
\left\{ w \in \mathbb{R}^n \middle| w_{q-3} \ge 0, w_{q-2} \ge 0, w_{q-1} \ge 0, w_q =\cdots = w_n=0\right\}.
\end{equation}

\noindent
\textbf{Proof of ``if'' part:} Without loss of generality, we can assume 
\begin{math}
\epsilon_q = -1
\end{math}. Take any 
\begin{math}
w \in C^+ \left( g \right)^\circ
\end{math}. Since 
\begin{math}
\pm e_q \in C^+ \left( g \right)
\end{math}, 
\begin{math}
w_q = 0
\end{math}
holds. Since 
\begin{math}
\pm e_q \pm e_j \in C^+ \left( g \right)
\end{math}
holds for 
\begin{math}
j \in \left\{ q+1, \ldots, n \right\}
\end{math}, 
\begin{math}
w_j = 0
\end{math}
holds. Since 
\begin{math}
- e_j \pm e_q \in C^+ \left( g \right)
\end{math}
holds for 
\begin{math}
j \in \left\{ q-3, q-2, q-1 \right\}
\end{math}, 
\begin{math}
w_j \ge 0
\end{math}
holds for 
\begin{math}
j \in \left\{ q-3, q-2, q-1 \right\}
\end{math}. This proves 
\begin{math}
w \in L^+ \left( g \right)^\circ
\end{math}. 

\noindent
\textbf{Proof of ``only if'' part:} We show the contraposition of the statement, that is, if all of 
\begin{math}
\epsilon_q, \ldots, \epsilon_n
\end{math}
are 
\begin{math}
1
\end{math}, GCQ is violated. Let 
\begin{math}
\lambda_1, \lambda_2, \lambda_3 \neq 0
\end{math}
be the eigenvalues of 
\begin{math}
\begin{pmatrix}
\delta_1 & \alpha_{12}/2 & \alpha_{13}/2 \\
\alpha_{12}/2 & \delta_2 & \alpha_{23}/2 \\
\alpha_{13}/2 & \alpha_{23}/2 & \delta_3 
\end{pmatrix}
\end{math}. We claim that 
\begin{math}
w = R\left( e_{q-1} + e_{q-2} + e_{q-3} \right) + e_q
\end{math}
is contained in 
\begin{math}
C^+ \left( g \right)^\circ
\end{math}
for 
\begin{math}
R = \sqrt{\max \left\{ \left| \lambda_1 \right|, \left| \lambda_2 \right|, \left| \lambda_3 \right| \right\}}
\end{math}. Take arbitrary 
\begin{math}
d \in C^+ \left( g \right)
\end{math}. Then, 
\begin{align}
\sum_{j=q}^n d_j^2 &\le - \left( \sum_{j=1}^{3}\delta_jd_{q-4+j}^2+\sum_{1\leq i<j\leq 3}\alpha_{ij}d_{q-4+i}d_{q-4+j} \right) \\
&\le \max \left\{ \left| \lambda_1 \right|, \left| \lambda_2 \right|, \left| \lambda_3 \right| \right\} \|\left( d_{q-3}, d_{q-2},d_{q-1}\right)\|^2
\end{align}
holds. 
Therefore, we obtain the following inequality:
{\allowdisplaybreaks
\begin{align*}
w \cdot d =& R \left( d_{q-3} + d_{q-2} + d_{q-1} \right) + d_q \\
\leq & -R(|d_{q-1}|+|d_{q-2}|+|d_{q-3}|) +|d_q| & (\because d_{q-1},d_{q-2},d_{q-3}\leq 0)\\
\leq & -R\left\|\left(d_{q-1},d_{q-2},d_{q-3}\right)\right\| + |d_q| & \left(\because \left\|\left(d_{q-1},d_{q-2},d_{q-3}\right)\right\| \leq |d_{q-1}|+|d_{q-2}|+|d_{q-3}|\right)\\
\leq & -R \sqrt{\dfrac{\sum_{j=q}^{n} d_j^2}{\max\{|\lambda_1|,|\lambda_2|,|\lambda_3|\}}}+|d_q| \leq 0.
\end{align*}
}%
This proves 
\begin{math}
w \in C^+ \left( g \right)^\circ
\end{math}. Since 
\begin{math}
w
\end{math}
is not contained in 
\begin{math}
L^+ \left( g \right)^\circ
\end{math}, GCQ is violated. 
\end{proof}

{
\subsubsection{ACQ and GCQ in Table~\ref{table:generic constraint q>0 r=1}}

\begin{proposition}\label{prop:ACQ for q>0 r=1}

ACQ does not hold for any germ in Table~\ref{table:generic constraint q>0 r=1}.

\end{proposition}

\begin{proof}
Let $(g,h) = (x_1,\ldots, x_q,h)$ be a germ in Table~\ref{table:generic constraint q>0 r=1}.
It is easy to check that the linearized cone $L^+(g,h)$ is equal to $\{d\in \R^n~|~d_1,\ldots, d_q\leq 0\}$. 
By Lemma~\ref{lem:tangent cone for r=1}, the tangent cone $C^+(g,h)$ is contained in $\overline{C}(q,Q) =\{d\in \R^n~|~d_1,\ldots, d_q\leq 0, Q(d)=0\}$, where $Q$ is the quadratic part of $h$.
One can easily check that $e_n$ is contained in $L^+(g,h)$ but not in $\overline{C}(q,Q)$, in particular $C^+(g,h) \subset \overline{C}(q,Q)\subsetneq L^+(g,h)$. 
\end{proof}

\subsubsection*{The germ of type $(1,k)$ ($k\geq 3$) in Table~\ref{table:generic constraint q>0 r=1}}

Let $(g,h) = \left(x_1,x_1^k+\sum_{j=2}^{n} \epsilon_j x_j^2\right)$, where $\epsilon_j\in \{1,-1\}$, $Q = \sum_{j=2}^{n} \epsilon_j x_j^2$, and $R=x_1^k$. 

\begin{proposition}\label{prop:tangent cone type (1,k) q,r>0}

The tangent cone $C^+(g,h)$ is equal to $\{0\}$ if all the $\epsilon_j$'s are $(-1)^k$, and $C^+(g,h)$ is equal to $\overline{C}(1,Q)$ otherwise. 

\end{proposition}

\begin{proof}
In the former case, it is easy to see that $M(g,h)$ is equal to $\{0\}$, and thus its tangent cone is also $\{0\}$. 
In the latter case, we can assume that $\epsilon_2=(-1)^{k+1}$ without loss of generality. 
Let $d\in \overline{C}(1,Q)$. 
Since $R_k=x_1^k$ and $R_r=0$ for $r\neq k$, one can deduce from Lemma~\ref{lem:tangent cone for r=1} that $d$ is contained in $C^+(g,h)$ if $d_1=0$.
Suppose that $d_1 <  0$. 
If $(d_2,\ldots, d_n)\neq (0,\ldots,0)$, the vector $v = (-1)^{k+1}\nabla Q(d)$ satisfies the conditions in Lemma~\ref{lem:tangent cone for r=1}. 
Indeed, 
\[
v\cdot \nabla Q(d) = (-1)^{k+1}\| (0,\ldots,2\epsilon_j d_j,\ldots)\| = (-1)^{k+1} \sum_{j=2}^{n}d_j^2, 
\]
whose sign is opposite to $R_k(d) = d_1^k$. 
Thus, $d$ is contained in $C^+(g,h)$ by Lemma~\ref{lem:tangent cone for r=1}.
If $d_2=\cdots = d_n=0$, $\nabla Q (d)=0$ and the vector $v=e_2$ satisfies the conditions in Lemma~\ref{lem:tangent cone for r=1}. 
Indeed, ${}^tv\mathrm{Hess}(Q)v =2\epsilon_2 = 2(-1)^{k+1}$, whose sign is opposite to $R_k(d)$. 
Again, we can deduce from Lemma~\ref{lem:tangent cone for r=1} that $d$ is contained in $C^+(g,h)$. 
\end{proof}

\begin{proposition}

ACQ does not hold for $(g,h)$. 

\end{proposition}

\begin{proof}
It is easy to check that $L^+(g,h)$ is equal to $\{d\in \R^n~|~d_1\leq 0\}$, which contains $\overline{C}(1,Q)=\{d\in \R^n~|~ d_1\leq 0, \sum_{j=2}^{n}\epsilon_jd_j^2=0\}$ as a proper subset. 
Thus, $C^+(g,h)$ is not equal to $L^+(g,h)$ since $C^+(g,h)\subset \overline{C}(1,Q)$. 
\end{proof}

\begin{proposition}\label{prop:GCQ type (1,k) q,r>0}

GCQ holds for $(g,h)$ if and only if $\epsilon_j\epsilon_l=-1$ for some $j,l\in \{2,\ldots, n\}$. 

\end{proposition}

\begin{proof}
It is easy to check that $L^+(g,h)^\circ$ is equal to $\{w\in \R^n~|~w_1\geq 0, w_2=\cdots =w_n=0\}$. 

\noindent
\textbf{Proof of ``if'' part:}
Take $w\in C^+(g,h)^\circ$. 
For any $j\in \{2,\ldots, n\}$, we can take $l\in \{2,\ldots, n\}$ so that $\epsilon_j\epsilon_l=-1$ by the assumption. 
Since $\pm e_j \pm e_l\in C^+(g,h)$, the following inequality holds for any pair of signs:
\[
0 \leq w \cdot (\pm e_j\pm e_l) = \pm w_j \pm w_l. 
\]
We can thus deduce $w_j=0$, and $C^+(g,h)^\circ \subset L^+(g,h)^\circ$. 
The opposite inclusion also holds since $C^+(g,h)\subset L^+(g,h)$. 

\noindent
\textbf{Proof of ``only if'' part:}
If $\epsilon_j=(-1)^k$ for any $j\geq 0$, $C^+(g,h)$ is equal to $\{0\}$ by Proposition~\ref{prop:tangent cone type (1,k) q,r>0}, and its polar is $\R^n \neq L^+(g,h)^\circ$. 
If $\epsilon_j=(-1)^{k+1}$ for any $j\geq 0$, $C^+(g,h)$ is equal to $\overline{C}(1,Q) = \{d\in \R^n~|~d_1\leq 0,d_2=\cdots =d_n=0\}$ by Proposition~\ref{prop:tangent cone type (1,k) q,r>0}, and its polar is $\{w\in \R^n~|~w_1\geq 0\} \neq L^+(g,h)^\circ$. 
\end{proof}

\subsubsection*{The germ of type $(2)$ in Table~\ref{table:generic constraint q>0 r=1}}

Let $(g,h) = \left(x_1,x_2^3+x_1^2 + \sum_{j=3}^{n} \epsilon_j x_j^2\right)$, where $\epsilon_j\in \{1,-1\}$, $Q =x_1^2+ \sum_{j=3}^{n} \epsilon_j x_j^2$, and $R=x_2^3$. 

\begin{proposition}\label{prop:tangent cone type (2) q,r>0}

The tangent cone $C^+(g,h)$ is equal to $\overline{C}(1,Q)\setminus \{(0,d,0,\ldots, 0)\in \R^n~|~d>0\}$ if $\epsilon_3=\cdots = \epsilon_n=1$, and $C^+(g,h)$ is equal to $\overline{C}(1,Q)$ otherwise. 

\end{proposition}

\begin{proof}
The tangent cone $C^+(g,h)$ is contained in $\overline{C}(1,Q)$ by Lemma~\ref{lem:tangent cone for r=1}. 
Let $d\in \overline{C}(1,Q)$. 
Since $R_3 = R = x_2^3$, $d$ is contained in $\underline{C}(1,Q,R)\subset C^+(g,h)$ if $d_2=0$. 
If $d_2\neq 0$ and $d_j$ is not $0$ for $j\geq 3$, the vector $v = -d_2\epsilon_jd_je_j$ satisfies the conditions in Lemma~\ref{lem:tangent cone for r=1} for $d$. 
Indeed, $R_3(d) = d_2^3$ and $v\cdot \nabla Q(d) = -2d_2 \epsilon_j^2d_j^2$, in particular the signs of these values are mutually opposite. 

In what follows, we assume $d_2\neq 0$ and $d_3=\cdots =d_n=0$. 
Since $Q(d) = d_1^2$, $d_1$ is equal to $0$ and $\nabla Q (d)=0$. 
If $d_2<0$, the vector $v=-e_1$ satisfies the conditions in Lemma~\ref{lem:tangent cone for r=1} for $d$ since $v_1=-1<0$, $R_3(d)=d_2^3<0$, $v\cdot \nabla Q(d)=0$, and ${}^tv\mathrm{Hess}(Q)v = 2>0$. 
If $d_2>0$ and some $\epsilon_j$ is $-1$ for $j\geq 3$, the vector $v=e_j$ satisfies the conditions in Lemma~\ref{lem:tangent cone for r=1} for $d$ since $R_3(d)=d_2^3>0$, $v\cdot \nabla Q(d)=0$, and ${}^tv\mathrm{Hess}(Q)v = -2<0$. 
 
So far, we have shown that $\overline{C}(1,Q)\setminus \{(0,d,0,\ldots, 0)\in \R^n~|~d>0\}$ is contained in $C^+(g,h)$, and $C^+(g,h)=\overline{C}(1,Q)$ if $\epsilon_j=-1$ for some $j\geq 3$. 
Lastly, we will show that $(0,d,0,\ldots, 0)$ ($d>0$) is not contained in $C^+(g,h)$ when $\epsilon_3=\cdots =\epsilon_n = 1$. 
Let $x\in M(g,h)$.
Since $h(x)=0$, $x_2^3$ is equal to $-x_1^2 - \sum_{j=3}^{n}x_j^2 \leq 0$, and thus $x_2\leq 0$. 
Thus, $d_2$ is less than or equal to $0$ for any $d\in C^+(g,h)$. 
\end{proof}

\begin{proposition}

GCQ holds for $(g,h)$ if and only if $\epsilon_j\epsilon_l=-1$ for some $j,l\geq 3$. 

\end{proposition}

\begin{proof}
It is easy to check that $L^+(g,h)^\circ$ is equal to $\{w\in \R^n~|~w_1\geq 0, w_2=\cdots =w_n=0\}$. 

\noindent
\textbf{Proof of ``if'' part:}
When $\epsilon_j\epsilon_l=-1$ for some $j,l\geq 3$, one can show that $C^+(g,h)^\circ = L^+(g,h)^\circ$ in the same way as in the proof of Proposition~\ref{prop:GCQ type (1,k) q,r>0}. 

\noindent
\textbf{Proof of ``only if'' part:}
If $\epsilon_j=1$ for any $j\geq 3$, $C^+(g,h)$ is equal to $\overline{C}(1,Q)\setminus \{(0,d,0,\ldots, 0)\in \R^n~|~d>0\} = \{d\in \R^n~|~d_2\leq 0,d_1=d_3=\cdots =d_n=0\}$ by Proposition~\ref{prop:tangent cone type (2) q,r>0}, and its polar is $\{w\in \R^n~|~w_2\geq 0\} \neq L^+(g,h)^\circ$. 
If $\epsilon_j=-1$ for any $j\geq 3$, $(1,0,1,0,\ldots,0)$ is contained in $C^+(g,h)^\circ$. 
Indeed, any $d\in C^+(g,h)=\overline{C}(1,Q)$ satisfies $d_1^2 = \sum_{j=3}^{n}d_j^2 \geq d_3^2$, and thus $(1,0,1,0,\ldots,0)\cdot d = d_1+d_3\leq 0$ since $d_1\leq -|d_3|$. 
Hence $C^+(g,h)^\circ \neq L^+(g,h)^\circ$ in this case. 
\end{proof}

\subsubsection*{The germ of type $(3,k)$ in Table~\ref{table:generic constraint q>0 r=1}}

Let $(g,h) = \left(x_1,x_2^k+\epsilon_1x_1x_2 + \sum_{j=3}^{n} \epsilon_j x_j^2\right)$, where $\epsilon_j\in \{1,-1\}$, $Q =\epsilon_1x_1x_2+ \sum_{j=3}^{n} \epsilon_j x_j^2$, and $R=x_2^k$. 

\begin{proposition}

If $k$ is even and $\epsilon_3=\cdots =\epsilon_n=1$, $C^+(g,h)$ is equal to $\overline{C}(1,Q)\setminus \{(0,d,0,\ldots, 0)\in \R^n~|~\epsilon_1d<0\}$.
If $k$ is odd, $\epsilon_1=-1$, and $\epsilon_3=\cdots =\epsilon_n = \delta$ for some $\delta\in \{1,-1\}$, $C^+(g,h)$ is equal to $\overline{C}(1,Q)\setminus \{(0,d,0,\ldots, 0)\in \R^n~|~\delta d >0\}$. 
In the other cases, $C^+(g,h)$ is equal to $\overline{C}(1,Q)$.

\end{proposition}

\begin{proof}
The tangent cone $C^+(g,h)$ is contained in $\overline{C}(1,Q)$ by Lemma~\ref{lem:tangent cone for r=1}. 
Let $d\in \overline{C}(1,Q)$. 
Since $R_k = R = x_2^k$ and $R_j=0$ for $j\neq k$, $d$ is contained in $\underline{C}(1,Q,R)\subset C^+(g,h)$ if $d_2=0$. 
If $d_2\neq 0$ and $d_j$ is not $0$ for $j\geq 3$, the vector $v = -\epsilon_jd_2^kd_je_j$ satisfies the conditions in Lemma~\ref{lem:tangent cone for r=1} for $d$.
Indeed, $R_3(d) = d_2^k$ and $v\cdot \nabla Q(d) = -2d_2^k\epsilon_j^2d_j^2$, in particular the signs of these values are mutually opposite. 

In what follows, we assume $d_2\neq 0$ and $d_3=\cdots =d_n=0$. 
Since $Q(d) = \epsilon_1d_1d_2$, $d_1$ is equal to $0$ and $\nabla Q (d)=(\epsilon_1d_2,0,\ldots, 0)$. 
If $k$ is even and $\epsilon_1d_2 > 0$, the vector $v=-e_1$ satisfies the conditions in Lemma~\ref{lem:tangent cone for r=1} for $d$ since $v_1=-1<0$, $R_3(d)=d_2^k>0$, and $v\cdot \nabla Q(d)=-\epsilon_1d_2<0$.
If $k$ is even and $\epsilon_j=-1$ for some $j\geq 3$, the vector $v=e_j$ satisfies the conditions in Lemma~\ref{lem:tangent cone for r=1} for $d$ since $v_1=0$, $R_3(d)=d_2^k>0$, $v\cdot \nabla Q(d)=0$, and ${}^tv\mathrm{Hess}(Q)v = -2<0$.
If $k$ is odd and $\epsilon_1=1$, the vector $v=-e_1$ satisfies the conditions in Lemma~\ref{lem:tangent cone for r=1} for $d$ since $v_1=-1<0$, and $v\cdot \nabla Q(d)=-d_2$, whose sign is opposite to $R_3(d)=d_2^k$.
If $k$ is odd, $\epsilon_jd_2<0$ for some $j\geq 3$, the vector $v=e_j$ satisfies the conditions in Lemma~\ref{lem:tangent cone for r=1} for $d$ since $v_1=0$, $v\cdot \nabla Q(d)=0$, and ${}^tv\mathrm{Hess}(Q)v = 2\epsilon_j$, whose sign is opposite to $R_3(d)=d_2^k$.

So far, we have shown that 

\begin{itemize}

\item 
$C^+(g,h) = \overline{C}(1,Q)$ if one of the following conditions holds:

\begin{itemize}

\item 
$k$ is even and $\epsilon_j=-1$ for some $j\geq 3$, 

\item 
$k$ is odd and $\epsilon_1 = 1$, 

\item 
$k$ is odd and $\epsilon_j\epsilon_l=-1$ for some $j, l\geq 3$, 

\end{itemize}

\item 
$\overline{C}(1,Q)\setminus \{(0,d,0,\ldots, 0)\in \R^n~|~\epsilon_1d<0\}$ is contained in $C^+(g,h)$ when $k$ is even and $\epsilon_3=\cdots =\epsilon_n=1$,

\item 
$\overline{C}(1,Q)\setminus \{(0,d,0,\ldots, 0)\in \R^n~|~\delta d >0\}$ is contained in $C^+(g,h)$ when $k$ is odd, $\epsilon_1=-1$, and $\epsilon_3=\cdots =\epsilon_n = \delta$ for some $\delta\in \{1,-1\}$. 

\end{itemize}

\noindent
We will first show that $(0,d,0,\ldots, 0)$ ($\epsilon_1d<0$) is not contained in $C^+(g,h)$ when $k$ is even and $\epsilon_3=\cdots =\epsilon_n = 1$. 
Let $x\in M(g,h)$.
Since $h(x)=0$, $\epsilon_1x_1x_2$ is equal to $-x_2^k - \sum_{j=3}^{n}x_j^2 \leq 0$.
Since $g(x)=x_1\leq 0$, $\epsilon_1x_2\geq 0$. 
Thus, $\epsilon_1d_2$ is larger than or equal to $0$ for any $d\in C^+(g,h)$. 

We will next show that $(0,d,0,\ldots, 0)$ ($\delta d >0$) is not contained in $C^+(g,h)$ when $k$ is odd, $\epsilon_1=-1$, and $\epsilon_3=\cdots = \epsilon_n=\delta$.  
Suppose that there exists a sequence $\{x^{(m)}\}_{m=1}^\infty$ of points in $M(g,h)$ and $t_m>0$ such that $\lim_{m\to\infty} x^{(m)}=0$ and $\lim_{m\to\infty} t_mx^{(m)}=(0,d,0,\ldots,0)$. 
Since $h(x^{(m)})=0$, $x_1^{(m)}x_2^{(m)}$ is equal to $(x_2^{(m)})^k +\delta \sum_{j=3}^{n}(x_j^{(m)})^2$. 
By the assumption, $\delta x_2^{(m)}$ is larger than $0$ for $m\gg 0$, and thus the following inequality holds:
\[
x_1^{(m)}\cdot (\delta x_2^{(m)}) = \delta \left((x_2^{(m)})^k +\delta \sum_{j=3}^{n}(x_j^{(m)})^2\right)>0.
\]
However, $x_1^{(m)}\cdot (\delta x_2^{(m)})$ never be larger than $0$ since $x_1^{(m)} = g(x^{(m)})\leq 0$. 
\end{proof}

\begin{proposition}

GCQ holds for $(g,h)$ if and only if one of the following holds:
\begin{enumerate}
\item 
\begin{math}
k
\end{math}
is even and one of 
\begin{math}
\epsilon_3, \ldots, \epsilon_n
\end{math}
is 
\begin{math}
-1
\end{math}, 
\item
\begin{math}
k
\end{math}
is odd and 
(\begin{math}
\epsilon_1 = 1
\end{math}
or
\begin{math}
\left\{ \epsilon_3, \ldots, \epsilon_n \right\} = \left\{ -1, 1 \right\}
\end{math}).
\end{enumerate}

\end{proposition}
\begin{proof}
It is easy to check that 
\begin{math}
L^+ \left( g, h \right)^\circ
\end{math}
is equal to 
\begin{math}
\left\{ w \in \mathbb{R}^n \middle| w_1 \ge 0, w_2 = \cdots = w_n = 0 \right\}
\end{math}.

\noindent
\textbf{Proof of ``if'' part:} By the previous proposition, the conditions implies that 
\begin{math}
C^+ \left( g, h \right)
\end{math}
is equal to 
\begin{math}
\overline{C} \left( 1, Q \right)
\end{math}. Since 
\begin{math}
C^+ \left( g, h \right) \subset L^+ \left( g, h \right)
\end{math}
always holds, it is enough to show 
\begin{math}
\overline{C} \left( 1, Q \right)^\circ \subset L^+ \left( g, h \right)^\circ
\end{math}.

Take any 
\begin{math}
w \in \overline{C} \left( 1, Q \right)^\circ
\end{math}. Since 
\begin{math}
\left( d_1, 0, \ldots, 0 \right) \in \overline{C} \left( 1, Q \right)
\end{math}
holds for 
\begin{math}
d_1 \le 0
\end{math}, 
\begin{math}
w_1 \ge 0
\end{math}
holds. Since 
\begin{math}
\left( 0, d_2, 0, \ldots, 0 \right) \in \overline{C} \left( 1, Q \right)
\end{math}
holds for 
\begin{math}
d_2 \in \mathbb{R}
\end{math}, 
\begin{math}
w_2 = 0
\end{math}
holds. Since 
\begin{math}
-e_1 + \epsilon_1 \epsilon_j e_2 \pm e_j \in \overline{C} \left( 1, Q \right)
\end{math}
holds for all 
\begin{math}
j \in \left\{ 3, \ldots, n \right\}
\end{math}, the following inequality holds: 
{\allowdisplaybreaks
\begin{align*}
&0 \geq w \cdot \left(-e_1 + \epsilon_1 \epsilon_j e_2 \pm e_j\right)=-w_1 \pm w_j\\
\Leftrightarrow & \pm w_j \leq w_1 \leq 0. 
\end{align*}
}%
We thus obtain 
\begin{math}
w_j = 0
\end{math}
for all 
\begin{math}
j \in \left\{ 3, \ldots, n \right\}
\end{math}. This proves 
\begin{math}
w \in L^+ \left( g, h \right)^\circ
\end{math}.

\noindent
\textbf{Proof of ``only if'' part:}
We show the contraposition of the statement.
%
If $k$ is even and
\begin{math}
\epsilon_3 = \cdots = \epsilon_n = 1
\end{math}, the previous proposition implies that 
\begin{math}
C^+ \left( g, h \right)
\end{math}
is equal to 
\begin{math}
\overline{C} \left( 1, Q \right) \setminus \left\{ \left( 0, d, 0, \ldots, 0 \right) \in \mathbb{R}^n \middle| \epsilon_1 d < 0 \right\}
\end{math}. We claim that 
\begin{math}
w = \left( 0, w_2, 0, \ldots, 0 \right)
\end{math}
for 
\begin{math}
\epsilon_1 w_2 \le 0
\end{math}
is contained in
\begin{math}
C^+ \left( g, h \right)^\circ
\end{math}. Take any 
\begin{math}
d \in C^+ \left( g, h \right)
\end{math}. Then, 
\begin{math}
d_1 \le 0
\end{math}
and 
\begin{math}
\epsilon_1 d_1 d_2 + \sum_{j=3}^n d_j^2 = 0
\end{math}
holds. The latter implies that 
\begin{math}
\epsilon_1 d_1 d_2 \le 0
\end{math}. If 
\begin{math}
d_1 < 0
\end{math}, 
\begin{math}
\epsilon_1 d_2 \ge 0
\end{math}
holds and thus 
\begin{math}
w \cdot d = w_2 d_2 \le 0
\end{math}
holds. If 
\begin{math}
d_1 = 0
\end{math}, 
\begin{math}
d_3 = \cdots d_n = 0
\end{math}
holds and thus 
\begin{math}
\epsilon_1 d_2 \ge 0
\end{math}. Therefore, 
\begin{math}
w \cdot d = w_2 d_2 \le 0
\end{math}
holds. This proves the claim and thus GCQ is violated. 

If $k$ is odd and 
\begin{math}
\epsilon_1 = -1
\end{math}, and 
\begin{math}
\epsilon_3 = \cdots = \epsilon_n = \delta
\end{math}
for some 
\begin{math}
\delta \in \left\{ 1, -1 \right\}
\end{math}, the previous proposition implies that 
\begin{math}
C^+ \left( g, h \right)
\end{math}
is equal to 
\begin{math}
\overline{C} \left( 1, Q \right) \setminus \left\{ \left( 0, d, 0, \ldots, 0 \right) \in \mathbb{R}^n \middle| \delta d > 0 \right\}
\end{math}. We claim that 
\begin{math}
w = \left( 0, w_2, 0, \ldots, 0 \right)
\end{math}
for 
\begin{math}
\delta w_2 \ge 0
\end{math}
is contained in
\begin{math}
C^+ \left( g, h \right)^\circ
\end{math}. Take any 
\begin{math}
d \in C^+ \left( g, h \right)
\end{math}. Then, 
\begin{math}
d_1 \le 0
\end{math}
and 
\begin{math}
- d_1 d_2 + \delta \sum_{j=3}^n d_j^2 = 0
\end{math}
holds. The latter implies that 
\begin{math}
\delta d_1 d_2 \ge 0
\end{math}. If 
\begin{math}
d_1 < 0
\end{math}, 
\begin{math}
\delta d_2 \le 0
\end{math}
holds and thus 
\begin{math}
w \cdot d = w_2 d_2 \le 0
\end{math}
holds. If 
\begin{math}
d_1 = 0
\end{math}, 
\begin{math}
d_3 = \cdots =d_n = 0
\end{math}
holds and thus 
\begin{math}
\delta d_2 \le 0
\end{math}. Therefore, 
\begin{math}
w \cdot d = w_2 d_2 \le 0
\end{math}
holds. This proves the claim and thus GCQ is violated. 
\end{proof}

\subsubsection*{The germ of type $(4)$ in Table~\ref{table:generic constraint q>0 r=1}}
Let $(g,h) = \left(x_1,x_2,\delta_1 x_1^2+\delta_2x_2^2+ \alpha x_1x_2+\sum_{j=3}^{n} \epsilon_j x_j^2\right)$, where $\delta_j,\epsilon_j\in \{1,-1\}$, $\alpha \in \mathbb{R}, 4 \delta_1 \delta_2 - \alpha^2 \neq 0$, and $Q =\delta_1 x_1^2+\delta_2x_2^2+ \alpha x_1x_2+\sum_{j=3}^{n} \epsilon_j x_j^2$. In this case, 
\begin{math}
C^+ \left( g, h \right) = \overline{C} \left( 2, Q \right)
\end{math}
holds.

\begin{proposition}\label{prop:GCQ type (4) q,r>0}

GCQ holds for $(g,h)$ if and only if 
\begin{math}
\left\{ \epsilon_3, \ldots, \epsilon_n \right\} = \left\{ 1, -1 \right\}
\end{math}
holds.

\end{proposition}
\begin{proof}
It is easy to check that 
\begin{math}
L^+ \left( g, h \right)^\circ
\end{math}
is equal to 
\begin{math}
\left\{ w \in \mathbb{R}^n \middle| w_1 \ge 0, w_2 \ge 0, w_3 =\cdots = w_n = 0 \right\}
\end{math}.

\noindent
\textbf{Proof of ``if'' part:} Without loss of generality, we can assume 
\begin{math}
\epsilon_3 = 1
\end{math}
and 
\begin{math}
\epsilon_4 = -1
\end{math}. Take any 
\begin{math}
w \in C^+ \left( g, h \right)^\circ
\end{math}. For each 
\begin{math}
j \in \left\{ 3, \ldots, n \right\}
\end{math}, either 
\begin{math}
\pm e_j \pm e_3
\end{math}
or 
\begin{math}
\pm e_j \pm e_4
\end{math}
is contained in 
\begin{math}
C^+ \left( g, h \right)
\end{math}. This proves 
\begin{math}
w_j = 0
\end{math}
for all 
\begin{math}
j \in \left\{ 3, \ldots, n \right\}
\end{math}. Since either
\begin{math}
-e_j \pm e_3 \in C^+ \left( g, h \right)
\end{math}
or 
\begin{math}
-e_j \pm e_4 \in C^+ \left( g, h \right)
\end{math}
holds for 
\begin{math}
j \in \left\{ 1, 2 \right\}
\end{math}, 
\begin{math}
w_1 \ge 0
\end{math}
and 
\begin{math}
w_2 \ge 0
\end{math}
hold. This proves 
\begin{math}
C^+ \left( g, h \right)^\circ \subset L^+ \left( g, h \right)^\circ
\end{math}
and thus GCQ holds in this case. 

\noindent
\textbf{Proof of ``only if'' part:} 
We assume $\epsilon_3=\cdots = \epsilon_n = \delta$ for some $\delta\in \{1,-1\}$.
We take $P\in O(2)$ so that the following equality holds: 
\[
\delta_1x_1^2 + \delta_2x_2^2 +\alpha x_1x_2 = (x_1,x_2){}^tP\begin{pmatrix}
\lambda_1&0\\
0&\lambda_2
\end{pmatrix}P\begin{pmatrix}
x_1\\x_2
\end{pmatrix},
\]
where $\lambda_1,\lambda_2\neq 0$ is the eigenvalues of $\begin{pmatrix}
\delta_1&\alpha/2\\
\alpha/2&\delta_2
\end{pmatrix}$.
The following then holds for $d\in C^+(g,h)$: 
{\allowdisplaybreaks
\begin{align*}
& \delta_1d_1^2 + \delta_2d_2^2 +\alpha d_1d_2 = -\delta\sum_{j=3}^{n} d_j^2\\
\Rightarrow & \left|(d_1,d_2){}^tP\begin{pmatrix}
\lambda_1&0\\
0&\lambda_2
\end{pmatrix}P\begin{pmatrix}
d_1\\d_2
\end{pmatrix}\right| = \left|\sum_{j=3}^{n} d_j^2\right|\\
\Rightarrow& \max\{|\lambda_1|,|\lambda_2|\} \left\| \left(d_1,d_2\right)\right\|^2 \geq \left|\sum_{j=3}^{n} d_j^2\right|. 
\end{align*}
}%
Let $R > \sqrt{\max\{|\lambda_1|,|\lambda_2|\}}$. 
For any $d\in C^+(g,h)$, the inner product $(R(e_1+e_2)+e_3)\cdot d$ is estimated as follows:
{\allowdisplaybreaks
\begin{align*}
& (R(e_1+e_2)+e_3)\cdot d\\
=& R(d_1+d_2) + d_3 \\
\leq & -R(|d_1|+|d_2|) +|d_3| & (\because d_1,d_2\leq 0)\\
\leq & -R\left\|\left(d_1,d_2\right)\right\| + |d_3| & \left(\because \left\|\left(d_1,d_2\right)\right\| = \sqrt{d_1^2+d_2^2} \leq |d_1|+|d_2|\right)\\
\leq & -R \sqrt{\dfrac{\left|\sum_{j=3}^{n} d_j^2\right|}{\max\{|\lambda_1|,|\lambda_2|\}}}+|d_3| \\
\leq & \left(-\frac{R}{\sqrt{\max\{|\lambda_1|,|\lambda_2|\}}}+1\right)|d_3|\leq 0.
\end{align*}
}%
Thus, $R(e_1+e_2)+e_3$ is contained in $C^+(g,h)^\circ$. 
However, it is not in $L^+(g,h)^\circ$, and thus GCQ is violated. 
\end{proof}

\subsubsection*{The germ of type $(5)$ in Table~\ref{table:generic constraint q>0 r=1}}

Let $(g,h) = \left(x_1,x_2,x_1^3+\epsilon_1x_2^2+\epsilon_2x_1x_2+\sum_{j=3}^{n} \epsilon_j x_j^2\right)$, where $\epsilon_j\in \{1,-1\}$, $Q =\epsilon_1x_2^2+\epsilon_2x_1x_2 +\sum_{j=3}^{n} \epsilon_j x_j^2$, and $R=x_1^3$. 

\begin{proposition}\label{prop:GCQ type (5) q,r>0}

The tangent cone $C^+(g,h)$ is equal to $\overline{C}(2,Q)\setminus \{(d,0,\ldots, 0)\in \R^n~|~d<0\}$ if $\epsilon_2= \cdots =\epsilon_n=-1$, and $C^+(g,h)$ is equal to $\overline{C}(2,Q)$ otherwise. 

\end{proposition}

\begin{proof}
The tangent cone $C^+(g,h)$ is contained in $\overline{C}(2,Q)$ by Lemma~\ref{lem:tangent cone for r=1}. 
Let $d\in \overline{C}(2,Q)$. 
Since $R_3 = R = x_1^3$ and $R_r=0$ for $r\neq 3$, $d$ is contained in $\underline{C}(2,Q,R)\subset C^+(g,h)$ if $d_1=0$. 
If $d_1< 0$ and some $d_j$ is not $0$ for $j\geq 3$, the vector $v = \epsilon_jd_je_j$ satisfies the conditions in Lemma~\ref{lem:tangent cone for r=1} for $d$. 
Indeed, $R_3(d) = d_1^3<0$ and $v\cdot \nabla Q(d) = 2\epsilon_j^2 d_j^2>0$.
Hence $d$ is contained in $C^+(g,h)$.

In what follows, we assume $d_1<0$ and $d_3=\cdots =d_n=0$.
Since $Q(d) = \epsilon_1 d_2^2+\epsilon_2 d_1d_2$ is equal to $0$, either $d_2=0$ or $\epsilon_2d_1 = -\epsilon_1d_2 $. 
If $d_2\neq 0$, $\nabla Q(d)$ is equal to 
\[
(\epsilon_2d_2,2\epsilon_1d_2+\epsilon_2 d_1,0,\ldots, 0) =(\epsilon_2d_2,-\epsilon_2 d_1,0,\ldots, 0).
\]
Since either the first or the second component is negative, we can take $v\in \R^n$ so that $v_1,v_2\leq 0$ and $v\cdot \nabla Q (d)>0$ (by putting $v = -e_1$ or $-e_2$).
Hence $d$ is contained in $C^+(g,h)$.  

So far, we have shown that $\overline{C}(2,Q)\setminus \{(d,0,\ldots, 0)\in \R^n~|~d<0\}$ is contained in $C^+(g,h)$.
In what follows, we assume $d_2=0$.
If $\epsilon_2 = 1$, the vector $v = -e_2$ satisfies the conditions in Lemma~\ref{lem:tangent cone for r=1} for $d$ since $v\cdot \nabla Q (d) = -\epsilon_2d_1=-d_1>0$. 
Hence $d$ is contained in $C^+(g,h)$. 
If some $\epsilon_j$ is equal to $1$ for $j\geq 3$, the vector $v = e_j$ satisfies the conditions in Lemma~\ref{lem:tangent cone for r=1} since $v\cdot \nabla Q (d)=0$ and ${}^tv \mathrm{Hess}(Q)v = 2\epsilon_j=2>0$. 
Lastly, we will show that $(d,0,\ldots, 0)$ ($d<0$) is not contained in $C^+(g,h)$ when $\epsilon_2=\cdots =\epsilon_n = -1$. 
Let $x\in M(g,h)$. 
Since $x_1,x_2\leq 0$, $\epsilon_1x_2^2 -x_1x_2 = -x_1^3 +\sum_{j\geq 3}x_j^2 \geq 0$. 
If $\epsilon_1= -1$, $\epsilon_1x_2^2 -x_1x_2 =x_2(-x_1-x_2)\leq 0$ since $x_1,x_2\leq 0$. 
Thus, $x_2(-x_1-x_2)$ and $-x_1^3 +\sum_{j\geq 3}x_j^2$ are both equal to $0$, meaning that $x=0$. 
Hence $M(g,h)= C^+(g,h)=\{0\}$. (Note that $\overline{C}(2,Q)=\{(d,0,\ldots,0)\in \R^n~|~ d\leq 0\}$ in this case.)
If $\epsilon_1 = 1$, $x_2(x_1-x_2) \leq 0$, and thus, $x_1\geq x_2$. 
Thus, $d_1$ is not less than $d_2$ for any $d\in C^+(g,h)$, implying that $(d,0,\ldots, 0)$ is not in $C^+(g,h)$ for $d<0$. 
\end{proof}

\begin{proposition}

GCQ holds for $(g,h)$ if and only if 
\begin{math}
\left\{ \epsilon_3, \ldots, \epsilon_n \right\} = \left\{ 1, -1 \right\}
\end{math}
holds.

\end{proposition}
\begin{proof}
It is easy to check that 
\begin{math}
L^+ \left( g, h \right)^\circ
\end{math}
is equal to 
\begin{math}
\left\{ w \in \mathbb{R}^n \middle| w_1 \ge 0, w_2 \ge 0, w_3 = \cdots = w_n = 0 \right\}
\end{math}.

\noindent
\textbf{Proof of ``if'' part:} Under the assumption, 
\begin{math}
C^+ \left( g, h \right) = \overline{C} \left( 2,Q\right)
\end{math}
holds by the previous proposition. In addition, we can assume 
\begin{math}
\epsilon_3 = 1
\end{math}
and 
\begin{math}
\epsilon_4 = -1
\end{math}
without loss of generality. Since either 
\begin{math}
\pm e_j \pm e_3
\end{math}
or 
\begin{math}
\pm e_j \pm e_4
\end{math}
is in 
\begin{math}
\overline{C} \left( Q, 2 \right)
\end{math}
for 
\begin{math}
j \in \left\{ 3, \ldots, n \right\}
\end{math}, 
\begin{math}
w_j = 0
\end{math}
holds for all 
\begin{math}
j \in \left\{ 3, \ldots, n \right\}
\end{math}. Since either 
\begin{math}
- e_2 \pm e_3
\end{math}
or 
\begin{math}
- e_2 \pm e_4
\end{math}
is in 
\begin{math}
\overline{C} \left( 2,Q \right)
\end{math}, 
\begin{math}
w_2 \ge 0
\end{math}
holds. Since 
\begin{math}
- e_1 \in \overline{C} \left( 2,Q \right)
\end{math}
holds, 
\begin{math}
w_1 \ge 0
\end{math}
holds. This proves 
\begin{math}
w \in L^+ \left( g, h \right)^\circ
\end{math}. 

\noindent
\textbf{Proof of ``only if'' part:} 
We assume $\epsilon_3=\cdots = \epsilon_n$.
Let $\lambda_1,\lambda_2\neq 0$ be the eigenvalues of $\begin{pmatrix}
0&\epsilon_2/2\\
\epsilon_2/2&\epsilon_1
\end{pmatrix}$.
As in the proof of Proposition~\ref{prop:GCQ type (5) q,r>0}, we obtain the following inequality for $d\in C^+(g,h)$: 
{\allowdisplaybreaks
\begin{align*}
\max\{|\lambda_1|,|\lambda_2|\} \left\| \left(d_1,d_2\right)\right\|^2 \geq \left|\sum_{j=3}^{n} d_j^2\right|. 
\end{align*}
}%
We can thus show that $R(e_1+e_2)+e_3$ is contained in $C^+(g,h)^\circ$ for $R > \sqrt{\max\{|\lambda_1|,|\lambda_2|\}}$ in the same way as in the proof of Proposition~\ref{prop:GCQ type (5) q,r>0}.
\end{proof}

\subsubsection*{The germ of type $(6)$ in Table~\ref{table:generic constraint q>0 r=1}}

Let $(g,h) = \left(x_1,x_2,(x_1+\epsilon_1x_2)^2+\epsilon_2x_2^3 + \sum_{j=3}^{n} \epsilon_j x_j^2\right)$, where $\epsilon_j\in \{1,-1\}$, $Q =(x_1+\epsilon_1x_2)^2+ \sum_{j=3}^{n} \epsilon_j x_j^2$, and $R=\epsilon_2x_2^3$. 

\begin{proposition}

The tangent cone $C^+(g,h)$ is equal to $\{0\}$ if $\epsilon_2=-1$ and $\epsilon_3= \cdots =\epsilon_n=1$, and $C^+(g,h)$ is equal to $\overline{C}(2,Q)$ otherwise.

\end{proposition}

\begin{proof}
The tangent cone $C^+(g,h)$ is contained in $\overline{C}(2,Q)$ by Lemma~\ref{lem:tangent cone for r=1}. 
Let $d\in \overline{C}(2,Q)$. 
Since $R_3 = R = \epsilon_2x_2^3$ and $R_j=0$ for $j\neq 3$, $d$ is contained in $\underline{C}(2,Q,R)\subset C^+(g,h)$ if $d_2=0$. 
If $d_2\neq 0$ and $d_j$ is not $0$ for $j\geq 3$, the vector $v = -\epsilon_2\epsilon_jd_2d_je_j$ satisfies the conditions in Lemma~\ref{lem:tangent cone for r=1} for $d$.
Indeed, $R_3(d) = d_2^3$ and $v\cdot \nabla Q(d) = -2d_2\epsilon_j^2d_j^2$, in particular the signs of these values are mutually opposite. 

In what follows, we assume $d_2\neq 0$ and $d_3=\cdots =d_n=0$. 
Since $Q(d) = (d_1+\epsilon_1d_2)^2$, $d_1+\epsilon_1d_2=0$ and $\nabla Q (d)=0$. 
If $\epsilon_2=+1$, the vector $v=-e_1$ satisfies the conditions in Lemma~\ref{lem:tangent cone for r=1} for $d$ since $v_1=-1<0$, $R_3(d)=\epsilon_2d_2^3<0$ (note that $d_2\leq 0$), and ${}^tv\mathrm{Hess}(Q)v=2>0$.
If $\epsilon_2\epsilon_j > 0$ for some $j\geq 3$, the vector $v=e_j$ satisfies the conditions in Lemma~\ref{lem:tangent cone for r=1} for $d$ since $v_1=0$, $R_3(d)=\epsilon_2d_2^3$, and ${}^tv\mathrm{Hess}(Q)v=2\epsilon_j$.

So far, we have shown that $C^+(g,h)$ is equal to $\overline{C}(2,Q)$ unless $\epsilon_2=-1$ and $\epsilon_3=\cdots =\epsilon_n=1$. 
We can easily show that $M(g,h)$ is equal to $\{0\}$ if $\epsilon_2=-1$ and $\epsilon_3=\cdots =\epsilon_n=1$. 
Thus, the tangent cone $C^+(g,h)$ is also equal to $\{0\}$. 
\end{proof}

\begin{proposition}

GCQ holds for $(g,h)$ if and only if 
\begin{math}
\left\{ \epsilon_3, \ldots, \epsilon_n \right\} = \left\{ 1, -1 \right\}
\end{math}
holds.

\end{proposition}
\begin{proof}
It is easy to check that 
\begin{math}
L^+ \left( g, h \right)^\circ
\end{math}
is equal to 
\begin{math}
\left\{ w \in \mathbb{R}^n \middle| w_1 \ge 0, w_2 \ge 0, w_3 = \cdots = w_n = 0 \right\}
\end{math}.

\noindent
\textbf{Proof of ``if'' part:}  Under the assumption, 
\begin{math}
C^+ \left( g, h \right) = \overline{C} \left( 2, Q \right)
\end{math}
holds. Without loss of generality, we can assume 
\begin{math}
\epsilon_3 = 1
\end{math}
and 
\begin{math}
\epsilon_4 = -1
\end{math}. Take any 
\begin{math}
w \in C^+ \left( g, h \right)^\circ
\end{math}. Since either 
\begin{math}
\pm e_j \pm e_3
\end{math}
or 
\begin{math}
\pm e_j \pm e_4
\end{math}
is contained in 
\begin{math}
\overline{C} \left( 2, Q \right)
\end{math}
for 
\begin{math}
j \in \left\{ 3, \ldots, n \right\}
\end{math}, 
\begin{math}
w_j = 0
\end{math}
holds for all 
\begin{math}
j \in \left\{ 3, \cdots, n \right\}
\end{math}. Since either 
\begin{math}
-e_j \pm e_3
\end{math}
or 
\begin{math}
-e_j \pm e_4
\end{math}
is contained in 
\begin{math}
C^+ \left( g, h \right)
\end{math}
for 
\begin{math}
j \in \left\{ 1, 2 \right\}
\end{math}, 
\begin{math}
w_j \ge 0 
\end{math}
holds. This proves 
\begin{math}
w \in L^+ \left( g, h \right)^\circ
\end{math}.

\noindent
\textbf{Proof of ``only if'' part:}
We assume $\epsilon_3=\cdots = \epsilon_n$.
The eigenvalues of $\begin{pmatrix}
1&\epsilon_1\\
\epsilon_1&1
\end{pmatrix}$ are $0$ and $2$.
Thus, as in the proof of Proposition~\ref{prop:GCQ type (5) q,r>0}, we obtain $2\left\| \left(d_1,d_2\right)\right\|^2 \geq \left|\sum_{j=3}^{n} d_j^2\right|$ for $d\in C^+(g,h)$.
We can thus show that $2(e_1+e_2)+e_3$ is contained in $C^+(g,h)^\circ$ in the same way as in the proof of Proposition~\ref{prop:GCQ type (5) q,r>0}.
\end{proof}

\subsubsection*{The germ of type $(7)$ in Table~\ref{table:generic constraint q>0 r=1}}

Let $(g,h) = \left(x_1,x_2,x_3^3+\epsilon_1x_2x_3+\epsilon_2x_3x_1+\epsilon_3x_1x_2 + \sum_{j=4}^{n} \epsilon_j x_j^2\right)$, where $\epsilon_j\in \{1,-1\}$, $Q =\epsilon_1x_2x_3+\epsilon_2x_3x_1+\epsilon_3x_1x_2 + \sum_{j=4}^{n} \epsilon_j x_j^2$, and $R=x_3^3$. 

\begin{proposition}\label{prop:tangent cone type (7) q,r>0}

The tangent cone $C^+(g,h)$ is equal to $\overline{C}(2,Q)\setminus \{(0,0,d,0,\ldots, 0)\in \R^n~|~ \delta d>0\}$ if $\epsilon_1=\epsilon_2=-1$ and $\epsilon_4=\cdots =\epsilon_n=\delta$ for some $\delta\in \{0,1\}$, and $C^+(g,h)=\overline{C}(2,Q)$ otherwise. 

\end{proposition}

\begin{proof}
The tangent cone $C^+(g,h)$ is contained in $\overline{C}(2,Q)$ by Lemma~\ref{lem:tangent cone for r=1}. 
Let $d\in \overline{C}(2,Q)$. 
Since $R_3 = R = x_3^3$ and $R_j=0$ for $j\neq 3$, $d$ is contained in $\underline{C}(2,Q,R)\subset C^+(g,h)$ if $d_3=0$. 
In what follows, we assume $d_3\neq 0$.
If $d_j$ is not $0$ for $j\geq 4$, the vector $v = -\epsilon_jd_3d_je_j$ satisfies the conditions in Lemma~\ref{lem:tangent cone for r=1} for $d$.
Indeed, $R_3(d) = d_3^3$ and $v\cdot \nabla Q(d) = -2d_3\epsilon_j^2d_j^2$, in particular the signs of these values are mutually opposite. 
If $\epsilon_1d_2+\epsilon_2d_1$ is not $0$, the vector $v = -d_3(\epsilon_1d_2+\epsilon_2d_1)e_3$ satisfies the conditions in Lemma~\ref{lem:tangent cone for r=1} for $d$.
Indeed, $R_3(d) = d_3^3$ and $v\cdot \nabla Q(d) = -d_3(\epsilon_1d_2+\epsilon_2d_1)^2$, in particular the signs of these values are mutually opposite.

In what follows, we assume $d_3\neq 0$ and $\epsilon_1d_2+\epsilon_2d_1=d_4=\cdots =d_n=0$. 
Since $Q(d) = \epsilon_3d_1d_2 = -\frac{\epsilon_1\epsilon_3}{\epsilon_2}d_2^2$, $d_1=d_2=0$ and $\nabla Q (d)=(\epsilon_2d_3,\epsilon_1d_3,0,\ldots, 0)$. 
If $\epsilon_i=1$ for $i=1$ or $i=2$, the vector $v = -e_{i'}$ satisfies the conditions in Lemma~\ref{lem:tangent cone for r=1} for $d$, where $\{i,i'\}=\{1,2\}$.
Indeed, $R_3(d) = d_3^3$, $v_i=0$, $v_{i'} = -1<0$, and $v\cdot \nabla Q(d) = -d_3$.
If $\epsilon_j d_3<0$ for some $j\geq 4$, the vector $v=e_j$ satisfies the conditions in Lemma~\ref{lem:tangent cone for r=1} for $d$.
Indeed, $v_1=v_2=0$, $v\cdot \nabla Q(d)=0$, and ${}^tv\mathrm{Hess}(Q) v = 2\epsilon_j$, whose sign is opposite to $R_3(d)=d_3^3$.

So far, we have shown that 

\begin{itemize}
\item 
$C^+(g,h)$ is equal to $\overline{C}(2,Q)$ unless $\epsilon_1=\epsilon_2=-1$ and $\epsilon_4=\cdots =\epsilon_n=\delta$ for some $\delta\in \{-1,1\}$, and

\item
if $\epsilon_1=\epsilon_2=-1$ and $\epsilon_4=\cdots =\epsilon_n=\delta$ for some $\delta\in \{-1,1\}$, the set 
\[
\overline{C}(2,Q)\setminus\left\{(0,0,d,0,\ldots,0)\in \R^n~|~\delta d> 0\right\}
\]
is contained in $C^+(g,h)$. 

\end{itemize}

\noindent
In what follows, we show that $(0,0, d,0,\ldots, 0)$ is not contained in $C^+(g,h)$ for $\delta d>0$ when $\epsilon_1=\epsilon_2=-1$ and $\epsilon_4=\cdots =\epsilon_n=\delta$ for some $\delta\in \{-1,1\}$. 
Let $x\in M(g,h)$. 
The following equality holds: 
\[
0=h(x) = x_3(x_3^2-x_1-x_2)+\epsilon_3x_1x_2 + \delta\sum_{j=4}^{n}x_j^2.
\]
If $\epsilon_3x_1x_2 + \delta\sum_{j=4}^{n}x_j^2$ is equal to $0$, either $x_3$ or $x_3^2-x_1-x_2$ is also equal to $0$. 
Since $x_1$ and $x_2$ are less than or equal to $0$, $x_3$ is equal to $0$ even in the latter case. 
If $\epsilon_3x_1x_2 + \delta\sum_{j=4}^{n}x_j^2\neq 0$ and $\epsilon_3 = \delta$, the sign of $\epsilon_3x_1x_2 + \delta\sum_{j=4}^{n}x_j^2$ is same as that of $\delta$. 
Since $x_3^2-x_1-x_2>0$, the sign of $x_3$ is opposite to that of $\delta$. 
Thus, $(0,0,d,0,\ldots,0)$ is not contained in $C^+(g,h)$ for $\delta d >0$ when $\epsilon_3=\delta$. 
Suppose that $\epsilon_3$ is equal to $-\delta$ and $\delta x_3>-(x_1+x_2)$. 
The following equality holds:
\[
0 = \delta h(x) > -(x_1+x_2)^3 + (x_1+x_2)^2 -x_1x_2 + \sum_{j=4}^{n}x_j^2.
\]
However, the last value is larger than or equal to $0$ since $(x_1+x_2)^2 -x_1x_2\geq 0$. 
This contradiction implies that $\delta x_3\leq -(x_1+x_2)$, and thus $(0,0,d,0,\ldots, 0)$ is not contained in $C^+(g,h)$ for $\delta d >0$. 
\end{proof}

\begin{proposition}

GCQ holds for $(g,h)$ if and only if 
\begin{math}
\{\epsilon_4 ,\ldots,\epsilon_n\} =\{1,-1\}
\end{math}
or
either $\epsilon_1$ or $\epsilon_2$ is $1$.



\end{proposition}
\begin{proof}
It is easy to check that 
\begin{math}
L^+ \left( g, h \right)^\circ
\end{math}
is equal to 
\begin{math}
\left\{ w \in \mathbb{R}^n \middle| w_1 \ge 0, w_2 \ge 0, w_3 = \cdots = w_n = 0 \right\}
\end{math}.

\noindent
\textbf{Proof of ``if'' part:}
Under the assumption, $C^+(g,h) = \overline{C}(2,Q)$ holds. 
Take any $w\in C^+(g,h)^\circ$.
Since $-e_1,-e_2,\pm e_3$ are contained in $\overline{C}(2,Q)$, $w_1,w_2\geq 0$ and $w_3=0$ hold. 
For any $s>0$, since $-\frac{e_2}{s}+\epsilon_1\epsilon_jse_3\pm e_j$ is contained in $\overline{C}(2,Q)$, $-\frac{w_2}{s}\pm w_j \leq 0$ holds. 
Thus, $w_j=0$ holds and $w$ is contained in $L^+(g,h)^\circ$. 

\noindent
\textbf{Proof of ``only if'' part:}
We assume $\epsilon_1=\epsilon_2=-1$ and $\epsilon_4=\cdots =\epsilon_n = \delta$. 
By Proposition~\ref{prop:tangent cone type (7) q,r>0}, $C^+(g,h)$ is equal to $\overline{C}(2,Q)\setminus\{(0,0,d,0,\ldots, 0)\in \R^n~|~\delta d>0\}$. 
Let $d\in C^+(g,h)$.
If $d_1=d_2=0$, $d_4,\ldots,d_n$ are also equal to $0$ since $Q(d)=\delta \sum_{j=4}^{n}d_j^2=0$, and $\delta d_3\leq 0$. 
Otherwise, $d_3$ is equal to $\left.\left(\epsilon_3 d_1d_2+\delta\sum_{j=4}^{n}d_j^2\right)\right/(d_1+d_2)$. 

If $\epsilon_3 = \delta$, $w=\delta e_3$ is contained in $C^+(g,h)^\circ$. 
Indeed, for any $d\in C^+(g,h)$, $d\cdot w$ is equal to $\delta d_3\leq 0$ if $d_1=d_2=0$, and $\left.\left(d_1d_2+\sum_{j=4}^{n}d_j^2\right)\right/(d_1+d_2)\leq 0$ otherwise. (Note that $d_1,d_2\leq 0$.)
If $\epsilon_3 = -\delta$, $w=e_1+e_2+\delta e_3$ is contained in $C^+(g,h)^\circ$. 
Indeed, for any $d\in C^+(g,h)$, $d\cdot w$ is equal to $\delta d_3\leq 0$ if $d_1=d_2=0$, and $\left.\left((d_1+d_2)^2-d_1d_2+\sum_{j=4}^{n}d_j^2\right)\right/(d_1+d_2)\leq 0$ otherwise. 
In each case, the given $w$ is not contained in $L^+(g,h)^\circ$. 
\end{proof}

\subsubsection*{The germ of type $(8)$ in Table~\ref{table:generic constraint q>0 r=1}}

Let $(g,h) = \left(x_1,x_2,x_3,\sum_{j=1}^{3}\delta_jx_j^2+\sum_{1\leq i<j\leq 3}\alpha_{ij}x_ix_j+\epsilon_1x_1x_2x_3 + \sum_{j=4}^{n} \epsilon_j x_j^2\right)$, where $\delta_j, \epsilon_j\in \{1,-1\}$ and $\alpha_{ij}\in \R$ satisfies the condition $(\ast\ast)$ in Table~\ref{table:generic constraint r=0}, $Q =\sum_{j=1}^{3}\delta_jx_j^2+\sum_{1\leq i<j\leq 3}\alpha_{ij}x_ix_j+ \sum_{j=4}^{n} \epsilon_j x_j^2$, and $R=\epsilon_1x_1x_2x_3$. 

\begin{proposition}

The tangent cone $C^+(g,h)$ is equal to $\overline{C}(3,Q)$.

\end{proposition}

\begin{proof}
The tangent cone $C^+(g,h)$ is contained in $\overline{C}(3,Q)$ by Lemma~\ref{lem:tangent cone for r=1}. 
Let $d\in \overline{C}(3,Q)$. 
Since $R_3 = R = \epsilon_1x_1x_2x_3$ and $R_j=0$ for $j\neq 3$, $d$ is contained in $\underline{C}(3,Q,R)\subset C^+(g,h)$ if $d_1d_2d_3=0$. 
In what follows, we assume $d_1,d_2,d_3 < 0$.
If $d_j$ is not $0$ for $j\geq 4$, the vector $v = \epsilon_1\epsilon_jd_je_j$ satisfies the conditions in Lemma~\ref{lem:tangent cone for r=1} for $d$.
Indeed, $R_3(d) = \epsilon_1 d_1d_2d_3$ and $v\cdot \nabla Q(d) = \epsilon_1\epsilon_j^2d_j^2$, in particular the signs of these values are mutually opposite. 
Suppose that $d_4,\ldots, d_n$ are all equal to $0$. 
We put $d' = {}^t(d_1,d_2,d_3)$ and $A = \begin{pmatrix}
2\delta_1&\alpha_{12}&\alpha_{13}\\
\alpha_{12}&2\delta_2&\alpha_{23}\\
\alpha_{13}&\alpha_{23}&2\delta_3
\end{pmatrix}$, which is a $3\times 3$-submatrix of $\mathrm{Hess}(Q)$. 
By the direct calculation, we can deduce $\nabla Q(d) = Ad'$, which is not $0$ since $A$ is regular. 
Suppose that the $\ell$-th component of $Ad'$, denoted by $\beta$, is not $0$. 
Since $d$ is contained in $\overline{C}(3,Q)$, $Q(d) = \frac{1}{2}{}^td' A d'$ is equal to $0$. 
For $0<|\varepsilon|\ll 1$, the vector $v=d+\varepsilon e_\ell$ satisfies the conditions in Lemma~\ref{lem:tangent cone for r=1}. 
Indeed, $v_1,v_2,v_3 < 0$ since $|\varepsilon|\ll 1$ and $d_1,d_2,d_3<0$, and $v\cdot \nabla Q(d) = {}^td' A d' + \varepsilon \beta = \varepsilon \beta$. 
We can make the sign of $\varepsilon \beta$ opposite to that of $R_3(d)$ by making an appropriate $\varepsilon$.  
\end{proof}

\begin{proposition}
GCQ holds for $(g,h)$ if and only if 
\begin{math}
\left\{ \epsilon_4, \ldots, \epsilon_n \right\} = \left\{ 1, -1 \right\}
\end{math} 
\end{proposition}
\begin{proof}
It is easy to check that 
\begin{math}
L^+ \left( g, h \right)^\circ
\end{math}
is equal to
\begin{equation}
\left\{ w \in \mathbb{R}^n \middle| w_1 \ge 0, w_2 \ge 0, w_3 \ge 0, w_4 = \cdots = w_n = 0 \right\}.
\end{equation}

\noindent
\textbf{Proof of ``if'' part:} Without loss of generality, we can assume 
\begin{math}
\epsilon_4 = 1
\end{math}
and 
\begin{math}
\epsilon_5 = -1
\end{math}. Take any 
\begin{math}
w \in C^+ \left( g, h \right)^\circ
\end{math}. For each 
\begin{math}
j \in \left\{ 4, \ldots, n \right\}
\end{math}, either 
\begin{math}
\pm e_j \pm e_4
\end{math}
or 
\begin{math}
\pm e_j \pm e_5
\end{math}
is contained in 
\begin{math}
C^+ \left( g, h \right)
\end{math}. This proves 
\begin{math}
w_j = 0
\end{math}
for all 
\begin{math}
j \in \left\{ 4, \ldots, n \right\}
\end{math}. Since either
\begin{math}
-e_j \pm e_4 \in C^+ \left( g, h \right)
\end{math}
or 
\begin{math}
-e_j \pm e_5 \in C^+ \left( g, h \right)
\end{math}
holds for 
\begin{math}
j \in \left\{ 1, 2, 3 \right\}
\end{math}, 
\begin{math}
w_1 \ge 0
\end{math}, 
\begin{math}
w_2 \ge 0
\end{math}, and 
\begin{math}
w_3 \ge 0
\end{math}
hold. This proves 
\begin{math}
C^+ \left( g, h \right)^\circ \subset L^+ \left( g, h \right)^\circ
\end{math}
and thus GCQ holds in this case. 

\noindent
\textbf{Proof of ``only if'' part:} 
We assume $\epsilon_4=\cdots = \epsilon_n = \delta$ for some $\delta\in \{1,-1\}$.
%
Let $\lambda_1,\lambda_2,\lambda_3	
\neq 0$ be the eigenvalues of $\begin{pmatrix}
\delta_1&\alpha_{12}/2&\alpha_{13}/2\\
\alpha_{12}/2&\delta_2&\alpha_{23}/2\\
\alpha_{13}/2&\alpha_{23}/2&\delta_3\\
\end{pmatrix}$.
The following then holds for $d\in C^+(g,h)$ (cf.~the proof of Proposition~\ref{prop:GCQ type (4) q,r>0}): 
{\allowdisplaybreaks
\begin{align*}
& \sum_{j=1}^{3}\delta_jd_j^2+\sum_{1\leq i<j\leq 3}\alpha_{ij}d_id_j = -\delta\sum_{j=4}^{n} d_j^2\\
\Rightarrow& \max\{|\lambda_1|,|\lambda_2|,|\lambda_3|\} \left\| \left(d_1,d_2,d_3\right)\right\|^2 \geq \left|\sum_{j=4}^{n} d_j^2\right|. 
\end{align*}
}%
Let $R =\sqrt{\max\{|\lambda_1|,|\lambda_2|,|\lambda_3|\}}$. 
For any $d\in C^+(g,h)$, the inner product $(R(e_1+e_2+e_3)+e_4)\cdot d$ is estimated as follows:
{\allowdisplaybreaks
\begin{align*}
& (R(e_1+e_2+e_3)+e_4)\cdot d\\
=& R(d_1+d_2+d_3) + d_4 \\
\leq & -R(|d_1|+|d_2|+|d_3|) +|d_4| & (\because d_1,d_2,d_3\leq 0)\\
\leq & -R\left\|\left(d_1,d_2,d_3\right)\right\| + |d_4| & \left(\because \left\|\left(d_1,d_2.d_3\right)\right\| = \sqrt{d_1^2+d_2^2+d_3^2} \leq |d_1|+|d_2|+|d_3|\right)\\
\leq & -R \sqrt{\dfrac{\left|\sum_{j=4}^{n} d_j^2\right|}{\max\{|\lambda_1|,|\lambda_2|,|\lambda_3\}}}+|d_4|\leq 0.
\end{align*}
}%
Thus, $R(e_1+e_2+e_3)+e_4$ is contained in $C^+(g,h)^\circ$. 
However, it is not in $L^+(g,h)^\circ$, and thus GCQ is violated. 
\end{proof}

The results of this section are summarized in Theorem~\ref{thm:generic_CQ_classification}.

\begin{theorem}[Generic CQ Classification]\label{thm:generic_CQ_classification}

Only the fully regular class (with $n = q + r$ and constraints locally equivalent to $(g,h)$ with $g(x)=(x_1, \dots, x_q)$ and $h(x)=(x_{q+1}, \dots, x_{q+r})$) satisfies LICQ. 
All the other constraint classes (especially those in Tables~\ref{table:generic constraint q=0}--\ref{table:generic constraint q>0 r=1}) violate LICQ.

MFCQ fails whenever a singular equality constraint is present, in particular all the classes in Tables~\ref{table:generic constraint q=0} and \ref{table:generic constraint q>0 r=1} violate MFCQ.
In classes consisting of only inequalities (Table~\ref{table:generic constraint r=0}), MFCQ holds if and only if the parameter $l_1$ in the normal form in the caption of Table~\ref{table:generic constraint r=0} is positive.

The satisfaction or failure of ACQ and GCQ for the classes in Tables~\ref{table:generic constraint q=0}--\ref{table:generic constraint q>0 r=1} are completely determined as summarized in Tables~\ref{tab:eq-only}--\ref{tab:mixed}. 
(As explained in the captions, all the classes in Tables~\ref{table:generic constraint q=0} and \ref{table:generic constraint q>0 r=1} violate ACQ.)

%
%
%
%

\end{theorem}


\newcolumntype{C}[1]{>{\centering\arraybackslash}p{#1}} 
\newcolumntype{L}[1]{>{\raggedright\arraybackslash}p{#1}}

%
%

\begin{table}[htbp]
\centering
\begin{tabular}{|c|l|}
\hline
\textbf{Type} & \textbf{Conditions for GCQ} \\ \hline
$\left( 1, k \right)$ &
\begin{minipage}[c]{64mm}
\V{.2em}
($k=2$ and one of $\epsilon_2, \ldots, \epsilon_n$ is $-1$) 

\textbf{or}
($k\ge 3$ and $\{\epsilon_2,\dots,\epsilon_n\} = \{ 1,-1 \}$) 
\V{.2em}
\end{minipage}\\ \hline

$\left( 2 \right)$ &
one of $\epsilon_4, \ldots, \epsilon_n$ is $-1$ \\ \hline
\end{tabular}
\caption{Conditions for the classes in Table~\ref{table:generic constraint q=0} to satisfy GCQ. 
Note that ACQ fail for all the classes in Table~\ref{table:generic constraint q=0}.
}
\label{tab:eq-only}
\end{table}

\begin{table}[htbp]
\centering
\renewcommand{\arraystretch}{1.2}
\begin{tabular}{|c|l|l|}
\hline
\textbf{Type} &\textbf{Conditions for ACQ}& \textbf{Conditions for GCQ} \\ \hline
$\left( 1 ,k \right)$ &
\begin{minipage}[c]{55mm}
\V{.2em}
$\left(\epsilon_q=\cdots =\epsilon_n=-1\right)$

$\lor \left(k\ge3\land \epsilon_{q+1}=\cdots =\epsilon_n=-1\right)$\V{.2em}
\end{minipage}&
\begin{minipage}[c]{60mm}
\V{.2em}
$\left(k=2\land (\epsilon_q, \ldots, \epsilon_n)\neq(1,\ldots, 1)\right)$

$\lor \left((\epsilon_{q+1}, \ldots, \epsilon_n)\neq (1,\ldots, 1)\right)$\V{.2em}
\end{minipage}\\ \hline
$\left( 2 \right)$ &$\epsilon_{q+2}=\cdots =\epsilon_n=-1$&
$(\epsilon_{q+2}, \ldots, \epsilon_n)\neq (1,\ldots, 1)$ \\ \hline

$\left( 3, k \right)$ &
\begin{minipage}[c]{55mm}
\V{.2em}
$\left(\epsilon_{q-1}=\cdots =\epsilon_n=-1\right)$

$\lor \left(k\ge3\land \epsilon_{q}=\cdots =\epsilon_n=-1\right)$\V{.2em}
\end{minipage}
& 
$(\epsilon_q, \ldots, \epsilon_n)\neq (1,\ldots,1)$\\ \hline

$\left( 4, k \right)$ &$\bot$ (i.e.,~ACQ does not hold.)&
$(\epsilon_{q+1}, \ldots, \epsilon_n)\neq (1,\ldots, 1) \lor \epsilon_q = (-1)^{k+1}$\\ \hline

$\left( 5 \right)$ & $\epsilon_{q-1}=\epsilon_{q+1}=\cdots =\epsilon_n =-1$&
$(\epsilon_{q+1}, \ldots, \epsilon_n)\neq (1,\ldots, 1)$\\ \hline 

$\left( 6 \right)$ &
\begin{minipage}[c]{50mm}
\V{.2em}
$\delta_1=\delta_2=-1\land \alpha <2$

$\land \epsilon_{q}=\cdots =\epsilon_n =-1$\V{.2em}
\end{minipage}
&
$(\epsilon_q, \ldots, \epsilon_n)\neq (1,\ldots, 1)$\\ \hline 

$\left( 7 \right)$ &$\epsilon_{q-2}=\epsilon_q=\cdots =\epsilon_n=-1$&
$(\epsilon_q, \ldots, \epsilon_n)\neq (1,\ldots, 1)$ \\ \hline 

$\left( 8 \right)$ &$\epsilon_{q-1}=\epsilon_{q-1}'=\epsilon_q=\cdots =\epsilon_n=-1$&
$(\epsilon_q, \ldots, \epsilon_n)\neq (1,\ldots ,1)$\\ \hline 

$\left( 9 \right)$ &$\bot$ (i.e.,~ACQ does not hold.)&
$(\epsilon_{q+1}, \ldots, \epsilon_n)\neq (1,\ldots, 1)\lor \epsilon_{01}=1 \lor \epsilon_{02}=1$ \\ \hline 

$\left( 10 \right)$ & $(\dagger)$ &
$(\epsilon_q, \ldots, \epsilon_n)\neq (1,\ldots ,1)$\\ \hline
\end{tabular}
\captionsetup{singlelinecheck=off}
\caption{
Conditions for the classes in Table~\ref{table:generic constraint r=0}\emph{ with $l_1=0$} to satisfy ACQ or GCQ. 
Note that MFCQ holds if and only if $l_1 > 0$, and in this case ACQ and GCQ also hold.
The condition $(\dagger)$ (for type (10)) is the following condition: 
\begin{math}
\delta_1 = \delta_2 = \delta_3 = -1
\end{math}, 
\begin{math}
\epsilon_q = \cdots = \epsilon_n = -1
\end{math}
and there exist distinct indices $i,j,k \in \{1,2,3\}$ with
\begin{center}
$\left(\alpha_{ij}\le 0\land\alpha_{ik}\le 0\land\alpha_{jk}<2\right)\lor \left(0<\alpha_{ij}<2\land 0<\alpha_{ik}<2\land\alpha_{jk}+\frac{\alpha_{ij}\alpha_{ik}}{2}< 2\sqrt{\Bigl(1-\frac{\alpha_{ij}^2}{4}\Bigr)\Bigl(1-\frac{\alpha_{ik}^2}{4}\Bigr)}\right)$, 
\end{center}
where we regard that $\alpha_{ij}$ is attached to the unordered pair $\{i,j\}$ (that is, we assume $\alpha_{ji}=\alpha_{ij}$).
}
\label{tab:ineq-only}
\end{table}

\begin{table}[htbp]
\centering
\begin{tabular}{|c|l|}
\hline
\textbf{Type} & \textbf{Conditions on which GCQ holds} \\ \hline
$\left( 1, k \right)$ & 
$\{\epsilon_2,\dots,\epsilon_n\} = \{ 1, -1 \}$ \\

$\left( 2 \right)$ &
$\{\epsilon_3,\dots,\epsilon_n\} = \{ 1, -1 \}$ \\

$\left( 3, k \right)$ &
($k$ is even and one of $\epsilon_3, \ldots, \epsilon_n$ is $-1$) \textbf{or} \\
& ($k$ is odd and ($\epsilon_1 = 1$ or $\left\{ \epsilon_3, \ldots, \epsilon_n \right\} = \left\{ 1, -1 \right\}$)) \\

$\left( 4 \right)$ &
$\left\{ \epsilon_3, \ldots, \epsilon_n \right\} = \left\{ 1, -1 \right\}$ \\

$\left( 5 \right)$ &
$\left\{ \epsilon_3, \ldots, \epsilon_n \right\} = \left\{ 1, -1 \right\}$ \\

$\left( 6 \right)$ &
$\left\{ \epsilon_3, \ldots, \epsilon_n \right\} = \left\{ 1, -1 \right\}$ \\

$\left( 7 \right)$ &
$\left\{ \epsilon_4, \ldots, \epsilon_n \right\} = \left\{ 1, -1 \right\}$ \\

$\left( 8 \right)$ &
$\left\{ \epsilon_4, \ldots, \epsilon_n \right\} = \left\{ 1, -1 \right\}$ \\ \hline
\end{tabular}
\caption{Conditions for the classes in Table~\ref{table:generic constraint q>0 r=1} to satisfy GCQ.
Note that ACQ fails for all the classes in the table.}
\label{tab:mixed}
\end{table}

\section*{Acknowledgements}
K. H. and H. T. were supported by JSPS Grant-in-Aid for Scientific Research (C) (No. 23K03123). 
H. T. was supported by JSPS Grant-in-Aid for Scientific Research (B) (No. 23K26608), and Scientific Research (A) (No. 24H00685), Transformative Research Areas (B) (No. 23H03797). 

\appendix

\section{Genericity of counter examples to show the strictness of the hierarchy of the four classical constraint qualifications} \label{sec:appendix}
To demonstrate the strictness of the hierarchy of the four classical constraint qualifications in Eq.~\eqref{eq:strict_hierarchy}, several counterexamples have been constructed in the literature \cite{Peterson1973,Wright2001EffectsOF,andreani_silva_2014,doi:10.1287/moor.2017.0879}. Here, we examine these examples through the lens of $\mathcal{K}[G]$-equivalence and determine whether or not they are generic. As a measure of genericity, we compute the $\mathcal{K}[G]_e$-codimension for each case. (For the computational details, see the Appendix of \cite{HamadaHayanoTeramoto2025_1}.)

\begin{example}[Peterson \cite{Peterson1973} (MFCQ but not LICQ)]
Let
\begin{math}
n = 1
\end{math}, 
\begin{math}
q = 1
\end{math}, and 
\begin{math}
r = 0
\end{math}. Let 
\begin{math}
g \left( x \right) = \left( x, 2x \right)
\end{math}. Then, LICQ is violated at the origin since both of the inequality constraints are active there and their gradients are linearly dependent. Contrastingly, MFCQ holds because one can take 
\begin{math}
d = \left( -1 \right)
\end{math}
as an MF-vector. In this case, 
\begin{math}
\displaystyle T\mathcal{K} \left[ G \right]_e \left( g \right) = \langle 
\begin{pmatrix}
1 \\
2 
\end{pmatrix}, 
\begin{pmatrix}
x \\
0
\end{pmatrix}, 
\begin{pmatrix}
0 \\
2x
\end{pmatrix}
\rangle_{\mathcal{E}_1}
\end{math}
holds and thus 
\begin{math}
\displaystyle \frac{\mathcal{E}_1^2}{T\mathcal{K} \left[ G \right]_e \left( g \right)} \cong \langle
\begin{pmatrix}
0 \\
1
\end{pmatrix} \rangle_{\mathbb{R}}
\end{math}
holds. Therefore, this constraint has 
\begin{math}
\mathcal{K} \left[ G \right]_e
\end{math}-codimension $1$. 
\end{example}
\begin{example}[Wright \cite{Wright2001EffectsOF} (MFCQ but not LICQ)]
Let
\begin{math}
n = 2
\end{math}, 
\begin{math}
q = 2
\end{math}, and 
\begin{math}
r = 0
\end{math}. Let 
\begin{equation}
g \left( x \right) = \left( \left( x_1 - \frac{1}{3} \right)^2 + x_2^2 - \frac{1}{9}, \left( x_1 - \frac{2}{3} \right)^2 + x_2^2 - \frac{4}{9} \right).
\end{equation}
Then, LICQ is violated at the origin since both of the inequality constraints are active there but their gradients
\begin{math}
\left( - 2/3, 0 \right)
\end{math}
and 
\begin{math}
\left( - 4/3, 0 \right)
\end{math} are not linearly independent. However, there exists an MF-vector 
\begin{math}
d = \left( 1, 0 \right)
\end{math}
and thus MFCQ holds at the origin. In this case, note that 
\begin{math}
g \left( x \right) = \left( - \frac{2}{3} x_1 + x_1^2 + x_2^2, -\frac{4}{3} x_1 + x_1^2 + x_2^2 \right)
\end{math}
holds. By the coordinate transformation 
\begin{math}
\phi \colon \left( x_1, x_2 \right) \mapsto \left( X_1 = - \frac{2}{3} x_1 + x_1^2 + x_2^2, X_2 = x_2 \right)
\end{math}, we obtain
\begin{equation}
g \circ \phi^{-1} \left( X_1, X_2 \right) = \left( X_1, X_1 - \frac{2}{9} \left( 1 - \sqrt{1 + 9 X_1 - 9 X_2^2} \right) \right).
\end{equation}
and its $2$-jet at the origin is 
\begin{math}
j^2 \left( g \circ \phi^{-1} \right) \left( X \right) = \left( X_1, 2 X_1 - X_2^2 + \frac{9}{4} X_1^2 \right)
\end{math} This is 
\begin{math}
\mathcal{K} \left[ G \right]^2
\end{math}-equivalent to 
\begin{math}
\left( X_1, X_1 - X_2^2 \right)
\end{math}, which coincides with the germ of type 
\begin{math}
\left( 1, 2 \right)
\end{math}
in Table~\ref{table:generic constraint r=0} (with
\begin{math}
\epsilon_2 = -1
\end{math}). Since this normal form is $2$-$\mathcal{K} \left[ G \right]$-determined, 
\begin{math}
g
\end{math}
itself is 
\begin{math}
\mathcal{K} \left[ G \right]
\end{math}-equivalent to the normal form 
\begin{math}
\left( 1, 2 \right)
\end{math}(with $\epsilon_2 = -1$) in Table~\ref{table:generic constraint r=0}. Therefore, $g$ has
\begin{math}
\mathcal{K} \left[ G \right]_e
\end{math}-codimension 
\begin{math}
1
\end{math}. 
\end{example}
\begin{example}[Peterson \cite{Peterson1973} (ACQ but not MFCQ)]
Let
\begin{math}
n = 2
\end{math}, 
\begin{math}
q = 2
\end{math}, and 
\begin{math}
r = 1
\end{math}. Let 
\begin{math}
g \left( x \right) = \left( x_2 - x_1^2, -x_2 + x_1^2 \right)
\end{math}. This is 
\begin{math}
\mathcal{K} \left[ G \right]
\end{math}-equivalent to 
\begin{math}
g' \left( x \right) = \left( x_1, -x_1 \right)
\end{math}. Since MFCQ and ACQ are invariant under the action of
\begin{math}
\mathcal{K} \left[ G \right]
\end{math}, it is enough to consider 
\begin{math}
g'
\end{math}. Then, MFCQ is violated because there is no MF-vector 
\begin{math}
d \in \mathbb{R}^2
\end{math}
such that 
\begin{math}
dg'_{1,0} \left( d \right) = d_1 < 0
\end{math}
and 
\begin{math}
dg'_{2,0} \left( d \right) = -d_1 < 0
\end{math}
hold. ACQ holds because
\begin{math}
g'
\end{math}
is linear in 
\begin{math}
x_1, x_2
\end{math}
and thus the linearized cone coincides with the tangent cone. In this case, 
\begin{math}
T \mathcal{K} \left[ G \right]_e \left( g, h \right) = \langle 
\begin{pmatrix}
1 \\
-1
\end{pmatrix}, 
\begin{pmatrix}
x_1 \\
0
\end{pmatrix}, 
\begin{pmatrix}
0 \\
x_1
\end{pmatrix}
\rangle_{\mathcal{E}_2}
\end{math}
holds and thus 
\begin{math}
\displaystyle \frac{\mathcal{E}_1^2}{T\mathcal{K} \left[ G \right]_e \left( g, h \right)} \supset \left\langle
\begin{pmatrix}
0 \\
x_2^j
\end{pmatrix}
\middle| j \in \mathbb{N}
\right\rangle_{\mathbb{R}}
\end{math}
holds. This implies the 
\begin{math}
\mathcal{K} \left[ G \right]_e
\end{math}-codimension of 
\begin{math}
g
\end{math}
is infinite. 
\end{example}
\begin{example}[Andreani et al. \cite{doi:10.1287/moor.2017.0879} (ACQ but not MFCQ)]
Let
\begin{math}
n = 2
\end{math}, 
\begin{math}
q = 2
\end{math}, and 
\begin{math}
r = 0
\end{math}. Let 
\begin{math}
g \left( x \right) = \left( -x_1, -x_1^2-x_2^2 \right)
\end{math}. This is 
\begin{math}
\mathcal{K} \left[ G \right]
\end{math}-equivalent to the germ of type
\begin{math}
\left( 3, 2 \right)
\end{math}
in Table~\ref{table:generic constraint r=0}, 
\begin{math}
\epsilon_1 = \epsilon_2 = -1
\end{math}
of 
\begin{math}
\mathcal{K} \left[ G \right]_e
\end{math}-codimension $2$. In this case, 
\begin{math}
l_1 = 0
\end{math}
and thus MFCQ is violated. In addition, the result in Table~\ref{tab:ineq-only} implies that ACQ holds. 
\end{example}
\begin{example}[Andreani and Silva \cite{andreani_silva_2014} (GCQ but not ACQ)]
Let
\begin{math}
n = 2
\end{math}, 
\begin{math}
q = 2
\end{math}, and 
\begin{math}
r = 1
\end{math}. Let 
\begin{math}
g \left( x \right) = \left( x_1, x_2 \right).
\end{math}
and 
\begin{math}
h \left( x \right) = x_1 x_2
\end{math}. Then, the feasible set-germ at the origin is 
\begin{math}
M \left( g, h \right) = \left\{ \left( x_1, x_2 \right) \in \left( \mathbb{R}^2, 0 \right) \middle| x_1 = 0, x_2 \le 0 \; \text{or} \; x_1 \le 0, x_2 = 0 \right\}
\end{math}. In this case, it is easy to check that 
\begin{equation}
L^+ \left( g, h \right) = \left\{ \left( x_1, x_2 \right) \in \mathbb{R}^2 \middle| x_1 \le 0, x_2 \le 0 \right\}
\end{equation}
and 
\begin{equation}
C^+ \left( g, h \right) = \left\{ \left( x_1, x_2 \right) \in \mathbb{R}^2 \middle| x_1 \le 0, x_2 = 0 \; \text{or} \; x_1 = 0, x_2 \le 0 \right\}
\end{equation}
hold. Therefore, ACQ does not hold. However, 
\begin{math}
L^+ \left( g, h \right)^\circ = C^+ \left( g, h \right)^\circ
\end{math}
holds and thus GCQ holds. In this case, 
\begin{math}
T \mathcal{K} \left[ G \right]_e \left( g, h \right) = \langle 
\begin{pmatrix}
1 \\
0 \\
x_2
\end{pmatrix}, 
\begin{pmatrix}
0 \\
1 \\
x_1
\end{pmatrix}, 
\begin{pmatrix}
x_1 \\
0 \\
0
\end{pmatrix}, 
\begin{pmatrix}
0 \\
x_2 \\
0
\end{pmatrix}, 
\begin{pmatrix}
0 \\
0 \\
x_1 x_2
\end{pmatrix}
\rangle_{\mathcal{E}_2}
\end{math}
holds and thus 
\begin{math}
\displaystyle \frac{\mathcal{E}_1^2}{T\mathcal{K} \left[ G \right]_e \left( g, h \right)} \supset \left\langle
\begin{pmatrix}
0 \\
0 \\
x_1^{j_1}
\end{pmatrix}, 
\begin{pmatrix}
0 \\
0 \\
x_2^{j_2}
\end{pmatrix} 
\middle| j_1, j_2 \in \mathbb{N}
\right\rangle_{\mathbb{R}}
\end{math}
holds. This implies the 
\begin{math}
\mathcal{K} \left[ G \right]_e
\end{math}-codimension of 
\begin{math}
\left( g, h \right)
\end{math}
is infinite. 
\end{example}
\begin{example}[Peterson \cite{Peterson1973} (GCQ but not ACQ)]
Let
\begin{math}
n = 2
\end{math}, 
\begin{math}
q = 1
\end{math}, and 
\begin{math}
r = 0
\end{math}. Let 
\begin{math}
g \left( x \right) = x_1^2 x_2^2
\end{math}. Then, the feasible set-germ at the origin is 
\begin{math}
M \left( g \right) = \left\{ \left( x_1, x_2 \right) \in \left( \mathbb{R}^2, 0 \right) \middle| x_1 = 0 \; \text{or} \; x_2 = 0 \right\}
\end{math}. In this case, it is easy to check that 
\begin{math}
L^+ \left( g \right) = \mathbb{R}^2
\end{math}
and 
\begin{math}
C^+ \left( g \right) = \left\{ \left( x_1, x_2 \right) \in \left( \mathbb{R}^2, 0 \right) \middle| x_1 = 0 \; \text{or} \; x_2 = 0 \right\}
\end{math}
hold. Therefore, ACQ does not hold. However, 
\begin{math}
L^+ \left( g, h \right)^\circ = C^+ \left( g, h \right)^\circ
\end{math}
holds and thus GCQ holds. In this case, 
\begin{math}
T \mathcal{K} \left[ G \right]_e \left( g \right) = \langle 2 x_1 x_2^2, 2 x_1^2 x_2, x_1^2 x_2^2 \rangle_{\mathcal{E}_2} = \langle x_1 x_2^2, x_1^2 x_2 \rangle_{\mathcal{E}_2}
\end{math}
and thus
\begin{math}
\displaystyle \frac{\mathcal{E}_1^2}{T\mathcal{K} \left[ G \right]_e \left( g \right)} \supset \left\langle x_1^{j_1}, x_2^{j_2}  \middle| j_1, j_2 \in \mathbb{N} \right\rangle_{\mathbb{R}}
\end{math}
holds. This implies the 
\begin{math}
\mathcal{K} \left[ G \right]_e
\end{math}-codimension of 
\begin{math}
g
\end{math}
is infinite. 
\end{example}

\bibliography{dyn_recomb_initials_v3}
\bibliographystyle{plain}

\end{document}